\documentclass[a4paper,reqno,11pt]{amsart}

\usepackage{amsmath, amsfonts, amssymb, amsthm, amscd}
\usepackage{graphicx}
\usepackage{psfrag}
\usepackage{perpage}
\usepackage{url}
\usepackage{subfig}
\usepackage{makerobust}
\usepackage{color}

\usepackage[utf8]{inputenc}
\usepackage[T1]{fontenc}

\usepackage[a4paper,scale={0.72,0.74},marginratio={1:1},footskip=7mm,headsep=10mm]{geometry}

\usepackage{hyperref}

\numberwithin{equation}{section}

\newtheorem{theorem}{Theorem}[section]
\newtheorem{lemma}[theorem]{Lemma}
\newtheorem{proposition}[theorem]{Proposition}
\newtheorem{corollary}[theorem]{Corollary}
\newtheorem{remark}[theorem]{Remark}
\newtheorem{definition}[theorem]{Definition}

\newtheorem{hypothesis}[theorem]{Hypothesis}

\newcommand{\cA}{{\ensuremath{\mathcal A}} }

\newcommand{\cG}{{\ensuremath{\mathcal G}} }

\newcommand{\cT}{{\ensuremath{\mathcal T}} }

\renewcommand{\tilde}{\widetilde}          
\DeclareMathSymbol{\leqslant}{\mathalpha}{AMSa}{"36} 
\DeclareMathSymbol{\geqslant}{\mathalpha}{AMSa}{"3E} 
\DeclareMathSymbol{\eset}{\mathalpha}{AMSb}{"3F}     
\newcommand{\dd}{\text{\rm d}}             

\newcommand{\R}{\mathbb{R}}

\newcommand{\Z}{\mathbb{Z}}
\newcommand{\N}{\mathbb{N}}

\newcommand{\PEfont}{\mathrm}

\renewcommand{\P}{\ensuremath{\PEfont P}}

\newcommand{\E}{\ensuremath{\PEfont E}}

\DeclareMathOperator{\sign}{sign}

\newcommand{\ind}{{\sf 1}}

\renewcommand{\epsilon}{\varepsilon} 
\renewcommand{\theta}{\vartheta} 
\renewcommand{\rho}{\varrho} 
\renewcommand{\phi}{\varphi}

\newcommand\sfc{\mathsf{c}}

\newcommand\sff{\mathsf{f}}
\newcommand\sfC{\mathsf{C}}

\newenvironment{myenumerate}{%
\renewcommand{\theenumi}{\arabic{enumi}}%
\renewcommand{\labelenumi}{{\rm(\theenumi)}}%
\begin{list}{\labelenumi}
	{%
	\setlength{\itemsep}{0.4em}%
	\setlength{\topsep}{0.5em}%
	\setlength\leftmargin{2.45em}%
	\setlength\labelwidth{2.05em}%
	\setlength{\labelsep}{0.4em}%
	\usecounter{enumi}%
	}%
	}%
{\end{list}
}

\newenvironment{ienumerate}{%
\renewcommand{\theenumi}{\roman{enumi}}%
\renewcommand{\labelenumi}{{\rm(\theenumi)}}%
\begin{list}{\labelenumi}
	{%
	\setlength{\itemsep}{0.4em}%
	\setlength{\topsep}{0.5em}%
	\setlength\leftmargin{2.45em}%
	\setlength\labelwidth{2.05em}%
	\setlength{\labelsep}{0.4em}%
	\usecounter{enumi}%
	}%
	}%
{\end{list}
}

\newenvironment{aenumerate}{%
\renewcommand{\theenumi}{\alph{enumi}}%
\renewcommand{\labelenumi}{{\rm(\theenumi)}}%
\begin{list}{\labelenumi}
	{%
	\setlength{\itemsep}{0.4em}%
	\setlength{\topsep}{0.5em}%
	\setlength\leftmargin{2.45em}%
	\setlength\labelwidth{2.05em}%
	\setlength{\labelsep}{0.4em}%
	\usecounter{enumi}%
	}%
	}%
{\end{list}
}

{\end{myenumerate}}

\newenvironment{myitemize}{%
\begin{list}{$\bullet$}%
 	{%
	\setlength{\itemsep}{0.4em}%
	\setlength{\topsep}{0.5em}%
	\setlength\leftmargin{2.45em}%
	\setlength\labelwidth{2.05em}%
	\setlength{\labelsep}{0.4em}%
	}%
	}%
{\end{list}}

\renewenvironment{itemize}{
\begin{myitemize}}%
{\end{myitemize}}

\MakePerPage[2]{footnote} 

\title{}

\date{\today}

\def\dd{\mathrm{d}}

\newcommand\CBS{\mathsf{C_{BS}}}

\newcommand\imp{\mathrm{imp}}
\newcommand\bkappa{\boldsymbol{\kappa}}

\title[Volatility smile in a multiscaling model]{The asymptotic smile\\
of a multiscaling stochastic volatility model}

\author{Francesco Caravenna}
\address{Dipartimento di Matematica e Applicazioni, 
Universit\`a degli Studi di Milano-Bicocca,
via Cozzi 55,
I-20125 Milano, Italy}

\email{francesco.caravenna@unimib.it}

\author{Jacopo Corbetta}
\address{\'{E}cole des Ponts - ParisTech,
CERMICS,
6 et 8 avenue Blaise Pascal,
77420 Champs sur Marne, 
France}
\email{jacopo.corbetta@enpc.fr}



\begin{document}

\MakeRobustCommand\subref

\begin{abstract}
We consider a stochastic volatility model which captures relevant stylized facts of financial series,
including the multi-scaling of moments.
The volatility evolves according to
a generalized Ornstein-Uhlenbeck processes
with super-linear mean reversion.

Using large deviations techniques, we determine the asymptotic shape of
the implied volatility surface in any regime of small maturity $t \to 0$
or extreme log-strike $|\kappa| \to \infty$ (with bounded maturity).
Even if the price has continuous paths,
out-of-the-money implied volatility diverges for small maturity,
producing a very pronounced smile.
\end{abstract}

\maketitle

\section{Introduction}

The evolution of the price $(S_t)_{t\ge 0}$ of an asset 
is often described by a stochastic
volatility model $\dd S_t = S_t (\mu \, \dd t + \sigma_t \, \dd B_t)$, where 
$(B_t)_{t\ge 0}$ is a standard Brownian motion
and $(\sigma_t)_{t\ge 0}$ is a stochastic process.
A popular choice for $(\sigma_t)_{t\ge 0}$ is a process
of Ornstein-Uhlenbeck type:
\begin{equation}\label{eq:OU}
	\dd \sigma^2_t = - c \, (\sigma^2_t)^\gamma \, \dd t + \dd L_t \,,
\end{equation}
where $(L_t)_{t\ge 0}$ is a subordinator (i.e.\ a non-decreasing L\'evy process)
and $c, \gamma \in (0,\infty)$ are parameters, the usual choice
being the case $\gamma = 1$ when the mean reversion is linear,
cf.~\cite{cf:BS}.
This class of models is rich enough to
reproduce many empirically observed stylized facts, 
including heavy tails in the distribution of $S_t$
and clustering of volatility.

Another remarkable stylized fact is the so-called \emph{multi-scaling of moments}
\cite{cf:D,cf:DAD,cf:GBPTD}. This refers to the fact that $\E[|S_{t+h}-S_t|^q] \approx h^{A(q)}$
as $h \to 0$, where the scaling exponent is diffusive
only up to a finite threshold, i.e.\ $A(q) = q/2$ for $q < q^*$, while for $q > q^*$ an anomalous
scaling $A(q) < q/2$ is observed. Interestingly, it was recently proved
in \cite{cf:DP} that a stochastic volatility model with $\sigma_t$ as in
\eqref{eq:OU} does not exhibit multi-scaling of moments 
in the linear case $\gamma=1$; however, \emph{multi-scaling of moments does occur
in the super-linear case $\gamma > 1$}, if
the L\'evy measure of $(L_t)_{t\ge 0}$
has a polynomial tail at infinity.

\smallskip

It is natural to ask how stochastic volatility models with $\sigma_t$ as in
\eqref{eq:OU} behave with respect to pricing, when $\gamma > 1$.
This is a non-trivial problem, because
the moment generating function of $S_t$ typically
admits no closed form outside the linear case $\gamma = 1$. 
However, there is a special limiting case 
which is analytically more tractable, defined as follows.

Consider a subordinator with
finite activity: 
$L_t = \sum_{k=1}^{N_t} J_k$, where $(N_t)_{t\ge 0}$ is a Poisson process 
and $(J_k)_{k\in\N}$ 
are i.i.d.\ non-negative random variables. 
In this case equation \eqref{eq:OU} can be solved pathwise, i.e.\
for any fixed realization of $(L_t)_{t\ge 0}$,
because between jump times of the Poisson process $(N_t)_{t\ge 0}$
it reduces to the ordinary differential equation 
\begin{equation}\label{eq:ode}
	\dd (\sigma^2_t) = -c \, (\sigma^2_t)^\gamma \, \dd t \,,
\end{equation}
which admits
explicit solutions.
The point is that, when $\gamma > 1$, one can \emph{let the jump size
diverge $J_k \to \infty$} and $(\sigma_t)_{t\ge 0}$ converges to
a well-defined limiting process, which explodes at the
jump times of the Poisson process and solves \eqref{eq:ode} between them
(see Figure~\ref{ch3:fig:IT}{\sc \subref{figura2}}).
For $\gamma > 2$, this limiting process $(\sigma_t)_{t\ge 0}$ has square-integrable paths
and can therefore be used to define a stochastic volatility model.

\smallskip

In this paper we focus on 
this stochastic volatility model, which was introduced in \cite{cf:ACDP}
(in a more direct way, see Section~\ref{ch3:sec:model})
and was shown to display several interesting features,
including multi-scaling of moments, clustering of volatility
and the crossover in the log-return distribution from power-law (small time) to Gaussian (large time).
We are interested in the price of European option and in the corresponding
implied volatility.

We stress that, 
besides its own interest, our model retains a close link with the general
class of stochastic volatility models
$\dd S_t = S_t (\mu \, \dd t + \sigma_t \, \dd B_t)$ 
with $\sigma_t$ as in \eqref{eq:OU} with $\gamma > 2$. 
For instance, 
option price and implied volatility of our model \emph{provide an upper bound
for all models in this class with a finite activity subordinator $(L_t)_{t\ge 0}$}
(see \S\ref{sec:OU}).

\smallskip

Our main results are sharp estimates for the tail decay of the log-return distribution
(Theorem~\ref{ch3:th:tailprob}), which yield
explicit asymptotic formulas for the price
of European options (Theorem~\ref{ch3:lemma:calltgamma}) and for the corresponding
\emph{implied volatility surface} (Theorem~\ref{ch3:th:main}).
Let us summarize some of the highlights,
referring to \S\ref{ch3:sec:discussion} for a more detailed discussion.
\begin{itemize}
\item We allow for any regime of either extreme log-strike $|\kappa| \to \infty$
(with arbitrary bounded maturity $t$, possibly varying with $\kappa$) or
small maturity $t \downarrow 0$ (with arbitrary log-strike $\kappa$, possibly varying with $t$).
This flexibility yields uniform estimates for
the implied volatility surface $\sigma_\imp(\kappa,t)$
in \emph{open regions} of the plane $(\kappa,t)$, cf.\ Corollary~\ref{ch3:th:surface}.

\item We show that
out-of-the-money implied volatility \emph{diverges} for small maturity, i.e.\
$\sigma_\imp(\kappa,t) \to \infty$ as
$t \downarrow 0$ for any $\kappa \ne 0$, while $\sigma_\imp(0,t) \to \sigma_0 < \infty$
(see Figure~\ref{ch3:fig:smile}).
This shows that stochastic volatility models
without jumps in the price can produce very steep skews for the small-time volatility smile,
cf.\ \cite[Chapter 5, ``Why jumps are needed'']{cf:Gat}.
What lies behind this phenomenon is the asymptotic emergence of heavy tails
in the small-time distribution of the volatility.
Interestingly, the same mechanism is responsible for the multi-scaling of moments.

\item We obtain the asymptotic expression
$\sigma_\imp(\kappa,t) \sim f(\kappa/t)$, for an explicit function
$f(\cdot)$ of \emph{just the ratio $(\kappa/t)$}, in a variety of interesting regimes
(including $t \downarrow 0$ for fixed $\kappa \ne 0$,
and $|\kappa| \to \infty$ for fixed $t > 0$).
In \S\ref{ch3:sec:discussion}
we provide a heuristic explanation for this phenomenon,
which is shared by different models
without moment explosion.
\end{itemize}

The moment generating function of our model admits no closed formula,
but is still manageable enough to derive sharp tail estimates, cf.\ Theorem~\ref{ch3:th:tailprob}.
These are based on large deviations bounds for suitable
functionals of a Poisson process, which might be of independent interest
(see Corollary~\ref{ch3:coro:LDP I} and
Remark~\ref{ch3:rem:LD}).
From these estimates, we derive asymptotic formulas for option price and implied volatility
using the general approach in \cite{cf:GL} and \cite{cf:CC},
that we summarize in \S\ref{ch3:sec:reminders} and \S\ref{ch3:sec:reminders2}.

\smallskip

The paper is organized as follows.
\begin{itemize}
\item In Section~\ref{ch3:sec:model} we define the model and we set up some notation.

\item In Sections~\ref{ch3:sec:impvol} and~\ref{ch3:sec:pricetail} 
we present our main asymptotic results on implied volatility, option price and tail probability,
with a general discussion in \S\ref{ch3:sec:discussion}.

\item In Section~\ref{ch3:sec:ld} we prove some key moment estimates,
which are the cornerstone of our approach, together with some large deviations results
for the Poisson process.

\item The following sections \ref{ch3:sec:LDP}, \ref{ch3:sec:lemma:calltgamma} 
and~\ref{ch3:sec:th:main} contain the proof of our main results concerning
tail probability, option price and implied volatility, respectively.

\item Finally, some technical results have been deferred to the Appendix~\ref{ch3:sec:app}.
\end{itemize}

\section{The model}
\label{ch3:sec:model}

In \S\ref{ch3:sec:hist} we recall the definition of the process $(Y_t)_{t\ge 0}$, introduced
in \cite{cf:ACDP}, for the de-trended log-price of a financial asset
under the historical measure. In \S\ref{ch:sec:mart} we describe its evolution under 
the risk-neutral measure (switching notation to $(X_t)_{t\ge 0}$ for clarity)
and in \S\ref{ch:sec:opiv} we define the price of a call option and the related implied volatility.

\subsection{The historical measure}
\label{ch3:sec:hist}

We fix four real
parameters $0 < D < \frac{1}{2}$, $V > 0$, $\lambda > 0$ and $\tau_0 < 0$,
whose meaning is discussed in a moment.
We consider a stochastic volatility model $(Y_t)_{t\ge 0}$,
with $Y_0 := 0$, defined by
\begin{equation} \label{ch3:eq:Ystart}
     \dd Y_t = \sigma_t \, \dd B_t \,,
\end{equation}
where $(B_t)_{t\ge 0}$ is a Brownian motion and
$(\sigma_t)_{t\ge 0}$ is an \emph{independent} process,
built as follows: denoting by $(N_t)_{t\ge 0}$ a Poisson process of intensity 
$\lambda$ (independent of $(B_t)_{t\ge 0}$)
with jump times $0 < \tau_1 < \tau_2 < \ldots$, we set
\begin{equation} \label{ch3:eq:sigmat}
	\sigma_t := \sfc
	\, \frac{\sqrt{2D}}{(t-\tau_{N_t})^{\frac{1}{2}-D}} \,, \qquad \text{where} \qquad
	\sfc := \frac{\lambda^{D-\frac{1}{2}} \, V}{\sqrt{\Gamma(2D+1)}} \,,
\end{equation}
and $\Gamma(\alpha) := \int_0^\infty x^{\alpha-1} e^{-x} \dd x$ is Euler's gamma function.
Note that $\tau_{N_t} = \max\{\tau_k: \ \tau_k \le t\}$ 
is the last jump time of the Poisson process
before $t$, hence the volatility $\sigma_t$ \emph{diverges} at the jump
times of the Poisson process, which can be thought as shocks in the market.
We refer to Figure~\ref{ch3:fig:IT} for a graphical representation.

\begin{figure}[t]
 \centering
\subfloat[][\emph{Volatility 
$(\sigma_t)_{t\ge 0}$}]{\includegraphics[width=.45\columnwidth]{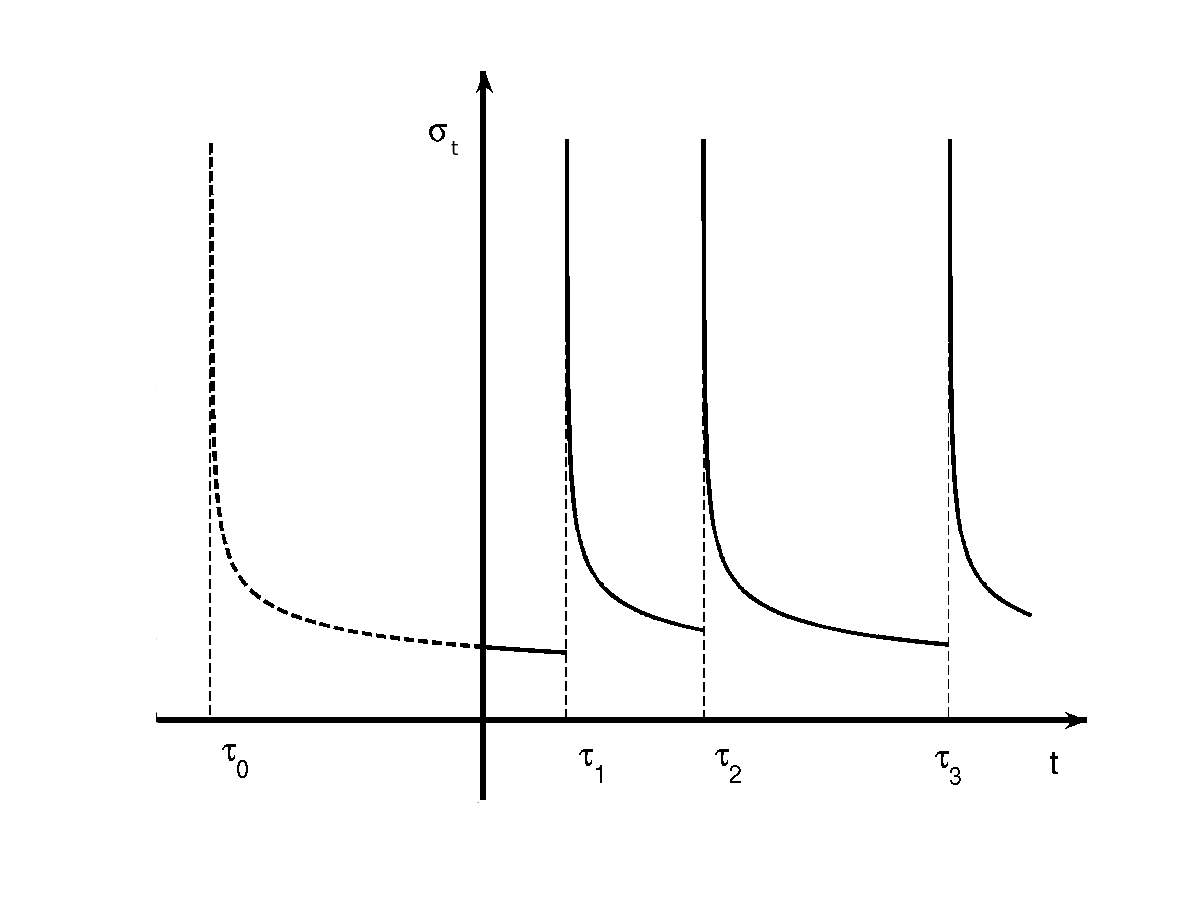}\label{figura2}}
\subfloat[][\emph{Time change $(I_t)_{t\ge 0}$}]{\includegraphics[width=.45\columnwidth]{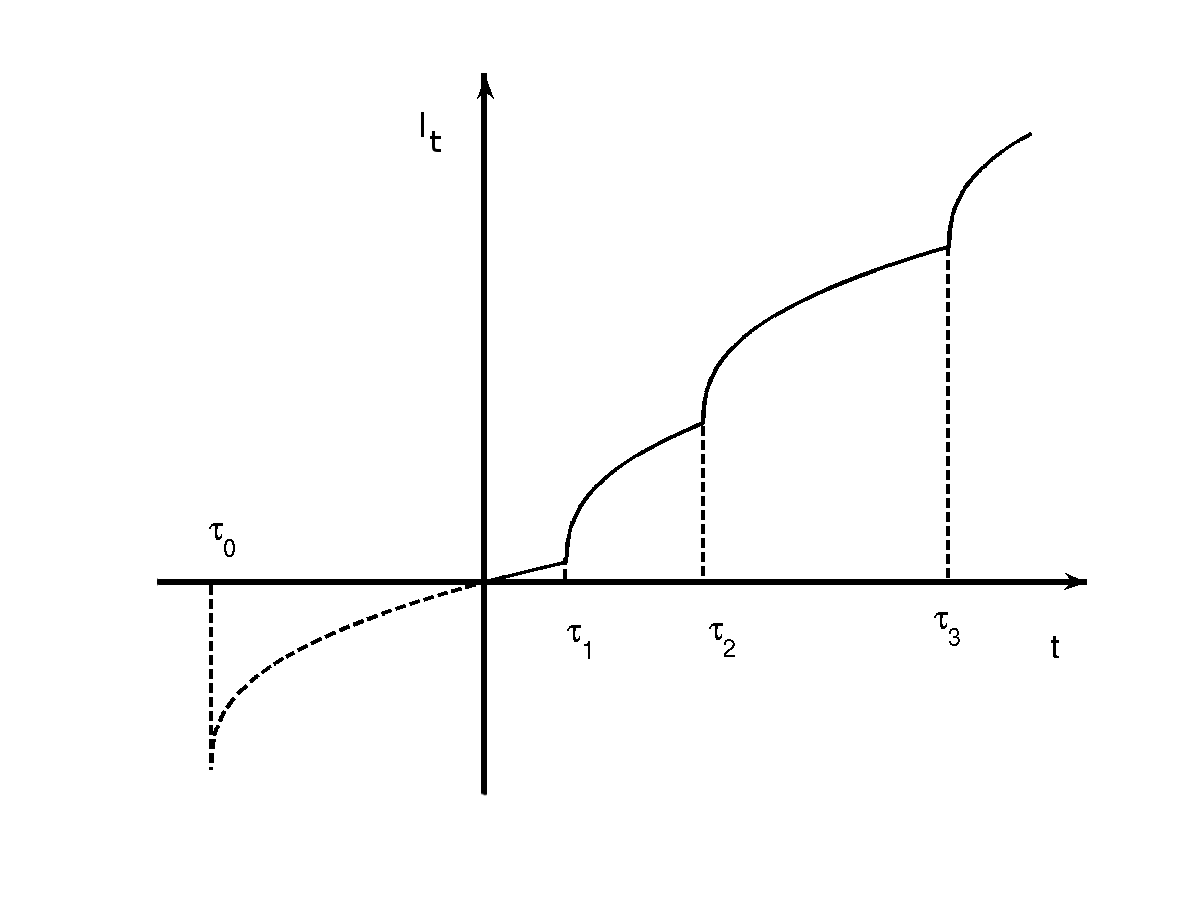}}
\caption{Paths of the \emph{time change} 
and of the \emph{spot volatility process}}
\label{ch3:fig:IT}
\end{figure}

We can now describe the meaning of the parameters:

\begin{itemize}
\item $\lambda \in (0,\infty)$ represents the average frequency of shocks;
\item $D \in (0,\frac{1}{2})$ tunes the decay exponent of the volatility after
a shock;
\item $V \in (0,\infty)$ represents the \emph{large-time volatility},\footnote{The
constant $\sfc$ in \eqref{ch3:eq:sigmat}
was called $\sigma$ in \cite{cf:ACDP} and
used as a parameter in place of $V$
(note that $\sfc$ and $V$ are proportional). 
Our preference for $V$ is due to its direct meaning as large-time
volatility, by \eqref{ch3:eq:V}.}
because (see Appendix~\ref{sec:avvol})
\begin{equation} \label{ch3:eq:V}
	V = \lim_{t\to\infty}\sqrt{\E[\sigma_t^2]} \,;
\end{equation}

\item $\tau_0 \in (-\infty,0)$
tunes the \emph{initial volatility} $\sigma_0$, since
\begin{equation} \label{ch3:eq:sigma0}
	\sigma_0 = \sfc \, \sqrt{2D} \, (-\tau_0)^{D-\frac{1}{2}}
	= \frac{\lambda^{D-\frac{1}{2}} \, V}{\sqrt{\Gamma(2D)}} \,
	(-\tau_0)^{D-\frac{1}{2}}  \,.
\end{equation}
Given this correspondence,
\emph{one can use $\sigma_0$ as a parameter
instead of $\tau_0$}.\footnote{We point out that in \cite{cf:ACDP}
the parameter $-\tau_0$ was chosen randomly, as an independent $Exp(\lambda)$ random variable
(just like $\tau_1$, $\tau_2-\tau_1$, $\tau_3-\tau_2$, \ldots).
With this choice, the process $(t-\tau_{N_t})_{t\ge 0}$ becomes \emph{stationary}
(with $Exp(\lambda)$ one-time marginal distributions), hence
the volatility $(\sigma_t)_{t\ge 0}$ is a stationary process too, by \eqref{ch3:eq:sigmat}.
In our context, it is more natural to have a fixed value for the initial
volatility.}
\end{itemize}

As discussed in \cite{cf:ACDP},
the process $(Y_t)_{t\ge 0}$ in \eqref{ch3:eq:Ystart}
can be represented as a \emph{time-changed Brownian motion}: more precisely,
\begin{equation} \label{ch3:eq:model0}
	Y_t := W_{I_t} \,, \qquad \text{with} \qquad
       I_t := \int_0^t \sigma_s^2 \, \dd s \,,
\end{equation}
where $(W_t)_{t\ge 0}$ is another Brownian motion, independent of $(I_t)_{t\ge 0}$.
It follows by \eqref{ch3:eq:sigmat} that
for $t \in [\tau_{k}, \tau_{k+1}]$ one has $I_t - I_{\tau_k} =
\sfc^2 (t-\tau_k)^{2D}$, cf.\ \eqref{ch3:eq:sigmat}, hence
\begin{equation} \label{ch3:eq:It}
\begin{split}
	I_t & := \sfc^2
	\Bigg\{ (t-\tau_{N_t})^{2D} - (-\tau_0)^{2D} +
	\sum_{k=1}^{N_t} (\tau_k - \tau_{k-1})^{2D} \Bigg\} \,,
\end{split}
\end{equation}
with the convention that the sum
in \eqref{ch3:eq:It} is zero when $N_t = 0$
(see Figure~\ref{ch3:fig:IT}).

\begin{remark}\label{ch3:rem:BS}\rm
In the limiting case $D = \frac{1}{2}$ one has $\sigma_t = V$ and $I_t = V^2 t$,
hence our model reduces to Brownian motion with constant volatility:
$Y_t = V B_t = W_{V^2 t}$. We exclude this case from our analysis just
because it has to be treated separately in the proofs.
\end{remark}

\subsection{The risk-neutral measure}
\label{ch:sec:mart}

We are going to consider a natural risk-neutral measure, 
under which the price $(S_t)_{t\ge 0}$
evolves according to the stochastic differential equation
\begin{equation} \label{ch3:sdeprice}
	\frac{\dd S_t}{S_t} = \sigma_t \, \dd B_t \,,
\end{equation}
where $\sigma_t$ is the process defined in \eqref{ch3:eq:sigmat}.
As a matter of fact, 
there is a one-parameter class of equivalent martingale measures
which allow to modify the value of the parameter $\lambda \in (0,\infty)$
freely (see Appendix~\ref{ch3:sec:martingale}). 
Here we assume to have fixed that parameter, and still call it $\lambda$.

Let us denote by $(X_t)_{t\ge 0}$ the \emph{log-price} process under the risk-neutral measure:
\begin{equation*}
	X_t := \log S_t \,,
\end{equation*}
with $X_0 = 0$, i.e.\ $S_0 = 1$.
It follows by \eqref{ch3:sdeprice} that
$\dd X_t = \sigma_t \, \dd B_t - \frac{1}{2} \sigma_t^2 \, \dd t$, hence
\begin{equation} \label{ch3:equivX}
	X_t = W_{I_t} - \frac{1}{2} I_t \,,
\end{equation}
where the process $(I_t)_{t\ge 0}$,
cf.\ \eqref{ch3:eq:model0}-\eqref{ch3:eq:It},
is independent of the Brownian motion $(W_t)_{t\ge 0}$.
As a consequence, the price $(S_t)_{t\ge 0}$ is
a \emph{time-changed geometric Brownian motion}:
\begin{equation} \label{ch3:eq:St}
	S_t = e^{X_t} = e^{W_{I_t} - \frac{1}{2}I_t} \,.
\end{equation}
Representations \eqref{ch3:equivX}, \eqref{ch3:eq:St} 
are so useful that we can take them as definitions
of our model.

\begin{definition}\label{def:def}
The log-price $(X_t)_{t\ge 0}$ and price $(S_t)_{t\ge 0}$ processes,
under the risk-neutral measure,
evolve according to \eqref{ch3:equivX} and \eqref{ch3:eq:St}
respectively,
with $(I_t)_{t\ge 0}$ defined in \eqref{ch3:eq:It}.
\end{definition}

\subsection{Option price and implied volatility}
\label{ch:sec:opiv}

The price of a (normalized) European call option, with log-strike $\kappa \in \R$
and maturity $t \ge 0$, under our model is
\begin{equation} \label{ch3:c}
	c(\kappa,t) := \E[(S_t - e^\kappa)^+] = \E[(e^{X_t}-e^{\kappa})^+] \,.
\end{equation}

We recall that, for a given volatility parameter $\sigma \in (0,\infty)$, 
the Black\&Scholes price of a European call option
equals $\CBS(\kappa,\sigma \sqrt{t})$, where
\begin{equation} \label{ch3:eq:BS1}
	\CBS(\kappa,v) 
	:= \E [ (e^{W_{v^2} - \frac{1}{2} v^2} - e^{\kappa})^+ ]
	= \begin{cases}
	(1-e^\kappa)^+ & \text{if } v = 0 \,, \\
	\rule{0pt}{1.2em}\Phi(d_1) - e^\kappa \Phi(d_2)  & \text{if } v > 0 \,,
	\end{cases}
\end{equation}
where
\begin{equation} \label{ch3:eq:BS2}
	\Phi(x) := \int_{-\infty}^x \frac{e^{-\frac{1}{2} t^2}}{\sqrt{2\pi}} \, \dd t \,, \qquad
	d_1 := - \frac{\kappa}{v} + \frac{v}{2} \,, \qquad
	d_2 := - \frac{\kappa}{v} - \frac{v}{2} \,.
\end{equation}
Since $\Phi(-x) = 1-\Phi(x)$, the following symmetry relation holds:
\begin{equation}\label{ch3:eq:symBS}
	\CBS(-\kappa,v) = 1-e^{-\kappa} + e^{-\kappa} \CBS(\kappa,v) \,.
\end{equation}

\begin{definition}\label{ch3:def:iv}
For $t > 0$ and $\kappa\in\R$,
the \emph{implied volatility} $\sigma_{\rm imp}(\kappa,t)$
of our model is
the unique value of $\sigma\in (0,\infty)$ such that $c(\kappa,t)$ in \eqref{ch3:c} equals
$\CBS(\kappa, \sigma\sqrt{t})$, that is
\begin{equation} \label{ch3:eq:IV}
	c(\kappa,t) = \CBS(\kappa, \sigma_{\rm imp}(\kappa,t) \sqrt{t}) \,.
\end{equation}
\end{definition}

Recalling \eqref{ch3:eq:St}, since $(I_t)_{t\ge 0}$ is independent of $(W_t)_{t\ge 0}$, the call price
$c(\kappa,t)$ in \eqref{ch3:c} enjoys the representation
\begin{equation} \label{ch3:HW}
	c(\kappa,t) = \E \big[ \CBS(\kappa, v)
	\big|_{v = \sqrt{I_t}} \big] \,,
\end{equation}
known as \emph{Hull-White formula} \cite{cf:HW}.
As a consequence, the symmetry relation \eqref{ch3:eq:symBS} transfers from Black\&Scholes to our model:
\begin{equation} \label{ch3:eq:asd}
	c(-\kappa,t) = 1-e^{-\kappa}  + e^{-\kappa} c(\kappa,t) \,.
\end{equation}
Looking at \eqref{ch3:eq:IV}, it follows that \emph{the implied volatility
of our model is symmetric in $\kappa$}:
\begin{equation} \label{ch3:eq:sym}
	\sigma_\imp(-\kappa,t) = \sigma_\imp(\kappa,t) \,.
\end{equation}
As a consequence, in the sequel we focus on the regime $\kappa \ge 0$.

\begin{remark}\rm\label{ch3:rem:shared}
Properties \eqref{ch3:HW}-\eqref{ch3:eq:asd}-\eqref{ch3:eq:sym}
hold for \emph{any} stochastic volatility model \eqref{ch3:sdeprice}
for which the volatility $(\sigma_t)_{t\ge 0}$ is independent of the Brownian motion
$(B_t)_{t\ge 0}$, because any such model enjoys the representation \eqref{ch3:eq:St},
with $(I_t)_{t\ge 0}$ defined as in \eqref{ch3:eq:model0}
(cf.\ \cite{cf:RT}).
\end{remark}

\section{Main results: implied volatility}
\label{ch3:sec:impvol}

In this section we present our main results on
the asymptotic behavior of the implied volatility $\sigma_\imp(\kappa,t)$
of our model. We allow for a variety of regimes with bounded maturity.
More precisely, we consider an \emph{arbitrary family of values of $(\kappa,t)$} such that
\begin{equation}\label{ch3:eq:assfamily}
	\text{either $t \to \bar t \in (0,\infty)$ and $\kappa \to \infty$\,,
	\quad \
	or $t \to 0$ with arbitrary $\kappa \ge 0$} \,.
\end{equation}
Allowing for both sequences $((\kappa_n,t_n))_{n\in\N}$
and functions $((\kappa_s,t_s))_{s\in [0,\infty)}$, we omit subscripts.

We agree with the conventions $\N := \{1,2,3,\ldots\}$ and $\N_0 := \N \cup \{0\}$.
We are going to use the following asymptotic notations, for positive functions $f,g$:
\begin{align}\label{ch3:eq:notation}
	f \sim g, \ \  f \ll g, \ \ f \gg g \qquad &\iff \qquad
	\frac{f}{g} \to 1, \ \ \frac{f}{g} \to 0, \ \ \frac{f}{g} \to \infty 
	\quad \ \text{respectively}, \\
	\label{ch3:eq:notation2}
	f \asymp g \qquad &\iff \qquad \log f \sim \log g \,, 
	\quad \ \text{i.e.} \quad \
	\frac{\log f}{\log g} \to 1 \,.
\end{align}

\subsection{Auxiliary functions}
We introduce two functions $\bkappa_1, \bkappa_2 : (0,1) \to (0,\infty)$ by
\begin{equation} \label{ch3:eq:gh}
	\bkappa_1(t) := \sqrt{t} \sqrt{\log \tfrac{1}{t}} \,, \qquad
	\bkappa_2(t) := t^D \sqrt{\log \tfrac{1}{t}} \,,
\end{equation}
which will act as boundaries for $\kappa$, separating
different asymptotic regimes as $t \to 0$.
Recall that $D$ determines the decay exponent of the volatility after a shock 
(cf. \eqref{ch3:eq:sigmat}), and
note that $\bkappa_1(t) < \bkappa_2(t)$, because $D < \frac{1}{2}$ by assumption.
 
\begin{remark}\rm
We point out that $\bkappa_1(t)$ is the same scaling considered 
by Mijatovi\'{c} and Tankov in the paper \cite{cf:MT}.
For exponential L\'evy models with
jumps, they show that
\begin{equation*}
	\text{for } a \in (0,\infty): \qquad
	\lim_{t\to 0} \, \sigma_\imp^{\text{\cite{cf:MT}}}\big( a\, \bkappa_1(t), t \big) =
	\max\bigg\{\sigma, \frac{a}{\sqrt{1-(\alpha_+-1)^+}} \bigg\} \,,
\end{equation*}
where $\sigma$ is the volatility of the Brownian component and $\alpha_+$
is the activity index of the positive jumps of the L\'evy process 
(see equation (1.1) in \cite{cf:MT}).
Our model is quite different ---it has jump in the volatility, not in the price---
but, remarkably, it exhibits similar features: by formulas \eqref{ch3:eq:smile1} 
and \eqref{ch3:eq:smile0} below
\begin{equation*}
	\text{for } a \in (0,\infty): \qquad
	\lim_{t\to 0} \, \sigma_\imp\big( a\, \bkappa_1(t), t \big) =
	\max\bigg\{\sigma_0, \frac{a}{\sqrt{2D+1}} \bigg\} \,.
\end{equation*}

Note that $\kappa \sim a \, \bkappa_1(t)$ is the regime when the tail probability
of Brownian motion is polynomial in $t$, since
$\P(W_t > \kappa) \asymp \exp(-\kappa^2/(2t)) = t^{a^2/2}$.
The probability that a Poisson process has $k$ jumps in the interval
$[0,t]$ is also polynomial in $t$, namely $O(t^k)$. These
considerations suggest that $\kappa \sim a \, \bkappa_1(t)$ is a
``transition regime'',
when the Brownian component and the jumps 
(in the price for \cite{cf:MT}, in the volatility for us)
have comparable effects.
\end{remark}

We also define an auxiliary function $\sff: (0,\infty) \to \R$ by
\begin{equation}\label{ch3:eq:f}
	\sff(a) := \min_{m \in \N_0} \sff_m(a) \,, \qquad
	\text{with} \qquad
	\sff_m(a) := m + \frac{a^2}{2\, m^{1-2D}}  \,,
\end{equation}
We point out that the minimization
can be performed explicitly (see Appendix~\ref{ch3:sec:minimum}).
In particular, the function $\sff$ 
is continuous, strictly increasing and satisfies
\begin{equation}\label{ch3:eq:asf}
	\sff(a) \sim \begin{cases}
	\displaystyle 1 + \frac{ a^2}{2}
	& \text{as } a \downarrow 0 \\
	\rule[-0.8em]{0pt}{2.9em}\displaystyle
	\Big((1-D)^{\frac{1/2-D}{1-D}}
	\, \sfC\Big) \, 
	a^{\frac{1}{1-D}}
	& \text{as } a \uparrow \infty
	\end{cases} , \qquad \text{with} \quad \ 
	\sfC := 
	\frac{(1-D)^{\frac{1/2}{1-D}}}
	{(\frac{1}{2}-D)^{\frac{1/2-D}{1-D}}} \,.
\end{equation}


\subsection{Implied volatility}

The next theorem, proved in Section~\ref{ch3:sec:th:main}, is our main result.
It provides a \emph{complete asymptotic picture} of the implied volatility
in any regime \eqref{ch3:eq:assfamily} of small maturity and/or large strike
(see Figures~\ref{ch3:fig:smile} and~\ref{fig:comp}).
The corresponding asymptotic results for the tail probability
$\P(X_t > \kappa)$ and for the option price $c(\kappa,t)$
are presented in Section~\ref{ch3:sec:pricetail}.

\begin{theorem}[Implied volatility]\label{ch3:th:main}
The implied volatility $\sigma_\imp(\kappa,t)$ diverges
in the small maturity regime $t \to 0$ as soon as $\kappa \gg 
\bkappa_1(\sigma_0^2 \, t)$ 
(in particular: for any fixed $\kappa \ne 0$, i.e.\ out of the money). 
More precisely, consider a family of values of $(\kappa,t)$ with $\kappa \ge 0$, $t > 0$. 
We recall that $\sff(\cdot)$
is defined in \eqref{ch3:eq:f} and $\sfc$, $\sigma_0$, $\sfC$ are defined in 
\eqref{ch3:eq:sigmat},
\eqref{ch3:eq:sigma0}, \eqref{ch3:eq:asf}.

\begin{aenumerate}
\item\label{iv:a} If $\,t \to \bar t \in (0,\infty)$ and $\,\kappa \to \infty$, 
or if $\,t \to 0$ and $\kappa \gg \bkappa_2(\sfc^{1/D} \, t)$ (e.g.,
$\kappa \to \bar\kappa \in (0,\infty]$),
\begin{equation}\label{ch3:eq:smile}
\begin{split}
\sigma_\imp(\kappa,t) & \sim 
\sqrt{\frac{\sfc^{1/D}}{2 \sfC}} 
	\left(\frac{\rule[-1em]{0pt}{1em}\displaystyle\frac{\kappa}{\sfc^{1/D} \, t}}
	{\displaystyle\sqrt{\log \frac{\kappa}{\sfc^{1/D} \, t}}}\right)^{\frac{1/2-D}{1-D}} 
	\,.
\end{split}
\end{equation}

\item\label{iv:b} If $\,t\to 0$ and 
$\, \kappa \sim a \, \bkappa_2(\sfc^{1/D} \, t)$, for some $a \in (0,\infty)$,
\begin{equation}\label{ch3:eq:smile2}
	\rule{0pt}{1.3em}
	\sigma_\imp(\kappa,t)\sim 
	\frac{\sqrt{\lambda}}{\sqrt{2\, 
	\sff\big( \rho_t \, a \big)}} 
	\,
	\frac{\kappa}{\bkappa_1(\lambda \, t)} \,, \qquad
	\text{where} \quad 
	\rho_t := \sqrt{\frac{\log (\sfc^{1/D} \, t)}{\log (\lambda \, t)} } \,.
\end{equation}
(One could actually replace $\sff(\rho_t \, a)$ by $\sff(a)$,
because $\lim_{t\to 0} \rho_t = 1$,
but keeping $\rho_t$ gives better approximations of the true
implied volatility, when $\sfc^{1/D}$ and $\lambda$ are different.)
 
\item\label{iv:c} If $\,t \to 0$ and 
$\sqrt{2D+1} \, \bkappa_1( \sigma_0^2 \, t) \leq 
\kappa \ll \bkappa_2( \sfc^{1/D} \, t)$, 
\begin{equation}\label{ch3:eq:smile1}
   \rule{0pt}{1.3em}
	\sigma_\imp(\kappa,t)\sim 
	\frac{ \sqrt{\lambda}}{\sqrt{2
	\big(1-  \xi_t(\kappa) \big) }}  \,
  \frac{\kappa}{\bkappa_1( \lambda\, t)} \,, \qquad
  	\text{where} \quad 
  \xi_t(\kappa) := \frac{\log
	\left( \kappa / \bkappa_2(\sfc^{1/D}\,t) \right)}{\log
	(\lambda\,  t)} \,,
\end{equation}
and note that $\xi_t(\kappa) \in [0, \frac{1}{2}-D]$
for $\kappa$ in the range under consideration.

\item\label{iv:d} Finally, if $\,t \to 0$ and 
$\,0 \leq \kappa \leq 
\sqrt{2D+1} \, \bkappa_1( \sigma_0^2 \, t)$,
\begin{equation}\label{ch3:eq:smile0}
	\sigma_\imp(\kappa,t) \sim \sigma_0 \,.
\end{equation}

\end{aenumerate}
\end{theorem}

Let us give a qualitative description of Theorem~\ref{ch3:th:main}.
Recall \eqref{ch3:eq:notation2}, \eqref{ch3:eq:gh} and
note that $\bkappa_1(t) \asymp \sqrt{t}$ and $\bkappa_2(t) \asymp t^D$.
If we fix $t > 0$ small and increase $\kappa \in [0,\infty)$,
we can describe the implied volatility $\sigma_\imp(\kappa,t)$ as follows
(cf.\ Figure~\ref{ch3:fig:smile}):
\begin{itemize}
\item $\sigma_\imp(\kappa,t)  \sim \sigma_0$ is roughly constant 
 from $\kappa = 0$ up to $\kappa \asymp \sqrt{t}$, cf.\ \eqref{ch3:eq:smile0};

\item then $\sigma_\imp(\kappa,t) \asymp \kappa / \sqrt{t}$ grows linearly
from $\kappa \asymp \sqrt{t}$
up to $\kappa \asymp t^D$, cf.\ \eqref{ch3:eq:smile1};

\item then $\sigma_\imp(\kappa,t)  \asymp (\kappa/t)^\gamma$ grows sublinearly
from $\kappa \asymp t^D$ to $\kappa = \infty$, cf.\ \eqref{ch3:eq:smile}, 
with an exponent 
$\gamma = \frac{1/2-D}{1-D}$ that can take any value
in $(0,\frac{1}{2})$ depending on $D$.
\end{itemize}
We stress that formula $\sigma_\imp(\kappa,t) \asymp (\kappa/t)^\gamma$ holds
also as $t \downarrow 0$ for fixed $\kappa > 0$. 

\begin{figure}[t]
\centering
\includegraphics[width=.75\columnwidth]{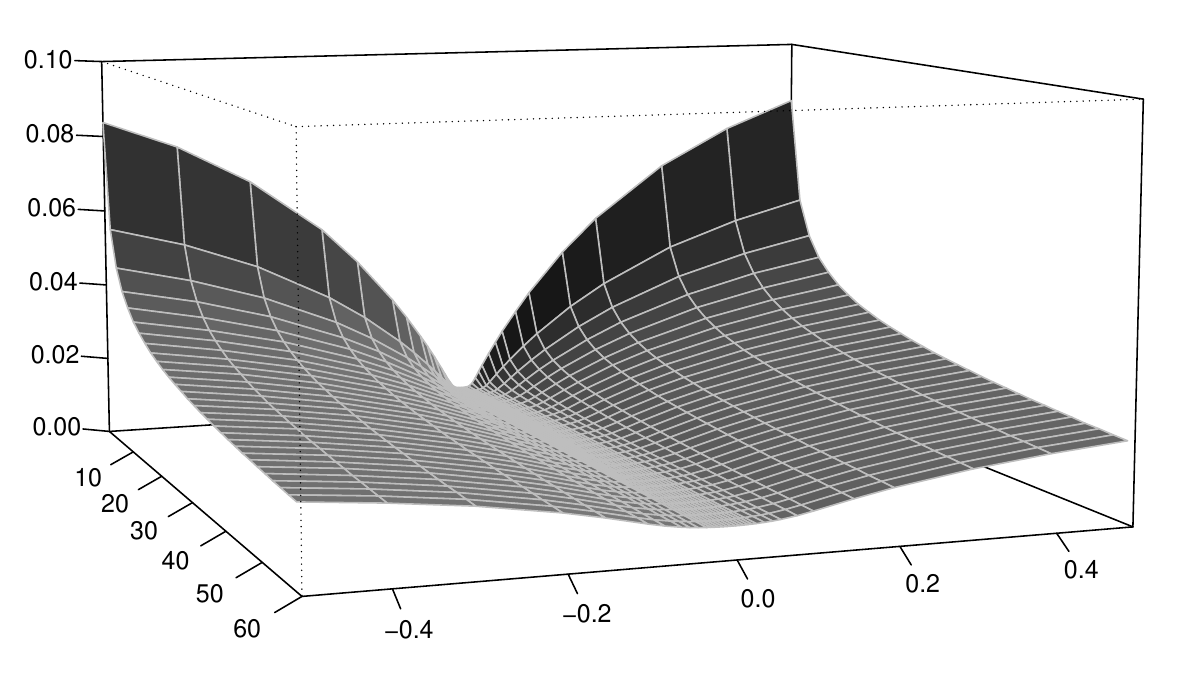}\\
\includegraphics[width=.75\columnwidth]{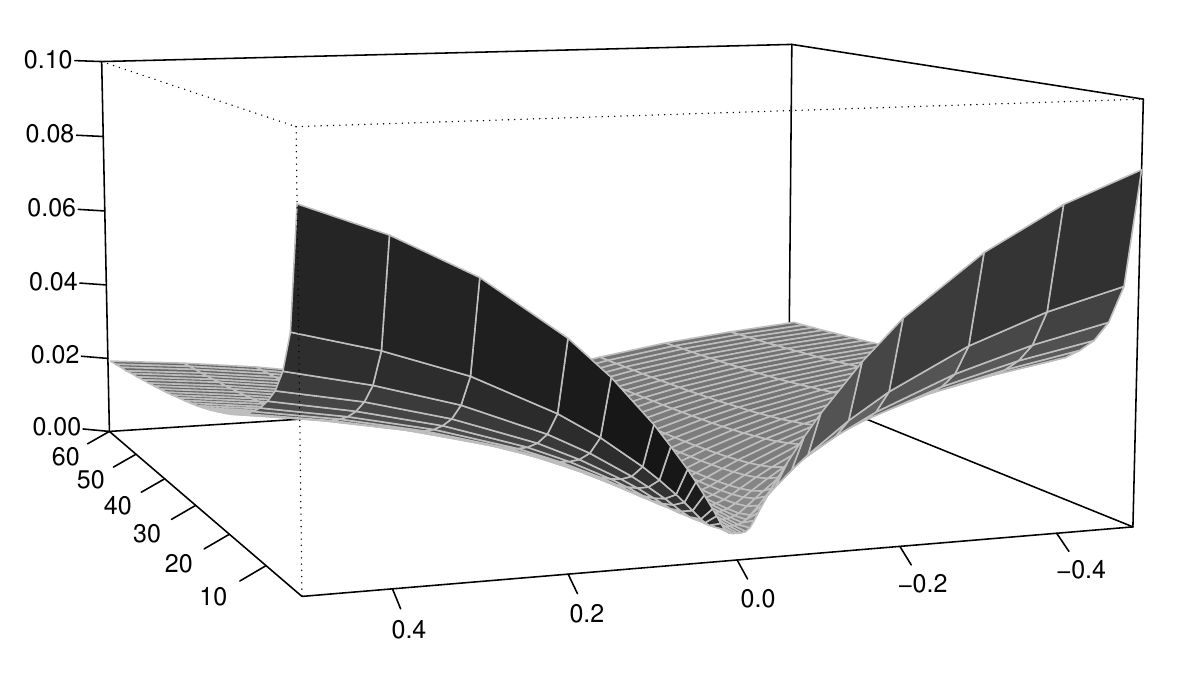}
\caption{\label{ch3:fig:smile}A plot of the implied volatility surface
$\sigma_\imp(\kappa,t)$ of our model, obtained
by Monte Carlo simulations,
for maturity $t \in \{1, \ldots, 60\}$ days and log-strike $\kappa
\in [-0.5, 0.5]$ (for graphical clarity, only odd days are drawn).
We refer to the caption of Figure~\ref{fig:comp} for the choice of parameters.}
\end{figure}

\begin{figure}
\centering
\subfloat[][\hfill $\kappa = a \, \bkappa_2(\sfc^{1/D} \, t)$ \hfill\, \\
$a = 3$ (bottom), $a = 9$ (top)]
{\includegraphics[width=.43\columnwidth]{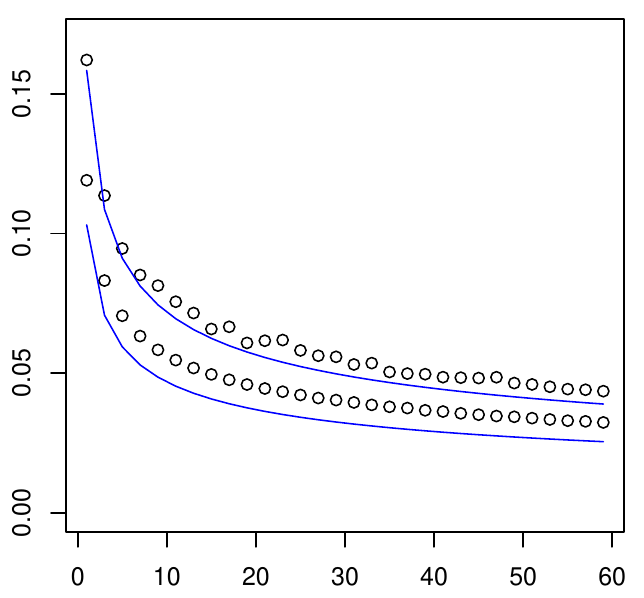}\label{fig:a}}
\qquad \
\subfloat[][\hfill $\kappa = a \, \bkappa_2(\sfc^{1/D} \, t)$ \hfill\, \\
$a \!\in\! \{.5, 1, 2\}$ (bottom to top)]
{\includegraphics[width=.43\columnwidth]{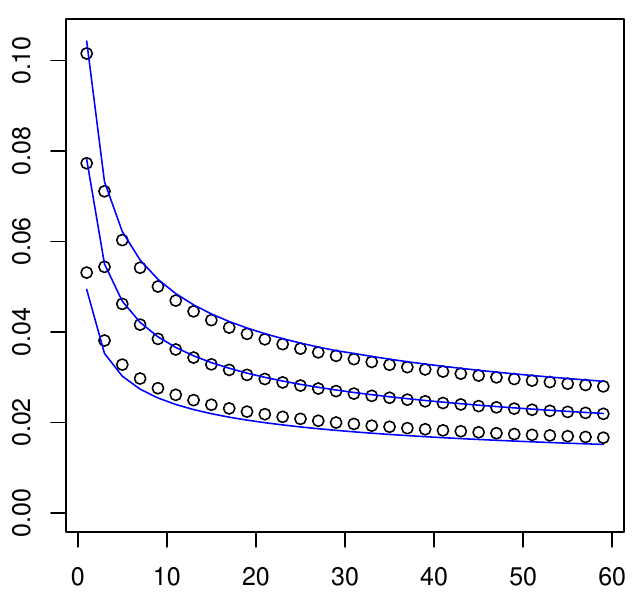}\label{fig:b}}\\
\smallskip
\subfloat[][\hfill $\kappa = a \, \bkappa_2(\sfc^{1/D} \, t)$ \hfill\, \\
$a =\! .1$ (bottom), $a=\!.3$ (top)]
{\includegraphics[width=.43\columnwidth]{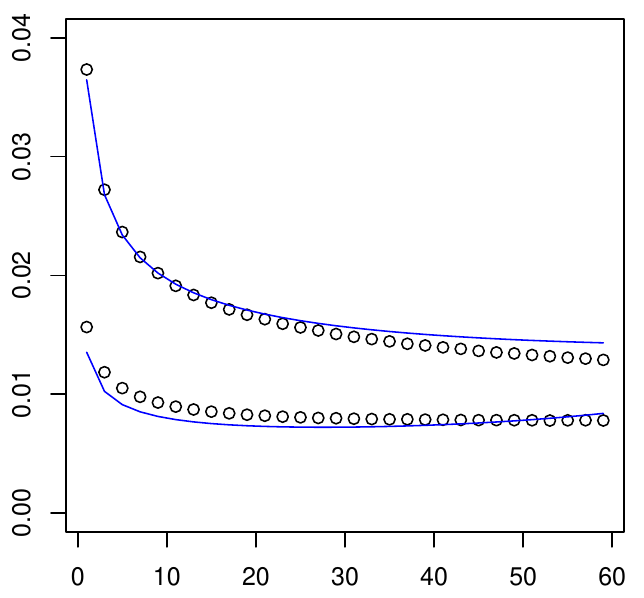}\label{fig:c}}
\qquad \
\subfloat[][\hfill $\kappa = .2 \, \bkappa_1(\sigma_0^2 \, t)$ \hfill\, \\
\phantom{safd asdf aadsfa s afsd}]
{\includegraphics[width=.43\columnwidth]{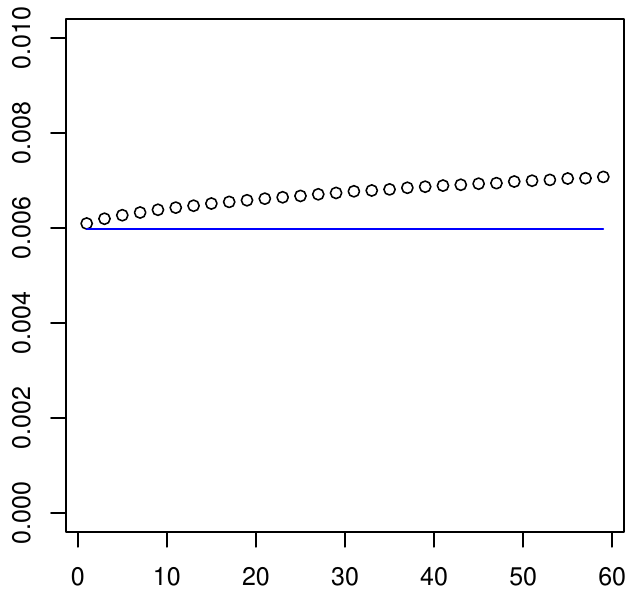}\label{fig:d}}
\caption{\label{fig:comp} Comparison between
the true implied volatility
$\sigma_\imp(\kappa,t)$ of our model (circles, Monte Carlo) and the asymptotic formulas in
Theorem~\ref{ch3:th:main} (solid lines),
for maturity $t \in \{1, \ldots, 60\}$ days and log-strike $\kappa$ as specified
in each caption (only odd days are drawn).
Formulas \eqref{ch3:eq:smile}, \eqref{ch3:eq:smile2}, \eqref{ch3:eq:smile1},
\eqref{ch3:eq:smile0} are plotted in 
 figures {\sc\subref{fig:a}}, {\sc\subref{fig:b}}, {\sc\subref{fig:c}}, {\sc\subref{fig:d}}, respectively.\\
We have fixed ``typical'' values of the parameters $D = 0.16$,
$\lambda = 10^{-3}$, $V = 10^{-2}$ (close to those estimated
in \cite{cf:ACDP} on the \emph{Dow Jones Industrial Average} time series)
and $\tau_0 = 1/\lambda = 10^3$
(equivalently, $\sigma_0 \simeq 6 \cdot 10^{-3}$,
cf.\ \eqref{ch3:eq:sigma0} and figure~{\sc\subref{fig:d}}), which yield
$\sfc^{1/D} \simeq 10^{-6}$, cf.\ \eqref{ch3:eq:sigmat}.\\
(The irregular behavior of the top circles in figure~{\sc\subref{fig:a}}
is due to Monte Carlo numerical inaccuracies, caused by the large
values of $\kappa$.) \\
The plots were generated using the software R \cite{R}.
The code is available on the web
page \url{http://www.matapp.unimib.it/~fcaraven/c.html}.}
\label{fig:variability}
\end{figure}

\begin{remark}\rm
In Theorem~\ref{ch3:th:main},
the maturity $t$ enters the functions $\bkappa_1(\cdot)$,
$\bkappa_2(\cdot)$
only through the \emph{dimensionless quantities $\lambda t$,
$\sigma_0^2 t$ and $\sfc^{1/D} t$}, which do not depend on the unit in which the maturity
is measured. This is relevant for the accuracy of our formulas, cf.\ Figure~\ref{fig:comp}.
(Note that both $\bkappa_1(t)$ and $\bkappa_2(t)$ contain the term
$\log \frac{1}{t}$: even though $\log \frac{1}{\lambda t} \sim \log \frac{1}{t}$ as $t \to 0$,
for non-zero values of $t$ there can be an important difference between 
$\log\frac{1}{\lambda t}$ and $\log\frac{1}{t}$.)
\end{remark}

\begin{remark}\rm
The four relations \eqref{ch3:eq:smile}, \eqref{ch3:eq:smile2},
\eqref{ch3:eq:smile1} and \eqref{ch3:eq:smile0}
match at the boundaries of the respective intervals of applicability: 
\begin{itemize}
\item relation \eqref{ch3:eq:smile1} reduces to \eqref{ch3:eq:smile0}
for $\kappa = \sqrt{2D+1} \, \bkappa_1(\sigma_0^2 \, t)$;

\item relation \eqref{ch3:eq:smile2} reduces to \eqref{ch3:eq:smile1} as $a \to 0$, by \eqref{ch3:eq:asf};

\item relation \eqref{ch3:eq:smile2} matches with \eqref{ch3:eq:smile} as $a \to \infty$.
In fact, plugging $a = \kappa / \bkappa_2(\sfc^{1/D}\, t)$ in \eqref{ch3:eq:smile2}, and
using the asymptotic relation \eqref{ch3:eq:asf} as $a \to \infty$,
relation \eqref{ch3:eq:smile2} becomes
\begin{equation*}
	\sigma_\imp(\kappa,t) \sim 
	\sqrt{\frac{\sfc^{1/D}}{2 \sfC}} \, 
	\left(\frac{\rule[-1em]{0pt}{1em}\displaystyle\frac{\kappa}{\sfc^{1/D} \, t}}
	{\sqrt{(1-D) \log \frac{1}{\lambda \, t}}}\right)^{\frac{1/2-D}{1-D}} \,,
\end{equation*}
and note that, for $\kappa \sim a \, \bkappa_2(\sfc^{1/D} \, t)$, one has
$(1-D) \log \frac{1}{\lambda \, t} \sim \log \frac{\kappa}{\sfc^{1/D} \, t}$ as $t\to 0$.
\end{itemize}

\end{remark}

\begin{remark}\rm
In the limiting case $D = \frac{1}{2}$ (that we exclude from our analysis)
one has $\sigma_0 = \sfc = V$, 
cf.\ \eqref{ch3:eq:sigmat} and \eqref{ch3:eq:sigma0},
and the minimum in \eqref{ch3:eq:f}
is attained for $m=0$  (we adopt the convention $0^0 := 1$), so that $\sff(a) = \sff_0(a) = \frac{a^2}{2}$.
As a consequence, relations \eqref{ch3:eq:smile}, \eqref{ch3:eq:smile2} and \eqref{ch3:eq:smile0}
reduce to $\sigma_\imp(\kappa,t) \sim V$,
in perfect agreement with the fact that for $D = \frac{1}{2}$ 
our model becomes Black\&Scholes model with
constant volatility $V$, cf.\ Remark~\ref{ch3:rem:BS}.\footnote{Note 
that relation \eqref{ch3:eq:smile1} 
does not apply for $D=\frac{1}{2}$, because 
in this case $\bkappa_1(\cdot) = \bkappa_2(\cdot)$ and consequently 
there is no $\kappa$ for which 
$\sqrt{2D+1} \,
\bkappa_1(\sigma_0^2 \, t) \leq \kappa \ll \bkappa_2(\sfc^{1/D}\, t)$.}
\end{remark}

\subsection{On generalized Ornstein-Uhlenbeck processes}
\label{sec:OU}

Let us denote the
set of jump times of the Poisson process $(N_t)_{t\ge 0}$
by $\cT := \{t \in [0,\infty): \ N_t = N_{t-}+1\}$.
Observe that $\sigma_t$ in \eqref{ch3:eq:sigmat} solves the
following differential equation, for any $t \not\in \cT$:
\begin{equation} \label{eq:sde}
	\dd (\sigma_t^2) 
	= - c \, (\sigma_t^2)^\gamma \, \dd t \,, \qquad \text{where} \qquad
	c := \frac{1-2D}{(2D \sfc^2)^{\frac{1}{1-2D}}} \,, \qquad
	\gamma := \frac{2-2D}{1-2D} \,, 
\end{equation}
while for $t\in\cT$ one has $\sigma_t = \infty$.
(Incidentally, note that $\gamma \in (2,\infty)$, since $D \in (0,\frac{1}{2})$.)

Consider a compound Poisson process $L_t = \sum_{k=1}^{N_t} J_k$
with non-negative jump variables $(J_k)_{k\in\N}$.
If we denote by $(\tilde \sigma_t)_{t\ge 0}$ the solution
of equation \eqref{eq:OU}, with $\tilde \sigma_0 = \sigma_0$ and
with the same parameters $c, \gamma$ as in \eqref{eq:sde},
it follows that $\tilde \sigma_t$ solves the same equation
\eqref{eq:sde} as $\sigma_t$, for $t\not\in \cT$. Since $\tilde \sigma_t < \sigma_t = \infty$
for $t \in \cT$, a monotonicity argument shows that 
\begin{equation} \label{eq:IItilde}
	\forall t\ge 0: \qquad
	\tilde\sigma_t \le \sigma_t \qquad \text{and} \qquad
	\tilde I_t := \int_0^t \tilde\sigma_s^2 \, \dd s
	\le \int_0^t \sigma_s^2 \, \dd s =: I_t \,.
\end{equation}

Let us now consider a stochastic volatility model $(\tilde S_t)_{t\ge 0}$
solving $\dd \tilde S_t = \tilde \sigma_t \, \tilde S_t \, \dd B_t$, where
the Brownian motion $(B_t)_{t\ge 0}$ is independent of $(\tilde \sigma_t)_{t\ge 0}$.
Denoting by $\tilde c(\kappa,t)$ and $\tilde\sigma_\imp(\kappa,t)$
the corresponding call price and implied volatility,
the following bounds hold:
\begin{equation}\label{eq:csigma}
	\forall \kappa \in \R, \ t \ge 0: \qquad
	\tilde c(\kappa,t) \le c(\kappa,t) \qquad \text{and} \qquad
	\tilde \sigma_\imp(\kappa,t) \le \sigma_\imp(\kappa,t) \,.
\end{equation}
These follow by \eqref{eq:IItilde} in conjunction with 
\eqref{ch3:eq:IV} and \eqref{ch3:HW},
using the monotonicity of the Black\&Scholes price $\CBS(\kappa,v)$ in the volatility $v$.

\smallskip

We have shown that option price $c(\kappa,t)$
and implied volatility $\sigma_\imp(\kappa,t)$ of our
model give an upper bound for the corresponding quantities of 
\emph{any stochastic volatility model,
with volatility evolving as in
\eqref{eq:OU}}, with $\gamma > 2$ and with a subordinator with finite activity.
In other terms, our model provides \emph{the most extreme volatility profile} in 
this class of models. 
This information could be useful, e.g.,
when matching an observed smile with
a model in this class,
to get a priori bounds on the mean-reversion exponent $\gamma = \frac{2-2D}{1-2D}$
in \eqref{eq:OU}.

In order to improve the upper bounds \eqref{eq:csigma}, 
and possibly to obtain matching lower bounds,
information from the jump variables $J_k$ needs of course to be used.
Extending our results to the general class of models in \eqref{eq:OU}
with a finite activity subordinator
does not appear out of reach, since many of the techniques
we use are quite robust (see Section~\ref{ch3:sec:ld}).

\subsection{Discussion}
\label{ch3:sec:discussion}

We conclude this section with a more detailed discussion of
Theorem~\ref{ch3:th:main}, highlighting the most
relevant points and outlining further directions of research.

\medskip
\noindent
\textbf{Joint volatility surface asymptotics}.
In Theorem~\ref{ch3:th:main}
we allow for arbitrary families of $(\kappa,t)$, besides the usual regimes
$\kappa \to \infty$ for fixed $t$, or $t \downarrow 0$ for fixed $\kappa$.
Interestingly, this flexibility
yields \emph{uniform estimates on the implied volatility surface in open regions of the plane},
as we now show. Recalling \eqref{ch3:eq:gh},
for $T, M \in (0,\infty)$ we define the region
\begin{equation*}
	\cA_{T,M} := \bigg\{(\kappa,t) \in \R^2: \ 0 < t < T, \ \kappa > 
	M \bkappa_2(\sfc^{1/D} \, t) \bigg\} \,.
\end{equation*}

\begin{corollary}[Joint surface asymptotics]\label{ch3:th:surface}
Fix $T > 0$. For every $\epsilon > 0$ there exists $M = M(T,\epsilon) > 0$ such that
for all $(\kappa,t) \in \cA_{T,M}$
\begin{equation}\label{eq:surface}
	(1-\epsilon)
	\sqrt{\frac{\sfc^{1/D}}{2 \sfC}} 
	\left(\frac{\frac{\kappa}{\sfc^{1/D} \, t}}
	{\sqrt{\log \frac{\kappa}{\sfc^{1/D} \, t}}}\right)^{\frac{1/2-D}{1-D}} 
	\le \sigma_\imp(\kappa,t) \le
	(1+\epsilon)
	\sqrt{\frac{\sfc^{1/D}}{2 \sfC}} 
	\left(\frac{\frac{\kappa}{\sfc^{1/D} \, t}}
	{\sqrt{\log \frac{\kappa}{\sfc^{1/D} \, t}}}\right)^{\frac{1/2-D}{1-D}} \,.
\end{equation}
\end{corollary}

\begin{proof}
By contradiction, assume that for some $T, \epsilon > 0$
and for every $M \in \N$ one can find $(\kappa, t) \in \cA_{T,M}$
such that relation \eqref{eq:surface} fails. We can then extract subsequences
$M_n \to \infty$ and $(\kappa_{n}, t_{n}) \in \cA_{T, M_n}$ such that $t_{n} \to \bar t \in [0,T]$.
The subsequence $((\kappa_{n}, t_{n}))_{n\in\N}$ satisfies the assumptions of
part \eqref{iv:a} in Theorem~\ref{ch3:th:main}: if $\bar t > 0$, then $\kappa_{n} \to \infty$, while
if $\bar t = 0$, then $\kappa_{n} \gg \bkappa_2(\sfc^{1/D} \, t_{n})$, because 
$(\kappa_{n}, t_{n}) \in \cA_{T, M_n}$
and $M_n \to \infty$.
However, relation \eqref{ch3:eq:smile} fails by construction and this
contradicts Theorem~\ref{ch3:th:main}.
\end{proof}

\smallskip
\noindent
\textbf{Small-maturity divergence of implied volatility}.
Relation \eqref{ch3:eq:smile} shows
that, for fixed $\kappa > 0$, the implied volatility \emph{diverges}
as $t \downarrow 0$, producing a very steep smile for
small maturity. This is typical
for models with jumps in the price \cite{cf:AL12}, but remarkably our
stochastic volatility model has \emph{continuous paths}.
What lies behind this phenomenon is the very same mechanism
that produces the multi-scaling of moments \cite{cf:ACDP}, i.e.,
the fact that the volatility $\sigma_t$ has \emph{approximate heavy tails}
as $t \downarrow 0$.

In order to give an explanation, 
we anticipate that, under mild assumptions, option price and tail probability are linked
by $c(\kappa,t) \asymp \P(X_t > \kappa)$ as $t \downarrow 0$
for fixed $\kappa > 0$
(see Theorem~\ref{ch2:th:main2b} below). In the Black\&Scholes
case $\CBS(\kappa, \sigma \sqrt{t}) \asymp
\exp(-\kappa^2 / (2 \sigma^2 t))$, hence by Definition~\ref{ch3:def:iv} it follows that
implied volatility and tail probability are linked by
\begin{equation} \label{ch3:eq:sigmaimpuseful}
	\sigma_\imp(\kappa,t) \sim \frac{\kappa}{\sqrt{2 t \,(- \log \P(X_t > \kappa))}} \,.
\end{equation}
This relation shows that $\sigma_\imp(\kappa,t)$ stays bounded as $t \downarrow 0$
when $-\log \P(X_t > \kappa) \sim C/t$ for some 
$C = C(\kappa) \in (0,\infty)$, as in the Heston model \cite{cf:JFL12}. 

This is \emph{not} the case for our model, where
$-\log \P(X_t > \kappa) \ll 1/t$. In fact, by \eqref{ch3:equivX},
$X_t =  W_{I_t}(1+o(1))$
as $t\downarrow 0$ hence $X_t \simeq W_{I_t}$
is approximately Gaussian with a random variance $I_t = \int_0^t \sigma_s^2 \, \dd s$,
which yields $\P(X_t > \kappa) \asymp \E[\exp(-\kappa^2/(2I_t))]$.
Although $\E[I_t] = O(t)$, the point is that $I_t$ can take with non negligible
probability \emph{atypically large values}, as large as $t^{D/(1-D)} \gg t$,
and this affects $\P(X_t > \kappa)$.
More precisely, by \eqref{ch3:eq:boundIt} below we can write
\begin{equation*}
	\P\Big(I_t > t^{\frac{D}{1-D}} \Big) \ge e^{-c \, t^{-\frac{D}{1-D} + o(1)}}
\end{equation*}
where $c \in (0,\infty)$ (the logarithm in \eqref{ch3:eq:boundIt} is absorbed in the $o(1)$ term).
This leads to 
\begin{equation*}
	\P(X_t > \kappa) \asymp \E\Big[ e^{-\frac{\kappa^2}{2I_t}} \Big]
	\ge e^{-\frac{\kappa^2}{2}  t^{-\frac{D}{1-D}}} \, \P\Big(I_t > t^{\frac{D}{1-D}} \Big)
	\ge e^{-c' \, t^{-\frac{D}{1-D} + o(1)}}
\end{equation*}
for some $c' \in (0,\infty)$.
Plugged into \eqref{ch3:eq:sigmaimpuseful}, this estimate gives the $t$-dependence
in \eqref{ch3:eq:smile}, apart from logarithmic factors
(we refer to relation \eqref{ch3:eq:Cramer2} below for a more precise estimate).

Atypically large values of $I_t$ are also the source of \emph{multi-scaling of moments}.
By $X_t \sim W_{I_t}$ as $t \to 0$, cf.\ \eqref{ch3:equivX}, we get
$\E[|X_t|^q] \sim c \, \E[|W_{I_t}|^q] = c \, \E[W_1^q] \, \E[(I_t)^{q/2}]$.
The typical values of $I_t$ are of order $t$, which would suggest
the usual diffusive scaling $\E[|X_t|^q] \sim (const.) \, t^{q/2}$. However,
since $I_t \ge \sfc^2 (t-\tau_1)^{2D} \ge \sfc^2 (t/2)^{2D} = c_1 \, t^{2D}$ 
when $\tau_1 \le t/2$
(see \eqref{ch3:eq:bebo} below),
\begin{equation*}
	\E[(I_t)^{q/2}] \ge c_1 \, \big( t^{2D} \big)^{q/2} \, \P(\tau_1 \le t/2) \ge
	c_2 \, t^{Dq+1} \,,
\end{equation*}
which yields the \emph{anomalous scaling}
$\E[|X_t|^q] \ge (const.) \, t^{Dq+1} \gg t^{q/2}$ for $q$ large enough
(more precisely $q > q^* := 1/(\frac{1}{2}-D)$).
We refer to \cite{cf:ACDP} for more details.

\medskip
\noindent
\textbf{On a ``universal'' asymptotic relation}.
In the regime when \eqref{ch3:eq:smile} holds, the implied volatility
$\sigma_\imp(\kappa,t)$ is asymptotically a function $f(\kappa/t)$ of just the ratio $(\kappa/t)$.
This feature appears to be shared
by \emph{different models without moment explosion}
(with the function $f(\cdot)$ depending on the model). For instance, in
Carr-Wu’s finite moment logstable model \cite{cf:CW04},
as shown in \cite[Theorem 3.1]{cf:CC},
\begin{equation*}
	\sigma_\imp(\kappa,t) \sim B_\alpha \bigg( \frac{\kappa}{t} \bigg)^{-\frac{2-\alpha}{2(\alpha-1)}}
	\ \ \ \text{for } \ \kappa \gg t^{1/\alpha} \,, 
\end{equation*}
where $B_\alpha$ is an explicit constant.
Another example is provided by Merton's jump diffusion model \cite{cf:M}
for which, extending \cite{cf:BF}, we showed in \cite[Theorem 3.4]{cf:CC} that
\begin{equation*} 
 \sigma_\imp^2(\kappa,t) \sim \frac{\delta}{2\sqrt{2}}
 \frac{\frac{\kappa}{t}}{\sqrt{\log \frac{\kappa}{t}}}
 \ \ \ \text{for } \ \kappa \gg \sqrt{\log\tfrac{1}{t}} \,.
\end{equation*}

To understand heuristically the source of this phenomenon, note that
$\sigma_\imp(\kappa,t) \sim f(\kappa/t)$ means in particular that
$\sigma_\imp(2\kappa, 2t) \sim \sigma_\imp(\kappa,t)$, which by
\eqref{ch3:eq:sigmaimpuseful} translates into
\begin{equation} \label{ch3:eq:heuri0}
	\P(X_{2t} > 2\kappa) \asymp
	\P(X_{t} > \kappa)^2 \,.
\end{equation}
If the log-price increments are \emph{approximately stationary}, in the sense that
$\P(X_{t} > \kappa) \asymp \P(X_{2t} - X_{t} > \kappa)$,
the previous relation can be rewritten more expressively as
\begin{equation} \label{ch3:eq:heuri}
	\P(X_{2t} > 2\kappa) \asymp
	\P(X_{t} > \kappa) \, \P(X_{2t} - X_t > \kappa) \,.
\end{equation}
This says, heuristically, that the most likely way to produce the event
$\{X_{2t} > 2\kappa\}$ is through the events $\{X_t > \kappa\}$
and $\{X_{2t}-X_t > \kappa\}$, which are approximately independent.

Relation \eqref{ch3:eq:heuri0}
holds indeed for our model, see \eqref{ch3:eq:Cramer2} below,
as well as for Carr-Wu and Merton models, 
in the regime when $\kappa$ 
is large enough, depending on $t$ (and on the model).
On the other hand, relation \eqref{ch3:eq:heuri0} 
typically \emph{fails} for models with moment explosion, such as the Heston model,
for which the implied volatility $\sigma_\imp(\kappa,t)$ is \emph{not} asymptotically
a function of just the ratio $(\kappa/t)$, cf.\ \cite[\S 3.3]{cf:CC}.

\medskip
\noindent
\textbf{Comparison with other models.} In the recent literature,
other models with continuous paths with a steep implied volatility 
smile close to maturity have been considered. 

In \cite{cf:GJR}, Guennoun, Jacquier and Roome
use large deviations techniques to compute
the asymptotic behaviour of the implied volatility for the fractional Heston model,
both close to and far from maturity.
As for our model, they prove that in the small maturity regime it is possible to obtain 
an implied volatility of order $t^{-\gamma}$ for every  $\gamma \in (0,\frac 12)$. 
A drawback of fractional volatility models is the fact that, due to the dependence 
on the past of the fractional Brownian motion, they are not Markov.
On the contrast, our model is Markovian, which can be an advantage for pricing purposes.

In \cite{cf:JR}, Jacquier and Roome suggest a simple generalization of the Black\&Scholes 
model, obtained by plugging in a random initial volatility instead of a deterministic one.  
In the special case when the initial volatility is distributed as the solution, 
at some time $\tau>0$, of the CEV stochastic differential equqtion 
$\dd Y_u=\xi \, Y_u^p \, \dd B_u$, 
they obtain the explicit asymptotic behavior of the implied volatility close to maturity,
which displays steepness of the smile.
An advantage of our model is that we do not need to introduce a random initial
volatility ---although we could also do it--- 
to produce the steepness of the smile for small maturity, but we
can work with a fixed initial volatility, as it is more customary.

\medskip
\noindent
\textbf{Further directions of research}.
The tail probability asymptotics in Theorem~\ref{ch3:th:tailprob} below
include the regime $t \to \infty$, which is however excluded for
the implied volatility asymptotics in Theorem~\ref{ch3:th:main} (and for
the option price asymptotics in
Theorem~\ref{ch3:lemma:calltgamma} below). This is because 
we rely on the approach in \cite{cf:CC}, recalled in \S\ref{ch3:sec:reminders} below, which
assumes that the maturity is bounded from above, but extension to unbounded
maturity are certainly possible with further work.
For general results in the
regime $t \to \infty$, we refer to~\cite{cf:T09}.

\smallskip

It should also be stressed that our model has a symmetric smile 
$\sigma_\imp(-\kappa,t) = \sigma_\imp(\kappa,t)$,
a limitation shared by 
all stochastic volatility
models with independent volatility (recall Remark~\ref{ch3:rem:shared}).
To produce an asymmetry, one should correlate the volatility with
the price (leverage effect). In the framework of our model, 
this can be obtained e.g.\ introducing \emph{jumps in the price} correlated to
those of the volatility. This possibility is investigated in \cite{cf:C}.

\section{Main result: tail probability and option price}
\label{ch3:sec:pricetail}

In this section we present explicit asymptotic estimates
for the option price $c(\kappa,t)$
and for the tail probability $\P(X_t > \kappa)$ of our model.
Before starting, we note that 
the following convergence in distribution follows
from relations \eqref{ch3:equivX} and \eqref{ch3:eq:It}
(see \S\ref{ch3:eq:gammatsec}):
\begin{equation} \label{ch3:eq:gammat}
	\frac{X_t}{\sqrt{t}} \, \xrightarrow[t \downarrow 0]{d} \,
	\sigma_0 \, W_1 \,,
\end{equation}
where $\sigma_0$ is the constant in \eqref{ch3:eq:sigma0}

\subsection{Tail probability}
For families of $(\kappa,t)$ satisfying \eqref{ch3:eq:assfamily},
we distinguish the regime of \emph{typical deviations},
when $\P(X_t > \kappa)$ is bounded away from zero,
from the regime of \emph{atypical deviations},
when $\P(X_t > \kappa) \to 0$. The former regime corresponds
to $t\to 0$ with $\kappa= O(\sqrt{t})$ and the (strictly positive) limit
of $\P(X_t > \kappa)$ can be easily computed, by \eqref{ch3:eq:gammat}.

On the other hand, the regime of atypical deviations
$\P(X_t > \kappa) \to 0$ includes $t\to 0$ with $\kappa \gg \sqrt{t}$
and $t \to \bar t \in (0,\infty)$ with $\kappa \to \infty$,
and also $t\to\infty$ with $\kappa\gg t$ (not
included in \eqref{ch3:eq:assfamily}).
In all these cases we determine 
an asymptotic equivalent of $\log \P(X_t > \kappa)$
which, remarkably, is sharp enough to get the estimates
on the implied volatility in Theorem~\ref{ch3:th:main}.
We refer to \S\ref{ch3:sec:reminders}-\S\ref{ch3:sec:reminders2} for more details,
where we summarize the general results of \cite{cf:CC} linking tail probability,
option price and implied volatility.

\smallskip

The following theorem, on the asymptotic behavior
of $\P(X_t > \kappa)$, is proved in Section~\ref{ch3:sec:LDP}.
Note that items \eqref{tp:a}, \eqref{tp:b} and \eqref{tp:c} correspond to atypical deviations,
while the last item \eqref{tp:d} corresponds to typical deviations.
We recall that $\bkappa_1(\cdot)$ and $\bkappa_2(\cdot)$
are defined in \eqref{ch3:eq:gh}.

\begin{theorem}[Tail probability]\label{ch3:th:tailprob}
Consider a family of values of $(\kappa,t)$ with $\kappa \ge 0$, $t > 0$.
\begin{aenumerate}
\item\label{tp:a} If $\,t \to \infty$ and $\,\kappa \gg \sfc^{1/D} \, t$, or if
$\,t \to \bar t \in (0,\infty)$ and $\,\kappa \to \infty$,
or if $\,t\to 0$ and $\,\kappa \gg \bkappa_2(\sfc^{1/D} \, t)$,
\begin{equation} \label{ch3:eq:Cramer2}
	\rule{0pt}{1.9em}
	\log\P(X_t > \kappa) \sim  
	- \sfC \left( \frac{\kappa}{ \sfc \, t^D}\right)^{\frac{1}{1-D}}
	\left(\log \frac{\kappa}{ \sfc^{1/D} \, t}\right)^{\frac{1/2-D}{1-D}}  \,,
\end{equation}
where the constant $\sfC$ is defined in \eqref{ch3:eq:asf}.

\item\label{tp:b} If $\,t\to 0$ and $\,\sqrt{2}\,\bkappa_1(\sigma_0^2 \,t) <
\kappa \le M \, \bkappa_2( \sfc^{1/D} \,t)$, for some $\,M < \infty$,
\begin{align}\label{ch3:eq:gentail}
	\log\P(X_t > \kappa) \sim 
	- \sff \Bigg(\frac{\kappa}{\bkappa_2( \sfc^{1/D} \,t)}
	\, \sqrt{\frac{\log (\sfc^{1/D} \, t)}{\log (\lambda \, t)} } \Bigg) 
	\log \frac{1}{\lambda \,t}  \,,
\end{align}
where $\sff(\cdot)$ is defined in \eqref{ch3:eq:f}.

\item\label{tp:c} 
If $\,t\to 0$ and $\,\sqrt{t} \ll \kappa \le \sqrt{2} \, \bkappa_1(\sigma_0^2 \, t)$,
\begin{equation} \label{ch3:eq:withoutfa2}
	\log\P(X_t > \kappa) \sim - \frac{\kappa^2}{2\,\sigma_0^2 \, t}
	\sim - \frac{1}{2} \bigg( \frac{\kappa}{\bkappa_1( \sigma_0^2 \, t)}\bigg)^2
	\, \log\frac{1}{ \sigma_0^2 \, t} \,.
\end{equation}

\item\label{tp:d} Finally, if $\, t \to 0$ and $\,\kappa \sim a \sqrt{ \sigma_0^2 \, t}$
for some $a \in [0,\infty)$,
\begin{equation}\label{ch3:eq:boh}
	\P(X_t > \kappa) \to 1 - \Phi\big(a\big) > 0 \,,
\end{equation}
where $\Phi(\cdot)$ is the distribution function of a
standard Gaussian, cf.\ \eqref{ch3:eq:BS2}.
\end{aenumerate}
\end{theorem}

\begin{remark}\rm\label{ch3:rem:tpbexplicit}
Observe that item \eqref{tp:b} in Theorem~\ref{ch3:th:tailprob}
can be made more explicit:
\begin{itemize}
\item if $\,t\to 0$ and $\,\kappa \sim a\, \bkappa_2(\sfc^{1/D} \,t)$, 
for some $\,a \in (0,\infty)$,
\begin{equation} \label{ch3:eq:withfa}
	\log\P(X_t > \kappa) \sim - \sff( \tilde a) \, \log \frac{1}{ \lambda \,t} \,,
	\qquad \text{where} \qquad
	\tilde a := a \, \sqrt{\frac{\log (\sfc^{1/D} \, t)}{\log (\lambda \, t)} } \,;
\end{equation}

\item
if $\,t\to 0$ and $\sqrt{2} \, \bkappa_1(\sigma_0^2 \, t) 
< \kappa \ll \bkappa_2(\sfc^{1/D} \, t)$,
\begin{equation} \label{ch3:eq:withoutfa1}
	\log\P(X_t > \kappa) \sim - \log \frac{1}{ \lambda \, t} \,,
\end{equation}
because $\sff(0) = 1$ by \eqref{ch3:eq:asf}.
\end{itemize}
\end{remark}

\subsection{Option price}
We finally turn to the option price $c(\kappa,t)$.
As we discuss in \S\ref{ch3:sec:reminders2}, 
sharp estimates on the implied volatility, such as in Theorem~\ref{ch3:th:main},
can be derived from the asymptotic behavior of $\log c(\kappa,t)$
if $\kappa$ is bounded away from zero, or from
the asymptotic behavior of $\log (c(\kappa,t)/\kappa)$
if $\kappa \to 0$. For this reason, in 
the next theorem (proved in Section~\ref{ch3:sec:lemma:calltgamma})
we give the asymptotic behavior of $\log c(\kappa,t)$
and $\log ( c(\kappa,t) / \kappa )$, expressed
in terms of the tail probability $\P(X_t > \kappa)$
(whose asymptotic behavior can be read from Theorem~\ref{ch3:th:tailprob}).

\begin{theorem}[Option price]
\label{ch3:lemma:calltgamma}
Consider a family of values of $(\kappa,t)$ with $\kappa \ge 0$, $t > 0$.

\begin{aenumerate}
\item\label{op:a} If $\,t \to \bar t \in (0,\infty)$ and $\,\kappa \to \infty$,
or if $\,t \to 0$ and $\,\kappa \to \bar\kappa \in (0,\infty]$,
\begin{equation}\label{ch3:eq:casinfty}
	\log c(\kappa,t) \sim \log \P(X_t > \kappa)  \,.
\end{equation}

\item\label{op:b} If $\,t \to 0$ and $\,\kappa \to 0$ with 
$\,\kappa \gg \sqrt{ \sigma_0^2 \, t}$,
excluding the ``anomalous regime'' of next item,
\begin{equation}\label{ch3:eq:cas0good}
	\log \big( c(\kappa,t)/\kappa \big) \sim \log \P(X_t > \kappa)  \,.
\end{equation}

\item\label{op:c} If $\,t \to 0$ and $\,\sqrt{2D+1} \, \bkappa_1( \sigma_0^2 \, t) 
\leq \kappa \ll \bkappa_2( \sfc^{1/D} \, t)$
(``anomalous regime''),
\begin{equation}\label{ch3:eq:cas0bad2}
	\log \big( c(\kappa,t)/\kappa \big)
	\sim \log \frac{\bkappa_2(\sfc^{1/D} \, t)}{\kappa}
	- \log \frac{1}{\lambda \, t} \,.
\end{equation}

\item\label{op:d} If $\,t \to 0$ and $\,\kappa \sim a \sqrt{ \sigma_0^2 \, t}$
for some $a \in (0,\infty)$,
\begin{equation}\label{ch3:eq:typ}
	\frac{c(\kappa, t)}{\kappa} \to
	D( a ) \,, \qquad
	\text{with} \qquad
	D(x) := \frac{\phi(x)}{x} - \Phi(-x) \,,
\end{equation}
where $\phi(\cdot)$ and $\Phi(\cdot)$ are the density and distribution
function of a standard Gaussian.

\item\label{op:e} Finally, if $\,t \to 0$ and $\,\kappa \ll \sqrt{\sigma_0^2 \, t}$
(including $\kappa = 0$),
\begin{equation}\label{ch3:eq:typlast}
	c(\kappa, t) \sim \frac{ \sigma_0}{\sqrt{2\pi}} \, \sqrt{ t} \, .
\end{equation}
\end{aenumerate}
\end{theorem}

\section{Key large deviations estimates}

\label{ch3:sec:ld}

In this section we prove the following crucial estimate on the
exponential moments of the time-change process $I_t$,
defined in \eqref{ch3:eq:It}. As we show in the next Section~\ref{ch3:sec:LDP},
this will be the key to the proof of relation \eqref{ch3:eq:smile}
in Theorem~\ref{ch3:th:main}.
We recall that $\sfc$ is defined in \eqref{ch3:eq:sigmat}.

\begin{proposition} \label{ch3:th:crucial}
Fix a family of values of $(b,t)$ with $b > 0$, $t > 0$ such that
\begin{equation} \label{ch3:eq:assat}
	\text{either} \, \ t \to \bar t \in (0,\infty] \text{ and } \ b \to \infty \,,
	\quad \ \ \text{or} \ \, t \to 0 \, \text{ and } \
	\frac{b}{\frac{1}{t^{2D}} \log \frac{1}{t}}  \to \infty \,.
\end{equation}
Then the following asymptotic relation holds:
\begin{align} \label{ch3:eq:keygoal0}
	\log \E[e^{b I_t}] \sim 
	\tilde C\, \big(\sfc^{1/D} \, t \big) \, b^{\frac{1}{2D}}(\log b)^{\frac{2D-1}{2D}}\,, 
	\qquad \text{with} \qquad \tilde C
	= (2D)^{\frac{1}{2D}}(1-2D)^{\frac{1-2D}{2D}} \,.
\end{align}
\end{proposition}

From this one can easily derive Large Deviations estimates
 for the right tail of $I_t$.
 
\begin{corollary}\label{ch3:coro:LDP I}
Consider a family of values of $(\kappa, t)$ with $\kappa > 0$, $t > 0$ such that
\begin{equation} \label{ch3:eq:added}
\begin{cases}
\text{either $t \to 0$ and $\kappa \gg t^{2D}$, }\\
\text{or $t \to \bar t \in (0,\infty)$ and $\kappa \to \infty$,}\\
\text{or $t \to \infty$ and $\kappa \gg t$.}
\end{cases}
\end{equation}
Then the following relation holds:
\begin{equation} \label{ch3:eq:boundIt}
	\log\P(I_t>\kappa) \sim - \frac{1}{1-2D}
	\left(\frac{\kappa}{(\sfc^{1/D}  \,t)^{2D}}
	\right)^{\frac{1}{1-2D}}\left(\log \frac{\kappa }{\sfc^{1/D} \, t}\right)
	\,. 
\end{equation}
\end{corollary}

\begin{remark}\label{ch3:rem:LD}\rm
Recalling \eqref{ch3:eq:It}, the time-change process $I_t$ can be seen as
a natural additive functional of the inter-arrival times
$\tau_k - \tau_{k-1}$ of a Poisson process:
\begin{equation} \label{ch3:eq:concave}
	I_t = g(t-\tau_{N_t}) - g(-\tau_0) + \sum_{k=1}^{N_t} g(\tau_k - \tau_{k-1})  \,,
\end{equation}
with the choice $g(x) := \sfc^2 x^{2D}$.
Remarkably, Proposition~\ref{ch3:th:crucial} 
and Corollary~\ref{ch3:coro:LDP I} continue to hold
for a \emph{wide class of functions $g(\cdot)$}, as it is clear from the proofs:
what really matters is the asymptotic behavior $g(x) \sim \sfc^2 x^{2D}$
as $x \downarrow 0$.
(Also note that the value of $\tau_0$ in \eqref{ch3:eq:concave} plays no role
in Proposition~\ref{ch3:th:crucial} 
and Corollary~\ref{ch3:coro:LDP I}, so one can set $\tau_0 = 0$.)
\end{remark}

\begin{proof}[Proof of Corollary~\ref{ch3:coro:LDP I} (sketch).]
The proof is completely analogous to that of Theorem~\ref{ch3:th:LDP+} in Section~\ref{ch3:sec:LDP}, 
to which we refer for more details. Let us set
\begin{equation*} 
 \gamma_{\kappa,t}= \left(\frac{\kappa}{(\sfc^{1/D}  \,t)^{2D}}\right)^{\frac{1}{1-2D}}
 \left(\log \frac{\kappa }{\sfc^{1/D} \,t}\right)  \,, \quad \text{so that} \quad
 \frac{\gamma_{\kappa,t}}{\kappa} = 
 \left(\frac{\kappa}{\sfc^{1/D} \,t}\right)^{\frac{2D}{1-2D}}
 \left(\log \frac{\kappa }{\sfc^{1/D} \,t}\right)\,.
\end{equation*}
By \eqref{ch3:eq:added}, the family
$(t,b)$ with $b := \frac{\gamma_{\kappa,t}}{\kappa}$
satisfies \eqref{ch3:eq:assat}. Then \eqref{ch3:eq:keygoal0} yields, for $\alpha \ge 0$,
\[
 \log \E \left( \exp\left(\alpha \, \gamma_{\kappa, t} \, \frac{I_t}{\kappa}
\right) \right) \sim \Lambda(\alpha) \, \gamma_{\kappa, t} \,,
\qquad \text{where} \qquad
 \Lambda(\alpha) := \tilde C \left(\frac{1-2D}{2D}\right)^{\frac{1-2D}{2D}}\alpha ^{\frac{1}{2D}}\,,
\]
with $\tilde C$ defined in \eqref{ch3:eq:keygoal0}. 
By the G\"artner-Ellis Theorem \cite[Theorem 2.3.6]{cf:DZ},\footnote{In principle one
should compute $\Lambda(\alpha)$ for all $\alpha\in\R$ in order to apply the
G\"artner-Ellis Theorem, which yields a full Large Deviations Principle. However,
being interested in the right-tail behavior, cf.\ \eqref{ch3:eq:boundIt}, it is enough
to focus on $\alpha \ge 0$, as it is clear from the proof in \cite[Theorem 2.3.6]{cf:DZ}.}
we get
\begin{equation} \label{ch3:eq:GEmod}
	\log \P\bigg( \frac{I_t}{\kappa} > x \bigg) \sim -\gamma_{\kappa,t} \, I(x) \,,
\end{equation}
where $I(\cdot)$ is the Fenchel-Legendre transform of $\Lambda(\cdot)$, i.e. (for $x \ge 0$)
\[
\begin{split}
 I(x) & := \sup_{\alpha \in \R} \big\{ \alpha x - \Lambda(\alpha) \big\} 
 = \big( \bar\alpha x - \Lambda(\bar\alpha) \big) \Big|_{\bar \alpha = 
 \frac{2D}{1-2D} \left(\frac{2D}{\tilde C} x \right)^{\frac{2D}{1-2D}}} \\
 & = \frac{2D}{1-2D} \Bigg[ \left(\frac{2D}{\tilde C} \right)^{\frac{2D}{1-2D}}
 - \tilde C \left(\frac{2D}{\tilde C} \right)^{\frac{1}{1-2D}} \Bigg]
 x^{\frac{1}{1-2D}} 
= \left(\frac{2D \, x}{\tilde C^{2D}}\right)^{\frac{1}{1-2D}}
 = \frac{ x^{\frac{1}{1-2D}} }{1-2D} \,.
\end{split}
\]
Setting $x = 1$ in \eqref{ch3:eq:GEmod} yields \eqref{ch3:eq:boundIt}.
\end{proof}

\subsection{Preliminary results}
\label{ch3:prelsec}

We start with a useful upper bound on $I_t$
(defined in \eqref{ch3:eq:It}).

\begin{lemma} \label{ch3:th:usefulub}
For all $t \ge 0$ the following upper bound holds:
\begin{equation}\label{ch3:eq:eqItupfin}
\begin{split}
	 I_t & \le \sigma_0^2 \, t \,+\, \sfc^2 \, N_t^{1-2D}\, t^{2D} \,,
\end{split}
\end{equation}
where the constants $\sigma_0$ and $\sfc$ are defined in
\eqref{ch3:eq:sigma0} and \eqref{ch3:eq:sigmat}.
\end{lemma}

\begin{proof}
Since $(a+b)^{2D} - b^{2D} \le 2D \, b^{2D-1} \, a$ for all $a,b > 0$
by concavity (recall that $D < \frac{1}{2}$), on the event $\{N_t = 0\}$ we can write,
recalling \eqref{ch3:eq:It} and \eqref{ch3:eq:sigma0},
\begin{equation} \label{ch3:eq:Nt0}
	I_t = \sfc^2 \big\{ (t-\tau_0)^{2D} - (-\tau_0)^{2D}\big\} \le
	\sfc^2 \, 2D \, (-\tau_0)^{2D-1} \, t \, = \sigma_0^2 \, t \,,
\end{equation}
proving \eqref{ch3:eq:eqItupfin}.
Analogously, on the event $\{N_t \ge 1\} = \{0 \le \tau_1 \le t\}$ we have
\begin{equation}\label{ch3:eq:Itup}
 \begin{split}
I_t& := \sfc^2 \Bigg\{ (\tau_1-\tau_{0})^{2D} - (-\tau_0)^{2D} +
	\sum_{k=2}^{N_t} (\tau_k - \tau_{k-1})^{2D}
	+ (t-\tau_{N_t})^{2D} \Bigg\} \\
	& \leq  \sfc^2 \Bigg\{  2D(-\tau_0)^{2D-1}t+
	\sum_{k=2}^{N_t} (\tau_k - \tau_{k-1})^{2D}
	+ (t-\tau_{N_t})^{2D} \Bigg\} \,.
 \end{split}
\end{equation}
For all $\ell\in\N$ and $x_1, \ldots, x_\ell \in \R$,
it follows by H\"older's inequality with $p := \frac{1}{2D}$ that 
\begin{equation} \label{ch3:eq:Hold}
	\sum_{k=1}^\ell x_k^{2D} \leq 
	\Bigg( \sum_{k=1}^\ell (x_k^{2D})^p \Bigg)^{\frac{1}{p}}
	\Bigg( \sum_{k=1}^\ell 1 \Bigg)^{1-\frac{1}{p}}
	= \Bigg( \sum_{k=1}^\ell x_k \Bigg)^{2D} \ell^{1-2D} \,.
\end{equation}
Choosing $\ell = N_t$
and $x_1 = \tau_2-\tau_1$, $x_k = (\tau_{k+1} - \tau_{k})$ 
for $2 \le k \le \ell-1$
and $x_\ell = (t-\tau_{\ell-1})$,
since $\sum_{k=1}^\ell x_k = t - \tau_1 \le t$, we get from \eqref{ch3:eq:Itup}
\begin{equation*}
	I_t\leq \sfc^2 \left(2D(-\tau_0)^{2D-1}t+N_t^{1-2D}t^{2D}\right) =
	\sigma_0^2 \, t \,+\, \sfc^2 N_t^{1-2D}t^{2D} \,,
\end{equation*}
completing the proof of \eqref{ch3:eq:eqItupfin}.
\end{proof}

We now link the exponential moments of $I_t$ to those of the log-price $X_t$.

\begin{lemma}[No moment explosion]\label{ch3:th:noexpl}
For every $t \in [0,\infty)$ and $p \in 	\R$ one has
\begin{equation}\label{ch3:eq:noexpl}
	\E\big[e^{p X_t}\big] = \E\big[ e^{\frac{1}{2}p(p-1)I_t} \big] < \infty \,.
\end{equation}
\end{lemma}

\begin{proof}
By the definition \eqref{ch3:equivX} of $X_t$, the independence
of $I$ and $W$ gives
\begin{equation*}
	\E\big[e^{p X_t}\big] = \E\big[e^{p (W_{I_t} - \frac{1}{2}I_t)}\big] 
	= \E\big[e^{p (\sqrt{I_t}W_{1} - \frac{1}{2}I_t)}\big]
	= \E\big[e^{\frac{1}{2} (p \sqrt{I_t})^2 - \frac{1}{2}p I_t}\big]
	= \E\big[ e^{\frac{1}{2}p(p-1)I_t} \big] \,,
\end{equation*}
which proves the equality in \eqref{ch3:eq:noexpl}.
Applying the upper bound \eqref{ch3:eq:eqItupfin} yields
\begin{equation*}
	\E\big[ e^{\frac{1}{2}p(p-1)I_t} \big] \le
	\E\big[ e^{\frac{1}{2}p(p-1)(\sigma_0^2 \, t \,+\, \sfc^2 N_t^{1-2D}t^{2D})} \big]
	= \E\big[ e^{c_1 t \,+\, c_2 \, t^{2D} \, N_t^{1-2D}} \big] 
	\le \E\big[ e^{c_1 t \,+\, c_2 \, t^{2D} \, N_t} \big] \,,
\end{equation*}
for suitable $c_1, c_2 \in (0,\infty)$ depending on $p$ and on the
parameters of the model.
The right hand side is finite because $N_t \sim Pois(\lambda t)$
has finite exponential moments of all orders.
\end{proof}

\subsection{Proof of Proposition~\ref{ch3:th:crucial}}

Let us set
\begin{equation}\label{ch3:eq:Btb}
 B_{t,b}= \big( \sfc^{1/D} \, t \big) \, b^{\frac{1}{2D}} \,
 (\log b)^{\frac{2D-1}{2D}} \,.
\end{equation}
We are going to show that \eqref{ch3:eq:keygoal0} holds by proving
separately upper and lower bounds, i.e.\
\begin{equation}\label{ch3:ubound}
	\limsup \frac{1}{B_{t,b}} 
	\log \E[e^{ b I_t}] \le \tilde C \,, \qquad
	\liminf \frac{1}{B_{t,b}} 
	\log \E[e^{ b I_t}] \ge \tilde C \,.
\end{equation}
We start with the upper bound and we split the proof in steps.

\medskip
\noindent
\emph{Step 1. Preliminary upper bound.}
The upper bound \eqref{ch3:eq:eqItupfin} on $I_t$ yields
\begin{equation*}
\begin{split}
 \E\big[e^{bI_t}\big] &= 
 \sum_{j=0}^{\infty}
 \E[e^{bI_t}|N_t=j ] \, \P (N_t=j) \le
e^{\sigma_0^2 \, t b} \sum_{j=0}^{\infty} 
e^{\sfc^2 t^{2D} b \, j^{1-2D}}
 e^{-\lambda t}\frac{(\lambda t)^j}{j!} \,.
\end{split}
\end{equation*}
Since $j! \sim j^j e^{-j} \sqrt{2\pi j}$ as $j\uparrow \infty$,
there is $c_1 \in (0,\infty)$ such that $j! \ge \frac{1}{c_1}
j^j e^{-j}$ for all $j\in\N_0$. Bounding $e^{-\lambda t} \le 1$, we obtain
\begin{equation}\label{ch3:eq:acav}
	\E\big[e^{bI_t}\big] \le
	c_1 \, e^{\sigma_0^2 \, t b} \sum_{j=0}^{\infty} 
	e^{\sfc^2 t^{2D} b \, j^{1-2D}}
	\frac{(\lambda t)^j}{j^j e^{-j}} 
	= c_1 \, e^{\sigma_0^2 \, t b} \sum_{j=0}^{\infty} 
	e^{f(j)} \,,
\end{equation}
where for $x \in [0,\infty)$ we set
\begin{equation}\label{ch3:eq: eqdefnf}
	f(x)= f_{t,b}(x) := \sfc^2 \left(t^{2D}b\right) x^{1-2D} - x \bigg( \log \frac{x}{\lambda t} - 1 \bigg) \,,
\end{equation}
with the convention $0\log 0=0$. Note that
\begin{equation} \label{ch3:eq:f'}
	f'(x) = (1-2D) \sfc^2 b \, \bigg( \frac{x}{t} \bigg)^{-2D} 
	- \log \bigg(\frac{x}{t}\bigg) + \log \lambda \,,
\end{equation}
hence $f'(x)$ is continuous and strictly decreasing on $(0,\infty)$, with 
$\lim_{x \downarrow 0} f'(x) = +\infty$ and
$\lim_{x \uparrow \infty }f'(x) = -\infty$. As a consequence, there is a unique
$\bar x_{t,b} \in (0,\infty)$ with $f'(\bar x_{t,b}) = 0$ and the function $f(x)$ attains
its global maximum on $[0,\infty)$ at the point $x=\bar x_{t,b}$:
\begin{equation}\label{ch3:eq:maxf}
	\max_{x\in [0,\infty)} f(x) = f(\bar x_{t,b}) \,.
\end{equation}

We are going to show that 
the leading contribution to the sum in \eqref{ch3:eq:acav}
is given by a \emph{single term} $e^{f(j)}$, for $j \asymp \bar x_{t,b}$.
We first need asymptotic estimates on $\bar x_{t,b}$ and $f(\bar x_{t,b})$.

\medskip
\noindent
\emph{Step 2. Estimates on $\bar x_{t,b}$ and $f(\bar x_{t,b})$.}
We first prove that
\begin{equation} \label{ch3:eq:xinfi}
	\bar x_{t,b} \to \infty \,, \qquad
	\frac{\bar x_{t,b}}{t} \to \infty \,,
\end{equation}
by showing that for any fixed $M \in (0,\infty)$ one has 
$\bar x_{t,b} > M$ and $\bar x_{t,b}/t > M$ eventually.
Since $b\to\infty$ by assumption \eqref{ch3:eq:assat},
uniformly for $x$ such that $(x/t) \in [0,M]$ we have
\begin{equation*}
\begin{split}
	f'(x) & \ge (1-2D) \sfc^2 b \, M^{-2D} 
	- \log M + \log \lambda =: C_1 b + C_2 \to \infty \,.
\end{split}
\end{equation*}
Recalling that $\bar x_{t,b}$ is the solution of $f'(x) = 0$,
it follows that $(\bar x_{t,b} / t) > M$ eventually. Likewise,
uniformly for $x$ such that $x \in [0,M]$, by assumption \eqref{ch3:eq:assat} we can write
\begin{equation*}
\begin{split}
	f'(x) & \ge (1-2D) \sfc^2 b \, \bigg( \frac{M}{t} \bigg)^{-2D} 
	- \log \bigg(\frac{M}{t}\bigg) + \log \lambda 
	=: C_1 \, t^{2D} \, b - \log \frac{1}{t} + C_2 \to \infty \,,
\end{split}
\end{equation*}
hence $\bar x_{t,b} > M$ eventually, completing the proof of \eqref{ch3:eq:xinfi}.

Next we prove that $\bar x_{t,b}$ has the following asymptotic behavior:
\begin{equation} \label{eq:xbart}
	\bar x_{t,b} \sim \big( 2D(1-2D) \sfc^2 \big)^{\frac{1}{2D}} \, \bigg( 
	\frac{t^{2D}\,b}
	{\log b} \bigg)^{\frac{1}{2D}} \,,
\end{equation}
arguing as follows.
Recalling \eqref{ch3:eq:f'}, the equation $f'(\bar x_{t,b}) = 0$ can be rewritten as
\begin{equation} \label{eq:transc}
	\frac{\bar x_{t,b}}{t} 
	= \Bigg( \frac{(1-2D) \sfc^2 b}{\log \frac{\bar x_{t,b}}{t} - \log \lambda} 
	\Bigg)^{\frac{1}{2D}} 
	\sim \Bigg( \frac{(1-2D) \sfc^2 \, b}{\log \frac{\bar x_{t,b}}{t}} \Bigg)^{\frac{1}{2D}} \,,
\end{equation}
because $\bar x_{t,b} / t \to \infty$ by \eqref{ch3:eq:xinfi}.
Inverting \eqref{eq:transc} and using again $\bar x_{t,b} / t \to \infty$ gives
\begin{equation} \label{eq:tool}
	\log \frac{\bar x_{t,b}}{t} \sim
	(1-2D) \sfc^2 \, b \, \bigg( \frac{\bar x_{t,b}}{t} \bigg)^{-2D} = o(b) \,,
\end{equation}
and we recall that $b \to \infty$ by assumption \eqref{ch3:eq:assat}.
Taking $\log$ in \eqref{eq:transc} gives
\begin{equation} \label{ch3:eq:pplug}
	\log \frac{\bar x_{t,b}}{t} \sim \frac{1}{2D} \big\{ \log[(1-2D) \sfc^2] + \log b
	- \log \big( \log \tfrac{\bar x_{t,b}}{t} \big) \big\}
	\sim \frac{1}{2D}  \log b \,,
\end{equation}
having used \eqref{eq:tool}. Plugging \eqref{ch3:eq:pplug}
into \eqref{eq:transc} gives precisely \eqref{eq:xbart}.

Looking back at \eqref{ch3:eq: eqdefnf}, we obtain
the asymptotic behavior of $f(\bar x_{t,b})$:
by \eqref{ch3:eq:xinfi} and \eqref{eq:tool}
\begin{equation} \label{eq:estmax}
\begin{split}
	f(\bar x_{t,b}) & = \sfc^2 \, (t^{2D} b) \, \bar x_{t,b}^{1-2D} -
	\bar x_{t,b} \, \log \frac{\bar x_{t,b}}{t} \big(1+o(1)\big)  \\
	& = \sfc^2 \, (t^{2D} b) \bar x_{t,b}^{1-2D} - 
	\bar x_{t,b} \frac{(1-2D) \sfc^2 \,(t^{2D} b)}{\bar x_{t,b}^{2D}} \big(1+o(1)\big) \\
	& = 2D \, \sfc^2 \, (t^{2D} b) \, \bar x_{t,b}^{1-2D}  \big(1+o(1)\big) \,,
\end{split}
\end{equation}
hence applying \eqref{eq:xbart},
and recalling the definition of $B_{t,b}$ and $\tilde C$ in \eqref{ch3:eq:Btb}
and \eqref{ch3:eq:keygoal0},
\begin{equation} \label{eq:estmax2}
	f(\bar x_{t,b}) \sim (2D)^{\frac{1}{2D}} \, (1-2D)^{\frac{1}{2D}-1} \, \sfc^{\frac{1}{D}}
	\frac{t \, b^{\frac{1}{2D}}}
	{(\log b)^{\frac{1}{2D}-1}}
	= \tilde C \, B_{t,b} \,.
\end{equation}

\medskip
\noindent
\emph{Step 3. Completing the upper bound.}
We can finally come back to \eqref{ch3:eq:acav}.
Henceforth we set $\bar x := \bar x_{t,b}$ to lighten notation.
We control $f(x)$ for $x \ge 2 \bar x$
as follows: since $f'(\cdot)$ is strictly
decreasing, and $f(2\bar x) \le f(\bar x)$ by \eqref{ch3:eq:maxf},
\[
 f(x)=f(2\bar x)+\int_{2\bar x}^x f'(s) \dd s 
 \leq f(\bar x)+f'(2\bar x)(x-2\bar x) \,.
\]
Observe that $f'(2\bar x) = -|f'(2\bar x)| <0$, hence
\begin{equation}\label{ch3:eq:geeo}
 \begin{split}
\sum_{j \geq 2\bar x}e^{f(j)} &\leq e^{f(\bar x)} 
 \sum_{j \geq 2\bar x}e^{-|f'(2\bar x)|(j-2\bar x)}
= \frac{e^{f(\bar x)}}{1-e^{-|f'(2\bar x)|}}. 
 \end{split}
\end{equation}
By \eqref{ch3:eq:f'}, recalling that $f'(\bar x) = 0$, we can write
\begin{equation*}
	f'(2\bar x) = f'(2\bar x) - 2^{-2D} f'(\bar x)
	= 2^{-2D} \log \bigg( \frac{\bar x}{t} \bigg)
	- \log \bigg( \frac{2\bar x}{t} \bigg)
	+ (1-2^{-2D})\log \lambda
	\to -\infty \,,
\end{equation*}
because $\bar x / t \to \infty$ by \eqref{ch3:eq:xinfi}.
Then $1-e^{-|f'(2\bar x)|}> \frac{1}{2}$ eventually and
\eqref{ch3:eq:geeo} yields
\begin{equation}\label{ch3:eq:boundtail}
\begin{split}
\sum_{j \ge 2\bar x}e^{f(j)} 
\leq 2 \, e^{f(\bar x)} \,.
\end{split}
\end{equation}
The initial part of the sum can be simply bounded by
\begin{equation}\label{ch3:eq:boundiniz}
\sum_{0 \le j < 2\bar x}
 e^{f(j)} \leq (2 \bar x+1) \, e^{f(\bar x)} \,.
\end{equation}

Looking back at \eqref{ch3:eq:acav}, we can finally write
\begin{equation}\label{ch3:eq:limsup}
\begin{split}
 \log \E \big[ e^{bI_t} \big]
 &\leq \log c_1 +  \sigma_0^2\, b\,t + \log (2\bar x + 3) + f(\bar x) \,.
\end{split}
\end{equation}
Comparing \eqref{eq:xbart} and \eqref{eq:estmax2}, we see that
$\bar x = O (f(\bar x) / \log b) = o(f(\bar x))$, because $b \to \infty$,
hence $\log (2\bar x + 3) = o(\bar x) = o(f(\bar x))$.
Again by \eqref{eq:estmax2}
we have $bt = o(\bar x) = o(f(\bar x))$, because $D <\frac{1}{2}$.
This means that the first three terms in the right hand side of \eqref{ch3:eq:limsup}
are negligible compared to $f(\bar x)$, and since $f(\bar x) \sim \tilde C \, B_{t,b}$ by
relation \eqref{eq:estmax}, we obtain
\[
 \limsup \frac{1}{B_{t,b}}
 \log\E\big[e^{bI_t}\big]\leq \tilde C \,,
\]
proving the desired upper bound in \eqref{ch3:ubound}.

\medskip
\noindent
\emph{Step 4. Lower bound.}
By \eqref{ch3:eq:It},
since $(\tau_1 - \tau_0)^{2D} \ge (-\tau_0)^{2D}$, we have
the following lower bound on $I_t$ on the event $\{N_t \ge 1\}$:
\begin{equation} \label{ch3:eq:bebo}
      	I_t \ge \sfc^2
	\Bigg\{ (t-\tau_{N_t})^{2D} +
	\sum_{k=2}^{N_t} (\tau_k - \tau_{k-1})^{2D} \Bigg\} \,.
\end{equation}
To match the upper bound, note that
H\"older's inequality \eqref{ch3:eq:Hold} becomes an equality when 
all the terms $x_k = \tau_k - \tau_{k-1}$ are equal. We can make this
approximately true 
introducing for $m\in\N$ and $\varepsilon \in (0,1)$ the event $A_m$ defined by
\begin{equation} \label{ch3:eq:Am}
	A_m := \bigg\{ \tau_1 < \varepsilon\frac{t}{m} \bigg\}
        \cap \bigcap_{i=2}^{m} 
       \bigg\{[(i-1) - \varepsilon]\frac{t}{m} < \tau_i
       <[(i-1) + \varepsilon]\frac{t}{m} \bigg\}
       \cap \big\{ \tau_{m+1} > t \big\} \,,
\end{equation}
which ensures that $N_t = m$ and $\tau_k - \tau_{k-1} \geq (1- 2\varepsilon) \tfrac{t}{m}$
for $2 \le k \le m$ and $t - \tau_{m} \geq (1-2\varepsilon) \tfrac{t}{m}$.
In particular, recalling \eqref{ch3:eq:bebo},
on the event $A_m$ we have the lower bound
\begin{equation} \label{ch3:eq:lbIXK1}
	I_t \geq \sfc^2 \, m \big((1-2\varepsilon) 
	\tfrac{t}{m} \big)^{2D}
	= (1-2\varepsilon)^{2D}
	 \, \sfc^2 \, m^{1-2D} t^{2D}\,.
\end{equation}
Since $\tau_1, \tau_2 - \tau_{1}, \tau_3 - \tau_2, \ldots$ are i.i.d.\ $Exp(\lambda)$
random variables, 
and on the event $A_m$ one has $\tau_k - \tau_{k-1} \leq (1+ 2\varepsilon) \tfrac{t}{m}$
for $2 \le k \le m$, and also $\tau_1 \leq (1+ 2\varepsilon) \tfrac{t}{m}$,
a direct estimate on the densities yields
\begin{equation} \label{ch3:eq:lbIXK2}
	\P(A_m) \ge \big( \lambda e^{-\lambda(1+2\varepsilon)
	 \frac{t}{m}} \big)^{m} 
	\big( \varepsilon \tfrac{t}{m} \big)^{m}  e^{-\lambda(1+2\varepsilon)
	 \frac{t}{m}}=
	e^{-(1+2\varepsilon)(1+\frac{1}{m})\lambda t}
	\frac{(\varepsilon \lambda t)^m}{m^m} \,,
\end{equation}
hence by \eqref{ch3:eq:lbIXK1}
\begin{equation}\label{ch3:eq:eqLBXKE}
 \E\big[ e^{bI_t} \big] \geq 
 \E\big[ e^{bI_t} \ind_{A_m} \big] \geq 
 e^{(1-2\epsilon)^{2D} \sfc^2 \, (t^{2D} b) \, m^{1-2D}}\P(A_m)
 \geq e^{\tilde f(m)}
\end{equation}
where we define $\tilde f(x)$, for $x\geq 0$ by
\begin{displaymath}
	\tilde f(x) = \tilde f_{t,b,\epsilon}(x) := (1-2\epsilon)^{2D}\, \sfc^2\, ( t^{2D}b) \, x^{1-2D}
	- x\log \frac{x}{\epsilon \lambda t} 
       -(1+2\varepsilon)(1+\tfrac{1}{m})\lambda t\,.
\end{displaymath}

Note that $\tilde f(x)$ resembles $f(x)$,
cf.\ \eqref{ch3:eq: eqdefnf}. Since the leading contribution to the upper bound was 
given by $e^{f(\bar x)}$, where $\bar x = \bar x_{b,t}$ is
the maximizer of $f(\cdot)$, it is natural to 
choose $m = \lfloor \bar x \rfloor$ in the lower bound \eqref{ch3:eq:eqLBXKE}. 
Since $\bar x \to \infty$ and $t \ll \bar x$,
cf.\ \eqref{ch3:eq:xinfi}, we have
\begin{equation*}
	\tilde f(\lfloor\bar x\rfloor) \sim. \tilde f(\bar x) \sim
	 (1-2\epsilon)^{2D}\, \sfc^2 \, (t^{2D} b) \, \bar x^{1-2D} -
	\bar x \, \log \frac{\bar x}{t} \big(1+o(1)\big)   \,,
\end{equation*}
and recalling \eqref{eq:estmax}-\eqref{eq:estmax2} we obtain
\begin{equation*}
	\tilde f(\lfloor\bar x\rfloor) \sim f(\bar x) -
      \big[1 - (1-2\varepsilon)^{2D}\big]
       \sfc^2 \, (t^{2D} b) \, \bar x^{1-2D}   \sim 
       \bigg[1 - \frac{1 - (1-2\varepsilon)^{2D}}{2D}\bigg]
       \tilde C \, B_{t,b} \,,
\end{equation*}
which coupled to \eqref{ch3:eq:eqLBXKE} yields
\[
 \liminf \frac{1}{B_{t,b}}\log \E\big[e^{bI_t}\big]
 \geq   \bigg[1 - \frac{1 - (1-2\varepsilon)^{2D}}{2D}\bigg] \, \tilde C \,.
\]
Letting $\varepsilon \to 0$ we obtain the
desired lower bound in \eqref{ch3:ubound},
completing the proof.\qed

\section{Proof of Theorem~\ref{ch3:th:tailprob} (tail probability)}
\label{ch3:sec:LDP}

In this section we prove relation \eqref{ch3:eq:gammat}
and Theorem~\ref{ch3:th:tailprob}.

\subsection{Proof of relation \eqref{ch3:eq:gammat} and of
Theorem~\ref{ch3:th:tailprob}, part \eqref{tp:d}.}
\label{ch3:eq:gammatsec}

For any $t \ge 0$,  by \eqref{ch3:equivX}
\begin{equation*}
	X_t \overset{d}{=} \sqrt{I_t} \, W_1 - \frac{1}{2} I_t \,.
\end{equation*}
Since $I_0 = 0$, a.s.\ 
one has $I_t/t = (I_t - I_0)/t \to I'_0 = \sigma_0^2$ as $t \downarrow 0$,
cf.\ \eqref{ch3:eq:It}-\eqref{ch3:eq:sigma0}. Then
\begin{equation*}
	\frac{X_t}{\sqrt{t}} 
	\overset{d}{=} \sqrt{\frac{I_t}{t}} \, W_1 - \frac{1}{2}
	\sqrt{t} \frac{I_t}{t} \xrightarrow[t\downarrow 0]{\text{a.s.}}
	\sigma_0 \, W_1 \,,
\end{equation*}
proving relation \eqref{ch3:eq:gammat}.
Relation \eqref{ch3:eq:boh} follows
from \eqref{ch3:eq:gammat}, proving
part \eqref{tp:d} in Theorem~\ref{ch3:th:tailprob}. 

\subsection{Proof of Theorem~\ref{ch3:th:tailprob}, part \eqref{tp:a}}

Recall the definition of $\bkappa_1(\cdot)$ and $\bkappa_2(\cdot)$ in \eqref{ch3:eq:gh}.
Let us fix a family of $(\kappa,t)$ with $\kappa > 0$, $t > 0$ as in item \eqref{tp:a}
of Theorem~\ref{ch3:th:tailprob}, i.e.\
\begin{equation}\label{ch3:eq:ktd}
	\begin{cases}
	\displaystyle
	\text{either } \
	t \to \infty \ \text{ and } \ \frac{\kappa}{t} \to \infty \,,\\
	\displaystyle
	\rule{0pt}{1.7em}
	\text{or } \
	t \to \bar t \in (0,\infty) \ \text{ and } \ \kappa \to\infty \,,\\
	\displaystyle
	\rule{0pt}{1.7em}
	\text{or } \
	t \to 0 \ \text{ and } \ \frac{\kappa}{t^D 
	\sqrt{\log \frac{1}{t}}} \to \infty \,.
	\end{cases}
\end{equation}
We are going to prove the following result, which is stronger than
our goal \eqref{ch3:eq:Cramer2}.

\begin{theorem}\label{ch3:th:LDP+}
For any family of values of
$(\kappa,t)$ satisfying \eqref{ch3:eq:ktd}, the random variables
$\frac{X_t}{\kappa}$ satisfy
the \emph{large deviations principle}
with rate $\alpha_{t,\kappa}$ and good rate function $I(\cdot)$ given by
\begin{equation}\label{ch3:eq:gtIx}
\begin{split}
 \alpha_{t,\kappa} & :=\left( \frac{\kappa}{\sfc \, t^D}\right)
 ^{\frac{1}{1-D}}\left(\log \frac{\kappa}{\sfc^{1/D} \,t}\right)
 ^{\frac{1/2-D}{1-D}} \,, \qquad
 I(x) := \sfC \, |x|^{\frac{1}{1-D}} 
\end{split}
\end{equation}
where $\sfC$ is defined in \eqref{ch3:eq:asf}.
This means that for every Borel set $A \subseteq \R$
\begin{equation*}
	-\inf_{x\in \mathring{A}} I(x) \le 
	\liminf \frac{1}{\alpha_{t,\kappa}} 
	\log \P\left(\frac{X_{t}}{\kappa} \in A\right)
	\le \limsup
	\frac{1}{\alpha_{t,\kappa}} \log \P\left(\frac{X_{t}}{\kappa} \in A\right) 
	\le -\inf_{x\in \overline{A}} I(x) \,,
\end{equation*}
where $\mathring{A}$ and $\overline{A}$ denote respectively
the interior and the closure of $A$. 
In particular, choosing $A = (1,\infty)$,
relation \eqref{ch3:eq:Cramer2} in Theorem~\ref{ch3:th:tailprob} holds.
\end{theorem}

\begin{proof}
We are going to show that, with $\alpha_{t,\kappa}$ as in \eqref{ch3:eq:gtIx},
the following limit exists for $\beta\in\R$:
\begin{equation} \label{ch3:eq:Lambda}
	\Lambda(\beta) := \lim
	\frac{1}{\alpha_{t,\kappa}} \log \E[e^{\beta \alpha_{t,\kappa}
	\frac{X_t}{\kappa}}] \,,
\end{equation}
where $\Lambda: \R \to \R$ is everywhere finite and continuously differentiable.
By the G\"artner-Ellis Theorem \cite[Theorem 2.3.6]{cf:DZ},
it follows that $\frac{X_t}{\kappa}$ satisfies a LDP with good
rate $\alpha_{t,\kappa}$ and with rate function $I(\cdot)$ given
by the Fenchel-Legendre transform of $\Lambda(\cdot)$, i.e.
\begin{equation} \label{ch3:eq:Legendre}
	I(x) = \sup_{\beta \in \R} \big\{ \beta x - \Lambda(\beta) \big\} \,.
\end{equation}
The proof is thus reduced to
computing $\Lambda(\beta)$ and then showing that $I(x)$ coincides
with the one given in \eqref{ch3:eq:gtIx}.
Recalling \eqref{ch3:eq:noexpl},
the determination of $\Lambda(\beta)$ in 
\eqref{ch3:eq:Lambda} is reduced to the asymptotic 
behaviour of exponential moments of $I_t$.
This is possible by Proposition~\ref{ch3:th:crucial}.

Fix a family of values of $(\kappa,t)$ 
satisfying \eqref{ch3:eq:ktd} and note that $\alpha_{t,\kappa}$
in \eqref{ch3:eq:gtIx} satisfies
\begin{equation*}
	\alpha_{t,\kappa} \to \infty \,, \qquad
	\frac{\alpha_{t,\kappa}}{\kappa}
	= \left( \frac{\kappa}{ \sfc^{1/D} \, t}\right)^{\frac{D}{1-D}}
 \left(\log \frac{\kappa}{ \sfc^{1/D} \, t}\right)^{\frac{1/2-D}{1-D}} \to \infty \,.
\end{equation*}
For fixed $\beta \in\R \setminus \{0\}$ we set
\begin{equation}\label{ch3:eq:bbb}
	 b = b_{t,\kappa}:= \frac{1}{2} \, \beta\frac{\alpha_{t,\kappa}}{\kappa}
 \left(\beta\frac{\alpha_{t,\kappa}}{\kappa}-1\right)
 \sim \frac{\beta^2}{2} \bigg(\frac{\alpha_{t,\kappa}}{\kappa}\bigg)^2 \to \infty \,.
\end{equation}
In order to check the second condition in \eqref{ch3:eq:assat}, note that if $t \to 0$
\begin{equation*}
 \begin{split}
 \frac{b}{\frac{1}{t^{2D}}\log \frac{1}{t}} \sim
 \frac{\beta^2}{2}\left( \frac{\kappa}{ \sfc^{1/D} \, t^D \sqrt{\log\frac{1}{t}}}\right)
 ^{\frac{2D}{1-D}}
 \left(\frac{\log \frac{\kappa}{ \sfc^{1/D} \, t}}{\log \frac{1}{t}} \right)^{\frac{1-2D}{1-D}}
 \to \infty \,,
 \end{split}
\end{equation*}
again by \eqref{ch3:eq:ktd}.
Applying \eqref{ch3:eq:keygoal0}, by \eqref{ch3:eq:noexpl} and \eqref{ch3:eq:bbb} we get
\begin{equation*}
\begin{split}
\log \E\big[ & e^{\beta \alpha_{t,\kappa} \frac{X_t}{\kappa}} \big] =
	\log \E[e^{b_{t,k} I_t}] \sim 
	\tilde C\, \big(  \sfc^{1/D} \, t\big) 
	\, b_{t,k}^{\frac{1}{2D}}(\log b_{t,k})^{\frac{2D-1}{2D}} \\
	&\sim \tilde C \,
	\big( \sfc^{1/D} \,  t \big) \, \bigg(\frac{\beta^2}{2}\bigg)^{\frac{1}{2D}}
	\left( \frac{\kappa}{ \sfc^{1/D} \, t}\right)^{\frac{1}{1-D}}
 \left(\log \frac{\kappa}{ \sfc^{1/D} \, t}\right)^{\frac{1-2D}{2D(1-D)}}
  \bigg(\frac{2D}{1-D}\log \frac{\kappa}{ \sfc^{1/D} \, t} \bigg)^{\frac{2D-1}{2D}}
	\\
	& = 
	D \, \bigg(\frac{
	(1-2D) (1-D)}{2}\bigg)^{\frac{1-2D}{2D}} \,
	|\beta|^{\frac{1}{D}} \,
	\alpha_{t,\kappa} \,,
\end{split}
\end{equation*}
where in the last step we have used the definitions \eqref{ch3:eq:gtIx},
\eqref{ch3:eq:keygoal0} of $\alpha_{t,\kappa}$ and $\tilde C$.

This shows that the limit \eqref{ch3:eq:Lambda} exists with
\begin{align*}
 \Lambda(\beta)&=\hat C\, |\beta|^{\frac{1}{D}} \,, \qquad \text{and} \qquad
 \hat C = D \, \bigg(\frac{
	(1-2D) (1-D)}{2}\bigg)^{\frac{1-2D}{2D}} \,.
\end{align*}
To determine the rate function $I(x)$ in \eqref{ch3:eq:Legendre} we have 
to maximize over $\beta \in \R$ the function
\begin{equation*}
	h(\beta):=\beta x- \Lambda (\beta) \,.
\end{equation*}
Since $h'(\beta)=x-\Lambda'(\beta)=
x-\frac{1}{D}\hat C \sign(\beta)|\beta|^{\frac{1}{D}-1}$,
the only solution to $h'(\bar \beta)=0$ is
\[
 \bar \beta = \bar \beta_x= \sign (x)\left(\frac{D|x|}{\hat C}\right)^{\frac{D}{1-D}}
\]
and consequently 
\[
\begin{split}
	 I(x) & = h(\bar \beta_x)= \bar \beta_x\, x-\Lambda
	 ( \bar \beta_x) = 
	 |x|^{\frac{1}{1-D}} \bigg(\frac{D}{\hat C} \bigg)^{\frac{D}{1-D}}
	 (1- D ) = \sfC \, |x|^{\frac{1}{1-D}} \,,
\end{split}
\]
where $\sfC$ is the constant defined in \eqref{ch3:eq:asf}.
Having shown that $I(x)$ coincides
with the one given in \eqref{ch3:eq:gtIx}, the proof of 
Theorem~\ref{ch3:th:LDP+} is completed.
\end{proof}

\subsection{Technical interlude}

Let us give some estimates on $\P(X_t>\kappa|N_t=m)$.
Recall the definition \eqref{ch3:eq:It} of the time-change process $I_t$.
On the event $\{N_t=0\}$ we have
\begin{equation*}
	I_t = (t-\tau_0)^{2D} - (-\tau_0)^{2D} \underset{t\downarrow 0}{\sim} \sigma_0^2 \, t \,,
\end{equation*}
where $\sigma_0^2$ is defined in \eqref{ch3:eq:sigma0}. Then,
by the definition \eqref{ch3:equivX} of $X_t$, 
as $t \to 0$ and for $\kappa \gg \sqrt{t}$,
\begin{equation*}
 \begin{split}
 \P(X_t>\kappa|N_t=0)
 &=\P\left(W_1 > \frac{\kappa}{\sqrt{I_t}} + \frac{1}{2} \sqrt{I_t} \bigg| N_t=0 \right) 
 = \P\left(W_1 > \frac{\kappa}{\sqrt{I_t}} \big(1+o(1)\big)\bigg| N_t=0 \right)  \\
 &= 1 - \Phi\bigg(
 \frac{\kappa}{\sigma_0\sqrt{t}}
  \big(1+o(1)\big) \bigg) 
 = \exp\bigg( - \frac{\kappa^2}{2 \, \sigma_0^2 \, t} \big(1+o(1)\big) \bigg) \\
  & = \exp\bigg( - \frac{1}{2} \bigg(\frac{\kappa}{\bkappa_1(\sigma_0^2 \, t)}\bigg)^2
	\log \frac{1}{ \sigma_0^2 \, t} \big(1+o(1)\big)\bigg)
	\,,
 \end{split}
\end{equation*}
where $\Phi(z) = \P(W_1 \le z)$ and
we have used the standard estimate $\log(1-\Phi(z)) \sim -\frac{1}{2}z^2$ as $z\to\infty$,
together with the definition \eqref{ch3:eq:gh} of $\bkappa_1(\cdot)$.
We can rewrite the previous relation as:
\begin{equation}\label{ch3:eq:Pevento0}
 \begin{split}
 \text{for } t \to 0, \ \kappa \gg \sqrt{t}: \qquad
 \P(X_t>\kappa \,|\, N_t=0)
 = \big( \sigma_0^2 \, t \big)^{ \frac{1}{2} \big(\frac{\kappa}
 {\bkappa_1(\sigma_0^2\, t)} \big)^2
	(1+o(1)) }
	\,.
 \end{split}
\end{equation}

On the other hand, on the event $\{N_t=m\}$ with $m \ge 1$,
we claim that
\begin{equation} \label{ch3:eq:genub}
\begin{split}
	& \text{for fixed } m \ge 1, \ \text{ for } t \to 0, \ \kappa \gg t^D: \\
	& \P(X_t >\kappa \,|\, N_t=m) =
	\big( \sfc^{1/D} \, t \big)^{\frac{1}{2 \, m^{1-2D}}
  \big(\frac{\kappa}{\bkappa_2(\sfc^{1/D} \, t)}\big)^2 
  (1+o(1))} \,.
\end{split}
\end{equation}
We first prove an upper bound. Applying \eqref{ch3:eq:eqItupfin},
on $\{N_t=m\}$ with $m \ge 1$ we have
\begin{equation*}
	I_t \le \sigma_0^2 \, t \,+\, N_t^{1-2D} (\sfc^{1/D} \, t)^{2D}
	= \sigma_0^2 \, t \,+\, m^{1-2D} (\sfc^{1/D} \, t)^{2D}
	\underset{t\downarrow 0}{\sim} 
	m^{1-2D} (\sfc^{1/D} \, t)^{2D} \,.
\end{equation*}
In analogy with the previous estimates, for $\kappa \gg t^D$ 
we have $\frac{\kappa}{\sqrt{I_t}} + \frac{1}{2} \sqrt{I_t}
= \frac{\kappa}{\sqrt{I_t}}(1+o(1))$ and
we get the upper bound
\begin{equation*}
\begin{split}
 \P(X_t >\kappa | N_t=m) & \leq 
 1 - \Phi\bigg(
	\frac{\kappa}{ (\sfc^{1/D} \, t)^D \, m^{\frac{1}{2}-D}}
  \big(1+o(1)\big) \bigg) \\
  & = \exp\bigg( -\frac{1}{2 \, m^{1-2D}}
  \bigg(\frac{\kappa}{\bkappa_2(\sfc^{1/D} \, t)}\bigg)^2 
  \log\frac{1}{ \sfc^{1/D} \, t}
  \big(1+o(1)\big)  \bigg)
  \,,
\end{split}
\end{equation*}
having used the definition \eqref{ch3:eq:gh} of $\bkappa_2(\cdot)$.
This proves half of \eqref{ch3:eq:genub}.

For a lower bound, we argue as in the proof 
of Proposition~\ref{ch3:th:crucial}:
for any $\epsilon > 0$, on the event $A_m \subseteq \{N_t = m\}$
defined in \eqref{ch3:eq:Am}, with $m \ge 1$, one has the lower bounds \eqref{ch3:eq:lbIXK1}
on $I_t$ and \eqref{ch3:eq:lbIXK2} on $\P(A_m)$. 
Then, using $(1+\frac{1}{m}) \le 2$, for $\kappa \gg t^D$ we get
\begin{equation*} 
\begin{split}
	 &\P( X_t >\kappa  | N_t = m) \ge
	  \P(X_t>\kappa | A_m) \frac{\P(A_m)}{\P(N_t=m)} \\
	  & \geq \Bigg(1 - \Phi\bigg(
	\frac{\kappa}{(1-2\epsilon)^{D} \,
	m^{\frac{1}{2}-D} (\sfc^{1/D} \, t)^D}
  \big(1+o(1)\big) \bigg) \Bigg) \, e^{-[(1+2\varepsilon)(1+\frac{1}{m})-1]\lambda t} 
  \epsilon^m 
  \frac{m!}{m^m} \\
  & \ge \exp\bigg( -\frac{1}{2 (1-2\epsilon)^{2D} \, m^{1-2D}}
  \bigg(\frac{\kappa}{\bkappa_2(\sfc^{1/D} \, t)}\bigg)^2 
  \log\frac{1}{\sfc^{1/D} \, t}
  \big(1+o(1)\big)  \bigg) e^{-(1+4\epsilon)\lambda t} \epsilon^m 
  \frac{m!}{m^m}  
  \,.
\end{split}
\end{equation*}
Since $e^{-(1+4\epsilon)\lambda t} \to 1$ and $\epsilon^m \frac{m!}{m^m}
= o(\log \frac{1}{\sfc^{1/D} \, t})$ as $t \to 0$ (for fixed $m$), we obtain
\begin{equation*}
	\frac{\log \P( X_t >\kappa  | N_t = m)}{\log \frac{1}{\sfc^{1/D} \, t}}
	\ge -\frac{1}{2 (1-2\epsilon)^{2D} \, m^{1-2D}}
  \bigg(\frac{\kappa}{\bkappa_2( \sfc^{1/D} \, t)}\bigg)^2 
  \big(1 + o(1) \big) \,.
\end{equation*}
Since $\epsilon > 0$ is arbitrary, we have proved the second half of \eqref{ch3:eq:genub}.

Finally, we need a weaker version of \eqref{ch3:eq:genub} when $\kappa = O(t^D)$:
\begin{equation} \label{ch3:eq:genlbweak}
\begin{split}
	&\text{for fixed } m \ge 1, \ \text{ for } t \to 0, \ \kappa = O(t^D):
	\qquad \P(X_t >\kappa \,|\, N_t=m) =
	\big( \sfc^{1/D} \, t \big)^{o(1)} \,.
\end{split}
\end{equation}
It is enough to show that $\P( X_t >\kappa  | N_t = m) \ge c_3$,
for some constant $c_3 = c_3(D,m) > 0$, because this would imply
\begin{equation*}
	0 \ge \frac{\log \P( X_t >\kappa  | N_t = m)}{\log \frac{1}{\sfc^{1/D} \, t}}
	\ge \frac{\log c_3}{\log \frac{1}{\sfc^{1/D} \, t}} = o(1) \,,
\end{equation*}
hence \eqref{ch3:eq:genlbweak} follows.
We fix $\epsilon > 0$ arbitrarily, and note that
on the event $A_m \subseteq \{N_t = m\}$ in \eqref{ch3:eq:Am} we have $I_t \ge (c_1 \, t^{D})^2$
by \eqref{ch3:eq:lbIXK1}
(for a suitable $c_1 > 0$ depending on $\epsilon, D, m$). Then, for $\kappa \le c_2 \, t^D$ we have
$\kappa / \sqrt{I_t} \ge c_2/c_1$, and since $\lim_{t\to 0} I_t = 0$,
for small $t$ we get our goal:
\begin{equation*}
\begin{split}
	\P( X_t >\kappa  | N_t = m) & \ge
	\P\left(W_1 > \frac{\kappa}{\sqrt{I_t}} + \frac{1}{2} \sqrt{I_t} 
	\,\bigg|\, A_m \right) \frac{\P(A_m)}{\P(N_t=m)} \\
	& \ge \P\left(W_2 > \frac{c_2}{c_1} + 1 \right)
	e^{-(1+4\epsilon)\lambda t} \epsilon^m \frac{m!}{m^m} \ge c_3 > 0 \,.
\end{split}
\end{equation*}

\subsection{Proof of Theorem~\ref{ch3:th:tailprob}, part \eqref{tp:c}}

We start with some general considerations.
We are in the regime $t\to 0$, hence $e^{-\lambda t} \ge \frac{1}{2}$ for small $t$.
For every $M \in\N_0$, since $N_t \sim Pois(\lambda t)$, we have the lower bound
\begin{equation}\label{eq:lb}
\begin{split}
	\P(X_t>\kappa) & \ge \sum_{m=0}^{M}\P(X_t>\kappa| N_t=m) \,
	e^{-\lambda t} \frac{(\lambda t)^m}{m!}	\\
	& \ge \frac{1}{2 \, M!} \, \max_{m \in \{0,\ldots, M\}} 
	\Big\{ \P(X_t>\kappa| N_t=m) \, (\lambda t)^m \Big\} \,.
\end{split}
\end{equation}
To get a similar upper bound, note that
\begin{equation*}
	\P(N_t \ge M+1) = \sum_{k=M+1}^\infty e^{-\lambda t} \, \frac{(\lambda t)^k}{k!}
	\le (\lambda t)^{M+1} \,,
\end{equation*}
hence for every $M\in\N_0$
\begin{equation}\label{eq:ub}
\begin{split}
	\P(X_t>\kappa) & \le \sum_{m=0}^{M}\P(X_t>\kappa| N_t=m) \,
	e^{-\lambda t} \frac{(\lambda t)^m}{m!} \,+\, (\lambda t)^{M+1}\\
	& \le \max_{m \in \{0,\ldots, M\}} 
	\Big\{ \P(X_t>\kappa| N_t=m) \, (\lambda t)^m \Big\} \,+\, (\lambda t)^{M+1}\,.
\end{split}
\end{equation}
It remains to evaluate the maximum of $\{\P(X_t>\kappa| N_t=m) \, (\lambda t)^m\}$.

We can now prove part~\eqref{tp:c} of Theorem~\ref{ch3:th:tailprob}.
Fix a family of values of $(\kappa,t)$ with
\begin{equation}\label{ch3:eq:eheheh}
	t \to 0 \,, \qquad
	\sqrt{ \sigma_0^2 \, t} \ll \kappa \le \sqrt{2} \, \bkappa_1( \sigma_0^2 \, t) \,.
\end{equation}
Then, recalling \eqref{ch3:eq:Pevento0}, relations \eqref{eq:lb} and \eqref{eq:ub}
with $M=0$ yield
\begin{equation*}
	\frac{1}{2} \, \big( \sigma_0^2 \, t \big)^{ \frac{1}{2} \big(\frac{\kappa}
	{\bkappa_1(\sigma_0^2\, t)} \big)^2(1+o(1)) }
	\le \P(X_t > k) \le \big( \sigma_0^2 \, t \big)^{ \frac{1}{2} \big(\frac{\kappa}
	{\bkappa_1(\sigma_0^2\, t)} \big)^2(1+o(1)) }
	\,+\, (\lambda t) \,,
\end{equation*}
and note that last term $(\lambda t)$ is negligible, because
$\frac{1}{2} \big(\frac{\kappa}{\bkappa_1(\sigma_0^2\, t)} \big)^2 \le 1$
by \eqref{ch3:eq:eheheh}. Taking logs, we see that relation \eqref{ch3:eq:withoutfa2} is proved.

\subsection{Proof of Theorem~\ref{ch3:th:tailprob}, part \eqref{tp:b}}

Following Remark~\ref{ch3:rem:tpbexplicit}, we split this case in two:
\begin{itemize}
\item first we consider a family of values of $(\kappa,t)$ with
\begin{equation}\label{ch3:eq:eheheh2}
	t \to 0 \,, \qquad \sqrt{2} \, \bkappa_1(\sigma_0^2 \, t) <
	\kappa \ll \bkappa_2( \sfc^{1/D} \, t) \,,
\end{equation}
and our goal is to prove \eqref{ch3:eq:withoutfa1};

\item afterwards we will consider the regime
\begin{equation}\label{ch3:eq:samefam2}
	\,\kappa \sim a\, \bkappa_2( \sfc^{1/D} \, t) \,,
	\qquad \text{for some} \quad a \in (0,\infty) \,.
\end{equation}
and our goal is to prove \eqref{ch3:eq:withfa}.
\end{itemize}
By a subsequence argument, these cases prove relation \eqref{ch3:eq:gentail}
and hence part \eqref{tp:b}.

\smallskip

Let us assume \eqref{ch3:eq:eheheh2}. Then for $m\ge 1$ we have
$\P(X_t >\kappa \,|\, N_t=m) = \big( \sfc^{1/D} \, t \big)^{o(1)}$,
by either \eqref{ch3:eq:genlbweak} (in case $\kappa = O(t^D)$)
or \eqref{ch3:eq:genub} (in case $t^D \ll \kappa \ll \bkappa_2(\sfc^{1/D} \, t)$). Then
\begin{equation*}
	\P(X_t >\kappa \,|\, N_t=1) \, (\lambda t)^1 = (\lambda t)^{1+o(1)} \,,
\end{equation*}
while by \eqref{ch3:eq:Pevento0}
\begin{equation*}
	\P(X_t >\kappa \,|\, N_t=0) \, (\lambda t)^0 = 
	\big( \sigma_0^2 \, t \big)^{ \frac{1}{2} \big(\frac{\kappa}
	{\bkappa_1(\sigma_0^2\, t)} \big)^2(1+o(1)) } \,.
\end{equation*}
Note that $\frac{1}{2} \big(\frac{\kappa}{\bkappa_1(\sigma_0^2\, t)} \big)^2 > 1$
in our regime, cf.\ \eqref{ch3:eq:eheheh2}, hence if we apply
relations \eqref{eq:lb} and \eqref{eq:ub}
with $M=1$ the maximum therein is attained for $m=1$ and we obtain
\begin{equation*}
\begin{split}
	\frac{1}{8} \, (\lambda \, t )^{1+o(1)}
	\le \P(X_t > k) \le (\lambda \, t )^{1+o(1)}
	\,+\, (\lambda t)^2 \,.
\end{split}
\end{equation*}
Taking logs, we have proved that relation \eqref{ch3:eq:withoutfa1} holds.

\smallskip

Next we assume \eqref{ch3:eq:samefam2}. Since $\kappa \gg t^D$, we can apply 
\eqref{ch3:eq:genub}, which yields
\begin{equation} \label{eq:repp}
\begin{split}
	\P(X_t >\kappa \,|\, N_t=m) \, (\lambda \, t)^m & =
	\big( \sfc^{1/D} \, t \big)^{\frac{a^2}{2 \, m^{1-2D}} (1+o(1))} \, (\lambda \, t)^m \\
	& = \big( \lambda \, t \big)^{m + \frac{\tilde a^2}{2 \, m^{1-2D}} + o(1)} \,, \\ 
\end{split}
\end{equation}
where $\tilde a$ is defined in \eqref{ch3:eq:withfa}.
(Since $\tilde a \to a$ as $t\to 0$, we could actually replace $\tilde a$ by $a$ in \eqref{eq:repp},
because the difference can be absorbed in the $o(1)$ term; however, keeping $\tilde a$
gives a more accurate numerical approximation.)

Let $\bar m = \bar m_{\tilde a} \in \N$ be the value for which
the minimum in the definition \eqref{ch3:eq:f} of $\sff(\tilde a)$ is attained, i.e.\
\begin{equation} \label{ch3:eq:fmax}
	\sff(\tilde a) = \sff_{\bar m}(\tilde a) =
	\bar m + \frac{\tilde a^2}{2 \, \bar m^{1-2D}} \,.
\end{equation}
If we choose $M > \bar m$, relations \eqref{eq:lb} and \eqref{eq:ub} yield
\begin{equation*}
	\frac{1}{2 \, M!} \, \big( \lambda \, t \big)^{\sff(\tilde a)}
	\le \P(X_t > k) \le \big( \lambda \, t \big)^{\sff(\tilde a)} \,+\, (\lambda t)^{M+1} \,.
\end{equation*}
We may assume that $M$ is large enough, so that $M+1 > \sff(\tilde a)$,
hence the term $(\lambda t)^{M+1}$
is negligible. Taking logs, we see that relation \eqref{ch3:eq:withfa} holds.

\section{Proof of Theorem~\ref{ch3:lemma:calltgamma} (option price)}
\label{ch3:sec:lemma:calltgamma}

In this section we prove Theorem~\ref{ch3:lemma:calltgamma}, or more precisely
we derive it from Theorem~\ref{ch3:th:tailprob} 
(which is proved in Section~\ref{ch3:sec:LDP}). This is based
on the results recently obtained in \cite{cf:CC} that link
tail probability and option price asymptotics, that we now summarize.

\subsection{From tail probability to option price}
\label{ch3:sec:reminders}
In this subsection $(X_t)_{t\ge 0}$ denotes a generic stochastic process,
representing the risk-neutral log-price, such that
$(e^{X_t})_{t\ge 0}$ is a martingale.
In order to determine the asymptotic behavior of the call price $c(\kappa,t) =
\E[(e^{X_t}-e^\kappa)^+]$ along a given family of values of $(\kappa,t)$
with $\kappa > 0$, $t > 0$,
we need some assumptions.
We start with the regime of \emph{atypical deviations}, i.e.\ we consider a family
of $(\kappa,t)$ such that $\P(X_t>\kappa) \to 0$. 
 
\begin{hypothesis}\label{ch2:ass:rv}
Along the family of $(\kappa,t)$ under consideration,
one has $\P(X_t>\kappa)\to 0$ and for every fixed $\rho \in [1,\infty)$ 
the following limit exists in $[0,+\infty]$:
\begin{equation} \label{ch2:eq:rv}
\begin{split}
	I_+(\rho) := \lim \frac{\log \P(X_t >\rho\kappa)}{\log \P(X_t > \kappa)} \,,
	\qquad \text{and moreover} \qquad \lim_{\rho\downarrow 1} I_+(\rho) = 1 \,.
\end{split}
\end{equation}
\end{hypothesis}

\noindent
We also need to formulate some moment conditions.
The first condition is
\begin{equation}\label{ch2:eq:moment}
	\forall \eta \in (0,\infty): \qquad \limsup\, \E[e^{(1+\eta)X_t}] < \infty \,,
   \end{equation}
where the limit is taken along the given family of $(\kappa,t)$
(however, only $t$ enters in \eqref{ch2:eq:moment}).
Note that if $t$ is bounded from above, say $t \le T$, it suffices
to require that
\begin{gather} \label{ch2:eq:momentsimple}
	\forall \eta \in (0,\infty): \qquad\E[e^{(1+\eta)X_T}] < \infty \,,
\end{gather}
because $(e^{(1+\eta)X_t})_{t\ge 0}$ is a submartingale
and consequently $\E[e^{(1+\eta)X_t}] \le \E[e^{(1+\eta)X_T}]$.
The second moment condition, to be applied 
when $t \to 0$ and $\kappa \to 0$, is
\begin{equation}\label{ch2:eq:moment0p}
	\exists C \in (0,\infty): \qquad
\E [e^{2 X_t} ] \le 1 + C \kappa^2 \,.
\end{equation}
\noindent
(We have stated the moment assumptions \eqref{ch2:eq:moment}
and \eqref{ch2:eq:moment0p} in a form that is enough for our purposes, but
they can actually be weakened, 
as we showed in \cite{cf:CC}.)

The next theorem, taken from \cite[Theorem 2.3]{cf:CC},
links the tail probability $\P(X_t > \kappa)$ and the option price $c(\kappa,t)$
in the regime of atypical deviations, generalizing \cite{cf:BF}.

\begin{theorem}\label{ch2:th:main2b}
Consider a risk-neutral log-price $(X_t)_{t\ge 0}$ and a
family of values of $(\kappa,t)$ with $\kappa > 0$, $t > 0$ such that
Hypothesis~\ref{ch2:ass:rv} is satisfied.
\begin{itemize}
\item In case $\liminf\kappa > 0$ and $\limsup t < \infty$,
if the moment condition \eqref{ch2:eq:moment} hold, then
\begin{equation}\label{ch2:eq:ma1c}
\begin{split}
	\log c(\kappa,t) & \sim  \log \P(X_t > \kappa) + \kappa \,.
\end{split}
\end{equation}

\item In case $\kappa \to 0$ and $t \to 0$,
if the moment condition \eqref{ch2:eq:moment0p} holds, and if in addition
\begin{equation}\label{ch2:eq:Iplus}
 \lim_{\rho \to +\infty}I_{+}(\rho )=+\infty \,,
\end{equation}
then
\begin{align} \label{ch2:eq:ma2c}
	\log \big( c(\kappa,t) / \kappa \big) & \sim  \log \P(X_t > \kappa) \,.
\end{align}
\end{itemize}
\end{theorem}

Next we discuss the case of \emph{typical deviations},
i.e.\ we consider a family
of values of $(\kappa,t)$ such that $\kappa \to 0$, $t \to 0$
in such a way that $\P(X_t>\kappa)$ is bounded away from zero.
In this case we assume the convergence in distribution of $X_t$, suitably
rescaled, as $t \to 0$.

\begin{hypothesis} \label{ch2:ass:smalltime}
There is a positive function $(\gamma_t)_{t > 0}$ with
$\lim_{t\downarrow 0}\gamma_t = 0$ such
that $X_t / \gamma_t$ converges in law as $t \downarrow 0$ to
some random variable $Y$:
\begin{equation} \label{ch2:eq:scaling}
	\frac{X_t}{\gamma_t} \xrightarrow[t\downarrow 0]{d} Y \,.
\end{equation}
\end{hypothesis}

\noindent
The next result is \cite[Theorem 2.7]{cf:CC}.

\begin{theorem}\label{ch2:th:main2a}
Assume that Hypothesis~\ref{ch2:ass:smalltime} is satisfied,
and moreover the moment condition \eqref{ch2:eq:moment0p} holds
with $\kappa = \gamma_t$, i.e.
\begin{equation}\label{ch2:eq:assX}
\exists C \in (0, \infty):  \quad
\E\left[e^{2X_t}\right] < 1+ C \gamma_t^2\,.
\end{equation}
Consider a family of values of $(\kappa,t)$ such that
$t \to 0$ and $\kappa \sim a \gamma_t$, with $a \in [0,\infty)$
(in case $a=0$, we mean $\kappa = o(\gamma_t)$). Then,
assuming that $\P(Y > a) > 0$, one has
\begin{equation} \label{ch2:eq:astypical}
	c(\kappa, t) \sim \gamma_t \, \E[(Y-a)^+]  \,.
\end{equation}
\end{theorem}

\subsection{Proof of Theorem~\ref{ch3:lemma:calltgamma}, part \eqref{op:a}}

We fix a family of values of $(\kappa,t)$
such that either $t \to \bar t \in (0,\infty)$ and $\kappa \to \infty$,
or $t \to 0$ and $\kappa \to \bar\kappa \in (0,\infty]$.
Let us check the assumptions of Theorem~\ref{ch2:th:main2b}.
Relation \eqref{ch3:eq:Cramer2} shows that for all $\rho\ge 1$
\begin{equation} \label{ch3:eq:I+}
	\lim \frac{\log \P(X_t > \rho\kappa)}{\log \P(X_t > \kappa)}
	= \rho^{\frac{1}{1-D}} \,,
\end{equation}
hence Hypothesis~\ref{ch2:ass:rv} is satisfied
with $I_+(\rho) := \rho^{\frac{1}{1-D}}$.
The moment condition \eqref{ch2:eq:moment}
is implied by \eqref{ch2:eq:momentsimple}, which holds for all $T \in (0,\infty)$,
by Lemma~\ref{ch3:th:noexpl}.
By Theorem~\ref{ch2:th:main2b}, relation \eqref{ch2:eq:ma1c} holds.
However, since $-\log \P(X_t > \kappa) / \kappa \to \infty$
by \eqref{ch3:eq:Cramer2} (note that $\frac{1}{1-D} > 1$), relation \eqref{ch2:eq:ma1c} yields
\begin{equation} \label{ch3:eq:repc}
	\log c(\kappa,t) \sim  \log \P(X_t > \kappa) \,,
\end{equation}
which is precisely relation \eqref{ch3:eq:casinfty}. This completes the proof of
part \eqref{op:a} of Theorem~\ref{ch3:lemma:calltgamma}.\qed

\subsection{Proof of  Theorem~\ref{ch3:lemma:calltgamma}, part \eqref{op:b}}

Let us fix a family of values of $(\kappa,t)$ with $t \to 0$ and $\kappa \to 0$,
such that $\kappa \gg \sqrt{\sigma_0^2 \, t}$, excluding the regime
$\sqrt{2D+1} \, \bkappa_1( \sigma_0^2 \, t)  \leq \kappa 
\ll \bkappa_2( \sfc^{1/D} \, t)$
of part \eqref{op:c}. By a subsequence argument, it suffices to consider
the following regimes:
\begin{ienumerate}
\item\label{sr:1} $\sqrt{\sigma_0^2 \, t} \ll \kappa \ll \bkappa_1( \sigma_0^2 \, t)$;

\item\label{sr:2} $\kappa \sim a \, \bkappa_1( \sigma_0^2 \, t)$ 
with $a \in (0, \sqrt{2D+1}]$;

\item\label{sr:3} $\kappa \sim a \, \bkappa_2( \sfc^{1/D} \,  t)$ with $a \in (0, \infty)$;

\item\label{sr:4} $\kappa \gg \bkappa_2( \sfc^{1/D} \,  t)$.
\end{ienumerate}

We start checking Hypothesis~\ref{ch2:ass:rv}
in regimes \eqref{sr:1}, \eqref{sr:3} and \eqref{sr:4}
(the regime \eqref{sr:2} will be considered later).
In regime \eqref{sr:4}, relation \eqref{ch3:eq:Cramer2}
holds, cf.\ part \eqref{tp:a} in Theorem~\ref{ch3:th:tailprob}, hence
\eqref{ch3:eq:I+} applies again and 
$I_+(\rho) = \rho^{\frac{1}{1-D}}$ (recall \eqref{ch2:eq:rv}).
In regime \eqref{sr:1}, by relation \eqref{ch3:eq:withoutfa2},
\begin{equation} \label{ch3:eq:I+2}
	I_+(\rho) := \lim \frac{\log \P(X_t > \rho\kappa)}{\log \P(X_t > \kappa)}
	= \rho^2 \,.
\end{equation}
Finally, in regime \eqref{sr:3}, by \eqref{ch3:eq:gentail} (or equivalently \eqref{ch3:eq:withfa}),
\begin{equation} \label{ch3:eq:I+3}
	I_+(\rho) := \lim \frac{\log \P(X_t > \rho\kappa)}{\log \P(X_t > \kappa)}
	= \frac{\sff(\rho a)}{\sff(a)} \,.
\end{equation}
In all cases, Hypothesis~\ref{ch2:ass:rv} and relation \eqref{ch2:eq:Iplus} are satisfied.
As we show in a moment, also the moment condition \eqref{ch2:eq:moment0p} is satisfied.
Having checked all the assumptions of Theorem~\ref{ch2:th:main2b}
(recall that $t\to 0$ and $\kappa \to 0$), relation \eqref{ch2:eq:ma2c} holds.
This coincides with our goal \eqref{ch3:eq:cas0good}, completing the proof of 
part \eqref{op:b} of Theorem~\ref{ch3:lemma:calltgamma} 
in regimes \eqref{sr:1}, \eqref{sr:3} and \eqref{sr:4}.

\smallskip

It remains to check the moment condition \eqref{ch2:eq:moment0p}
in regimes \eqref{sr:1}, \eqref{sr:3} and \eqref{sr:4}. Since $\kappa \gg 
\sqrt{ \sigma_0^2 \, t}$
in all these regimes, this follows immediately from the next Lemma.

\begin{lemma}\label{ch3:lemma:moments}
There exists a constant $C \in (0,\infty)$
such that
 \[
  \E\big[e^{2X_t}\big]\leq 1 + C \, t \,, \qquad
  \forall 0 \le t \le 1 \,.
 \]
\end{lemma}

\begin{proof}
By the equality in \eqref{ch3:eq:noexpl} and the upper bound \eqref{ch3:eq:eqItupfin}, 
we can write
\begin{equation}\label{ch3:eq:mta}
	\E[e^{2X_t}] = \E[e^{I_t}] \le e^{\sigma_0^2 t} \, \E[e^{\sfc^2 t^{2D} N_t^{1-2D}}] \,.
\end{equation}
Next observe that, by Cauchy-Schwarz 
inequality and $\P(N_t = k) = e^{-\lambda t} \frac{(\lambda t)^k}{k!}$,
\begin{equation} \label{ch3:eq:ity}
\begin{split}
	\E[e^{\sfc^2 t^{2D} N_t^{1-2D}}] & = \P(N_t = 0)
	+ e^{\sfc^2 t^{2D}} \P(N_t=1) + \E[e^{\sfc^2 t^{2D} N_t^{1-2D}}
	\ind_{\{N_t \ge 2\}}] \\
	& \le 1 + e^{\sfc^2 t^{2D}} \lambda t  +
	\sqrt{\E[e^{2\sfc^2 t^{2D} N_t^{1-2D}}] \P(N_t\ge 2)} \,.
\end{split}
\end{equation}
Note that 
$\P(N_t\ge 2) = 1 - e^{-\lambda t} (1+\lambda t) = \frac{1}{2} (\lambda t)^2
+ o(t^2)$ as $t\downarrow 0$.
For all $0 \le t \le 1$ we can write
$\E[e^{2\sfc^2 t^{2D} N_t^{1-2D}}] \le \E[e^{2\sfc^2 N_1^{1-2D}}] =: c_1 < \infty$,
and $e^{\sfc^2 t^{2D}} \le e^{\sfc^2}$, hence \eqref{ch3:eq:ity} yields
\begin{equation*}
	\E[e^{\sfc^2 t^{2D} N_t^{1-2D}}] \le 1 + e^{\sfc^2} \lambda t
	+ \sqrt{\frac{c_1 \lambda^2}{2} \, (t^2+o(t^2))} \le 1 + c_2 \, t \,,
\end{equation*}
for some $c_2 < \infty$. Consequently, by \eqref{ch3:eq:mta},
\begin{equation*}
	\E[e^{2 X_t}] \le e^{\sigma_0^2 t} \big(1+c_2 \, t \big)
	= \big(1+ \sigma_0^2 \, t + o(t) \big)\big(1+c_2 \, t \big)
	\le 1 + C t \,,
\end{equation*}
for some $C < \infty$.
\end{proof}

We are left with considering regime \eqref{sr:2}, i.e.\ we fix a family of $(\kappa,t)$ such that
\begin{equation}\label{ch3:eq:family}
	t \to 0 \qquad\text{and}\qquad \kappa \sim a\, \bkappa_1(\sigma_0^2 \, t) \,,
	\quad \text{for some} \quad  a \in (0, \sqrt{2D+1}] \,.
\end{equation}
In this regime the assumptions of Theorem~\ref{ch2:th:main2b} are \emph{not}
verified, hence we proceed by bare hands estimates.
Our goal is to prove \eqref{ch3:eq:cas0good} which,
recalling \eqref{ch3:eq:withoutfa2}, can be rewritten as
\begin{equation}\label{ch3:eq:goalgoal}
	\log \big( c(\kappa,t)/\kappa \big) \sim 
	- \frac{a^2}{2} \, \log \frac{1}{ \sigma_0^2 \, t}  \,.
\end{equation}
We prove separately upper and lower bounds for this relation.

Let us set
\begin{equation}\label{ch3:eq:family'}
	\kappa' :=  2 \, \bkappa_1( \sigma_0^2 \, t) \,, \qquad
	\kappa'' := B \, \bkappa_2( \sfc^{1/D} \,  t) \,,
\end{equation}
for fixed $B \in (0,\infty)$, chosen later.
Noting that $\kappa < \kappa' < \kappa''$, since $D < \frac{1}{2}$, we can write
\begin{equation}\label{ch3:eq:upperboundcallK}
 \begin{split}
c(\kappa,t) &=\E\left[(e^{X_t}-e^{\kappa})1_{\{X_t > \kappa\}}\right]\\
&=\E\left[(e^{X_t}-e^{\kappa})1_{\{\kappa < X_t \le 
\kappa'\}}\right]
+\E\left[(e^{X_t}-e^{\kappa})1_{\{\kappa' < X_t \le
\kappa''\}}\right]  +\E\left[(e^{X_t}-e^{\kappa})1_{\{X_t>\kappa''\}}\right]\\
&=(1)+(2)+(3) \,.
 \end{split}
\end{equation}
By Fubini's theorem, for $\kappa \ge 0$ and $0 \le a < b$,
\begin{equation} \label{ch3:eq:alwaysby}
\begin{split}
	\E\big[ (e^{X_t}  - e^{\kappa})1_{\{a < X_t \le b\}}\big]
	&= \E\left[ \left( \int_\kappa^{\infty} 
	e^x \, \ind_{\{x < X_t\}} \, \dd x \right)  \ind_{\{a < X_t \le b\}} \right] \\
	& =\int_\kappa^{b} 
	e^x \, \P(\max\{a,x\} < X_t \le b) \, \dd x\\
	& \le (e^b-1) \, \P(X_t > \max\{a,\kappa\}) \,,
\end{split}
\end{equation}
hence
\begin{equation} \label{ch3:eqqe}
 (1) = \E\left[(e^{X_t}-e^{\kappa})1_{\{\kappa < X_t \le 
\kappa'\}}\right]
\leq (e^{\kappa'} - 1) \P(X_t>\kappa)
 \sim \kappa' \,\P(X_t>\kappa) \,, 
\end{equation}
because $\kappa' \to 0$.
Note that, by \eqref{ch3:eq:family} and \eqref{ch3:eq:withoutfa2},
\begin{equation} \label{ch3:eq:gonnause}
	\log \P(X_t > \kappa) \sim -
	\frac{a^2}{2}\, \log \frac{1}{ \sigma_0^2 \, t} \,,
\end{equation}
and since $\frac{\kappa'}{\kappa} \sim \frac{ 2}{a} = (const.)$,
recall \eqref{ch3:eq:family'}, it follows by \eqref{ch3:eqqe} that
\begin{equation} \label{eq:est1}
	\log \frac{(1)}{\kappa} \le 
	\log \frac{\kappa'}{\kappa} + \log \P(X_t > \kappa) =
	 - \frac{a^2}{2} \, \log \frac{1}{ \sigma_0^2 \, t} \big(1+o(1)\big)
	  =: \mathrm{Est}(1)\,.
\end{equation}
In a similar way, always using \eqref{ch3:eq:alwaysby},
since $\kappa < \kappa'$ and $\kappa'' \to 0$,
\begin{equation}
 (2) = \E\left[(e^{X_t}-e^{\kappa})1_{\{\kappa' < X_t \le
\kappa''\}}\right] \leq (e^{\kappa''}-1) \P(X_t > \kappa')
 \sim \kappa'' \,\P(X_t>\kappa') \,.
\end{equation}
By \eqref{ch3:eq:withoutfa1},
we can write
\begin{equation}\label{eq:2}
 \begin{split}
	\log \frac{(2)}{\kappa} & \le 
	\log \frac{B\, \bkappa_2(\sfc^{1/D} \, t)}{\kappa}
	+ \log \P(X_t>\kappa') \\
	& \le \big( 1 + o(1) \big) \bigg\{ \log \frac{\bkappa_2(\sfc^{1/D} \, t)}{\kappa}
	-  \log \frac{1}{\lambda \, t} \bigg\} 
	=: \mathrm{Est}(2) \,.
\end{split}
\end{equation}
Note that $\log \frac{\bkappa_2(\sfc^{1/D}t)}{\kappa} \sim (\frac{1}{2}-D)
\log \frac{1}{\lambda t}$, hence
\begin{equation} \label{eq:est2}
	\mathrm{Est(2)} = -\big( 1+o(1) \big) \bigg(D+\frac{1}{2}\bigg)
	\log \frac{1}{\lambda t} \,.
\end{equation}
Finally, by Cauchy-Schwarz inequality 
\begin{equation} \label{ch3:eq:est30}
 (3) = \E\left[(e^{X_t}-e^{\kappa})1_{\{X_t>\kappa''\}}\right]
 \leq \kappa \, 
\sqrt{ \E\left[\left(\frac{e^{X_t}-e^{\kappa}}{\kappa}\right)^2\right]
\P(X_t>\kappa'')} \,. 
\end{equation}
By Lemma~\ref{ch3:lemma:moments} and $\E[e^{X_t}] = 1$
(recall that $(e^{X_t})_{t\ge 0}$ is a martingale) we have
\begin{equation*}
	\E\left[\left(\frac{e^{X_t}-e^{\kappa}}{\kappa}\right)^2\right]
	= \frac{\E[e^{2X_t}] - 2 e^\kappa + e^{2\kappa}}{\kappa^2}
	\le \frac{1 + C t - 2 e^\kappa + e^{2\kappa}}{\kappa^2}
	= \frac{C t}{\kappa^2} + \frac{e^{2\kappa}  - 2 e^\kappa -1}{\kappa^2} 
	\to 1 \,,
\end{equation*}
because $\kappa \to 0$ and $\kappa /\sqrt{t} \to \infty$, by \eqref{ch3:eq:family}
and the definition \eqref{ch3:eq:gh} of $\bkappa_1(\cdot)$. 
In particular, 
\begin{equation*}
	 (3)\leq \big(1+o(1)\big) \, \kappa \, \sqrt{\P(X_t>\kappa'')} \,.
\end{equation*}
Recalling \eqref{ch3:eq:withfa} (where $\tilde a \to a$ as $t\to 0$), we see that
\begin{equation} \label{ch3:eq:est3}
	\log \frac{(3)}{\kappa} \le -\big(1+o(1)\big)\frac{1}{2} \sff(B) \log\frac{1}
	{ \lambda \, t}  =: \mathrm{Est}(3) \,.
\end{equation}
Let us choose $B > 0$ large enough, so that $\frac{\sff(B)}{2} > D+\frac{1}{2}$,
so that $\mathrm{Est(3)} < \mathrm{Est(2)}$.
Since $\log (a+b+c) \leq \log 3 +\max \{\log a, \log b, \log c \}$, we obtain
by \eqref{ch3:eq:upperboundcallK}
\begin{equation}\label{ch3:eq:ufa}
\begin{split}
 \log \frac{c(\kappa,t)}{\kappa} \leq 
	\max\big\{
	\mathrm{Est}(1),\, \mathrm{Est}(2) \big\} \,.
\end{split}
\end{equation}

We now use the assumption $a \le \sqrt{2D+1}$,
cf.\ \eqref{ch3:eq:family}. Then
$\mathrm{Est}(1) \ge (1+o(1))\mathrm{Est}(2)$, so
\begin{equation} \label{ch3:eq:halfgoal}
	\log \frac{c(\kappa,t)}{\kappa}
	\le \mathrm{Est}(1)
	= - \big(1+o(1)\big)
	\frac{a^2}{2} \, \log \frac{1}{ \sigma_0^2 \, t} \,,
\end{equation}
which is ``half'' of our goal \eqref{ch3:eq:goalgoal}.

In order to obtain the corresponding lower bound, we observe that for every $\hat\kappa >\kappa$
\begin{equation}\label{ch3:eq:lowerboundctgammafinale}
\begin{split}
c(\kappa,t) &= \E\left[\left(e^{X_t}-e^{\kappa} \right)1_{\{X_t>\kappa\}}\right]
\ge
\E\left[\left(e^{X_t}-e^{\kappa}\right)1_{\{X_t>
\hat\kappa\}}\right] \geq (e^{
\hat\kappa}-e^{\kappa})\P(X_t>
\hat\kappa)\\
&\geq (
\hat\kappa-\kappa)\P(X_t>\hat\kappa).
\end{split}
\end{equation}
Always for $\kappa$ as in \eqref{ch3:eq:family},
choosing $\hat\kappa = (1+\epsilon)\kappa$ gives, recalling \eqref{ch3:eq:gonnause},
\begin{equation} \label{ch3:eq:lbuff}
 \log \frac{c(\kappa,t)}{\kappa}
 \geq \log \epsilon + \log\P(X_t>(1+\epsilon)\kappa)
 =  -(1+\epsilon)^2 \frac{a^2}{2} \, \log \frac{1}{\sigma_0^2 \, t} \big(1+o(1)\big)\,.
\end{equation}
This shows that,
along the given family of values of $(\kappa,t)$,
\begin{equation*}
	\liminf \frac{\frac{c(\kappa,t)}{\kappa}}{\log \frac{1}{\sigma_0^2 \, t}}
	\ge - (1+\epsilon)^2 \frac{a^2}{2} \,.
\end{equation*}
Since $\epsilon > 0$ is arbitrary, we have shown that
\begin{equation}
 \log \frac{c(\kappa,t)}{\kappa}\geq - 
 \big(1+o(1)\big) \, \frac{a^2}{2} \, \log \frac{1}{ \sigma_0^2 \, t}\,.
\end{equation}
Together with \eqref{ch3:eq:halfgoal}, 
this completes the proof of \eqref{ch3:eq:goalgoal}
and of part \eqref{tp:b} of Theorem~\ref{ch3:lemma:calltgamma}.\qed

\subsection{Proof of  Theorem~\ref{ch3:lemma:calltgamma}, part \eqref{op:c}}

Let us fix a family of values of $(\kappa,t)$ with 
\begin{equation}\label{ch3:eq:familybis}
	t \to 0 \qquad\text{and}\qquad 
	\sqrt{2D+1}\, \bkappa_1(\sigma_0^2 \, t) \le \kappa \ll \, 
	\bkappa_2(\sfc^{1/D} \, t) \,.
\end{equation}
Our goal is to prove \eqref{ch3:eq:cas0bad2}, that is
\begin{equation}\label{ch3:eq:goalnew}
	\log \big( c(\kappa,t)/\kappa \big) \sim 
	\log \frac{\bkappa_2(\sigma_0^2 \, t)}{\kappa}
	- \log \frac{1}{\sigma_0^2 \, t}\,.
\end{equation}

Consider first the subregime of \eqref{ch3:eq:familybis}
given by $\kappa \le \sqrt{2} \, \bkappa_1(\sigma_0^2 \, t)$,
so assume (without loss of generality, by extracting
a subsequence) that $\kappa \sim a \, \bkappa_1(\sigma_0^2 \, t)$ with
$a \in [\sqrt{2D+1}, \sqrt{2}]$.
Note that all the steps from \eqref{ch3:eq:family'} until \eqref{ch3:eq:ufa}
can be applied verbatim. However, this time
$a \ge \sqrt{2D+1}$, hence $\frac{a^2}{2} \ge D+\frac{1}{2}$
which yields $\mathrm{Est}(1) \le (1+o(1))\mathrm{Est}(2)$, cf.\
\eqref{eq:est1} and \eqref{eq:est2}.
Then, instead of relation \eqref{ch3:eq:halfgoal},  we get (recall \eqref{eq:2})
\begin{equation} \label{ch3:eq:halfgoalbis}
	\log \frac{c(\kappa,t)}{\kappa}
	\le \mathrm{Est}(2) =
	\big( 1 + o(1) \big) \bigg\{ \log \frac{\bkappa_2(\sfc^{1/D} \, t)}{\kappa}
	-  \log \frac{1}{\lambda \, t} \bigg\} \,,
\end{equation}
which is ``half'' of our goal \eqref{ch3:eq:goalnew}.

Next we consider the subregime of \eqref{ch3:eq:familybis} 
of $\kappa > \sqrt{2} \, \bkappa_1(\sigma_0^2 \, t)$.
Defining $\kappa'' := B \,\bkappa_2(\sfc^{1/D} \, t)$ as in \eqref{ch3:eq:family'},
we modify \eqref{ch3:eq:upperboundcallK} as follows:
\begin{equation}\label{ch3:eq:upperboundcallKbis}
 \begin{split}
c(\kappa,t) &=\E\left[(e^{X_t}-e^{\kappa})1_{\{\kappa < X_t \le 
\kappa''\}}\right] +\E\left[(e^{X_t}-e^{\kappa})1_{\{X_t>\kappa''\}}\right]
=: (A) + (B) \,.
 \end{split}
\end{equation}
Applying \eqref{ch3:eq:alwaysby}, we estimate the first term as follows,
since $\kappa'' \to 0$:
\begin{equation*}
	(A) = \E\left[(e^{X_t}-e^{\kappa})1_{\{\kappa < X_t \le 
	\kappa''\}}\right] \le (e^{\kappa''}-1) \, \P(X_t > \kappa)
	\sim \kappa'' \, \P(X_t > \kappa) \,.
\end{equation*}
By \eqref{ch3:eq:withoutfa1} we have
$\log \P(X_t > \kappa) = -(1+o(1))\log \frac{1}{\lambda \, t}$,
hence
\begin{equation*}
	\log \frac{(A)}{\kappa} \le \log\frac{\kappa''}{\kappa}
	+ \log\P(X_t > \kappa) 
	= \big(1+o(1)\big)  \bigg\{ \log \frac{\bkappa_2(\sfc^{1/D} \, t)}{\kappa}
	-  \log \frac{1}{\lambda \, t} \bigg\} \,.
\end{equation*}
The term (B) in \eqref{ch3:eq:upperboundcallKbis} coincides with term
(3) in \eqref{ch3:eq:est30}, hence by \eqref{ch3:eq:est3} 
\begin{equation} \label{eq:neweq}
\begin{split}
	\log \frac{(B)}{\kappa} \le -\big(1+o(1)\big) \frac{\sff(B)}{2} \log\frac{1}{\lambda \, t} 
	& \le \big(1+o(1)\big)  
	\frac{\sff(B)}{2} \bigg\{ \log \frac{\bkappa_2(\sfc^{1/D} \, t)}{\kappa}
	-  \log \frac{1}{\lambda \, t} \bigg\} \,,
\end{split}
\end{equation}
where the second inequality holds since
$\kappa \le \bkappa_2(\sfc^{1/D} \, t)$ by \eqref{ch3:eq:familybis}.
Choosing $B$ large enough, so that $\sff(B) > 2$,
the inequality $\log(a+b) \le \log 2 + \log\max\{a,b\}$ yields
\begin{equation}\label{ch3:eq:ub}
	\log \frac{c(\kappa,t)}{\kappa} \le \big(1+o(1)\big) \bigg\{ \log 
	\frac{\bkappa_2(\sfc^{1/D} \, t)}{\kappa}
	-  \log \frac{1}{\lambda \, t} \bigg\} \,.
\end{equation}
We have thus proved ``half'' of our goal \eqref{ch3:eq:goalnew}.

\smallskip

We finally turn to the lower bound,
for which we do not need to distinguish subregimes,
but we work in the general regime \eqref{ch3:eq:familybis}.
We can apply \eqref{ch3:eq:lowerboundctgammafinale}
with $\hat\kappa = \epsilon\bkappa_2(\sfc^{1/D} \, t)$.
Recalling that $\log \P(X_t > \epsilon \, \bkappa_2(\sfc^{1/D} \, t)) 
\sim -\sff(\epsilon) \log\frac{1}{ \lambda \, t}$
by \eqref{ch3:eq:withfa}, and moreover
\begin{equation*}
	\log \frac{\hat\kappa - \kappa}{\kappa} \sim
	\log \bigg( \frac{\epsilon \bkappa_2( \sfc^{1/D} \, t)}{\kappa} - 1 \bigg) \sim
	\log \frac{\bkappa_2(\sfc^{1/D} \, t)}{\kappa} \,,
\end{equation*}
we see that relation \eqref{ch3:eq:lowerboundctgammafinale} gives
\begin{equation} \label{ch3:eq:lbuff2}
 \log \frac{c(\kappa,t)}{\kappa}
	\ge -\big(1+o(1)\big) 
	\bigg\{ \log \frac{\bkappa_2(\sfc^{1/D} \, t)}{\kappa}
	- \sff(\epsilon) \, \log \frac{1}{\lambda \, t} \bigg\} \,.
\end{equation}
Since $\epsilon > 0$ is arbitrary and $\lim_{\epsilon \downarrow 0} \sff(\epsilon)
= \sff(0) = 1$, cf.\ \eqref{ch3:eq:f}, we have shown that
\begin{equation*}
	 \log \frac{c(\kappa,t)}{\kappa}
	\ge -\big(1+o(1)\big) \bigg\{ \log \frac{\bkappa_2(\sfc^{1/D} \, t)}{\kappa}
	- \log \frac{1}{\lambda \, t} \bigg\} \,.
\end{equation*}
Together with \eqref{ch3:eq:halfgoalbis} and \eqref{ch3:eq:ub},
this completes the proof of our goal \eqref{ch3:eq:goalnew}.\qed

\subsection{Proof of  Theorem~\ref{ch3:lemma:calltgamma}, parts \eqref{op:d} and \eqref{op:e}}
\label{ch3:sec:callesigmat05}

By \eqref{ch3:eq:gammat}, Hypothesis~\ref{ch2:ass:smalltime} is satisfied
with $\gamma = \sqrt{\sigma_0^2 \, t}$ and $Y = W_1$, while the moment condition 
\eqref{ch2:eq:assX} is verified 
by Lemma~\ref{ch3:lemma:moments}.
We can then apply relation \eqref{ch2:eq:astypical} in Theorem~\ref{ch2:th:main2a}, 
which for $\kappa \sim a \sqrt{\sigma_0^2 \, t}$ yields
\begin{equation*}
 \begin{split}
c(\kappa,t) &\sim \sqrt{\sigma_0^2 \, t}
\, \E\left[\left(W_1- a\right)^+\right]
=\sqrt{\sigma_0^2 \, t} \, \left[\int_{ a}^{\infty}
x \, \frac{e^{-\frac{x^2}{2}}}{\sqrt{2\pi}} \, \dd x
\,-\,  a \int_{ a}^{\infty} 
\frac{e^{-\frac{x^2}{2}}}{\sqrt{2\pi}} \, \dd x \right]\\
&=\sqrt{\sigma_0^2 \, t}\,
\Bigg(\frac{e^{-\frac{a^2}{2}}}{\sqrt{2\pi}}-
 a \left(1-\Phi\left( a\right)\right)\Bigg)
=\sqrt{\sigma_0^2 \, t}\, \left(\phi\left( a\right) \,-\,
 a \, \Phi\left(- a\right)\right) \,.
 \end{split}
\end{equation*}
For $a > 0$ this coincides with \eqref{ch3:eq:typ}, while for $a=0$
it coincides with \eqref{ch3:eq:typlast}.\qed

\section{Proof of Theorem~\ref{ch3:th:main} (implied volatility)}
\label{ch3:sec:th:main}

In this section we prove Theorem~\ref{ch3:th:main}, or more precisely
we derive it from Theorem~\ref{ch3:lemma:calltgamma} 
(which is proved in Section~\ref{ch3:sec:lemma:calltgamma}). 
In fact, the link between option price and implied volatility asymptotics
is model independent, as recently shown in \cite{cf:GL}.
Let us summarize the results that will be needed in the sequel, following \cite{cf:CC}.

\subsection{From option price to implied volatility}
\label{ch3:sec:reminders2}

Let us define the function
\begin{equation} \label{ch2:eq:D}
	D(z) := \frac{1}{z}\phi(z) - \Phi(-z) , \qquad \forall z > 0 \,,
\end{equation}
where $\phi(\cdot)$ and $\Phi(\cdot)$ 
are the density and distribution function of a standard Gaussian.
Since $D: (0,\infty) \to (0,\infty)$ is smooth and strictly decreasing,
its inverse $D^{-1}:(0,\infty) \to (0,\infty)$ is also smooth, strictly decreasing and 
has the following asymptotic behavior \cite[\S 4.1]{cf:CC}:
\begin{equation}\label{ch2:eq:Das}
	D^{-1}(y) \sim \sqrt{2 \, (-\log y)}  \ \ \ \
	\text{as } y \downarrow 0  \,, \quad \quad \ \
	D^{-1}(y) \sim \frac{1}{\sqrt{2\pi}} \frac{1}{y} \ \ \ \
	\text{as } y \uparrow \infty\,.
\end{equation}

The next result links
option price and implied volatility in a model independent way.

\begin{theorem}\label{ch2:th:main1}
Consider a family of values of $(\kappa,t)$ with $\kappa\ge 0$, $t>0$,
such that $c(\kappa, t) \to 0$.
\begin{itemize}
\item In case $\liminf\kappa > 0$, one has
\begin{equation} \label{ch2:eq:Vas>00}
	\sigma_\imp(\kappa,t)  \sim
	\bigg( \sqrt{\frac{-\log c(\kappa,t)}{\kappa}+1} 
	-\sqrt{\frac{-\log c(\kappa,t)}{\kappa}} \, \bigg) \,
	\sqrt{\frac{2\kappa}{t}} \,.
\end{equation}
\item In case $\kappa\to 0$, with  $\kappa>0$, one has
\begin{equation} \label{ch2:eq:Vas<infty0}
	\sigma_\imp(\kappa,t)  \sim
	\frac{1}{\rule{0pt}{1.2em}D^{-1} 
	\Big(\frac{c(\kappa,t)}{\kappa} \Big)}
	\frac{\kappa}{\sqrt{t}} \,.
\end{equation}

\item In case $\kappa = 0$, one has
\begin{equation}\label{ch2:eq:Vas00}
	\sigma_\imp(0,t) \sim \sqrt{2\pi} \,\, \frac{c(0,t)}{\sqrt{t}}\,.
\end{equation}
\end{itemize}
\end{theorem}

Relation \eqref{ch2:eq:Vas>00}, with explicit estimates for the error,
was proved in \cite{cf:GL}, extending \cite{cf:BF,cf:L,cf:RR,cf:G}.
Relation \eqref{ch2:eq:Vas<infty0} was proved 
in \cite{cf:CC} (see also \cite{cf:MT}).
We refer to \cite[Theorem 2.9]{cf:CC} for 
a self-contained proof of Theorem~\ref{ch2:th:main1}.

\begin{remark}\rm\label{ch2:rem:simple}
Whenever $\frac{-\log c(\kappa,t)}{\kappa} \to \infty$, 
formula \eqref{ch2:eq:Vas>00} simplifies to
\begin{equation}\label{ch2:eq:Vas>00simple}
	\sigma_\imp(\kappa,t) \sim \frac{\kappa}
	{\sqrt{2t (-\log  c(\kappa,t) )}} \,.
\end{equation}
Analogously, by \eqref{ch2:eq:Das},  formula \eqref{ch2:eq:Vas<infty0}
can be made more explicit as follows:
\begin{align}\label{ch2:eq:Vas<infty0simple}
	& \sigma_\imp(\kappa,t) \sim \begin{cases}
	\rule{0em}{1.6em}\displaystyle
	\frac{\kappa}
	{\sqrt{2t ( - \log ( c(\kappa,t) / \kappa ) )}} & 
	\text{ if \ 
		$\displaystyle\frac{c(\kappa,t)}{\kappa} \to 0$} \,; \\
	\rule[-1.4em]{0em}{3.5em}\displaystyle
	\frac{\kappa}
	{D^{-1}(a) \sqrt{t}}  & 
	\text{ if \ 
	$\displaystyle\frac{c(\kappa,t)}{\kappa} \to a \in (0,\infty)$} \,; \\
	\rule[-1.4em]{0em}{3.5em}\displaystyle
	\sqrt{2\pi} \, \frac{c(\kappa,t)}{\sqrt{t}}  & 
	\text{ if \ 
	$\displaystyle\frac{c(\kappa,t)}{\kappa} \to \infty$ \ or
	if \ $\kappa = 0$} \,.
	\end{cases}
\end{align}
\end{remark}

\subsection{Proof of Theorem~\ref{ch3:th:main}, part \eqref{iv:a}}

Consider a family of values of $(\kappa,t)$
such that either $t \to \bar t \in (0,\infty)$ and $\kappa \to \infty$, or
$t \to 0$ and $\kappa \gg \bkappa_2(\sfc^{1/D} \, t)$.
We consider two subregimes:
\begin{ienumerate}
\item \label{it:i}
either $t \to \bar t \in (0,\infty)$ and $\kappa \to \infty$, or
$t \to 0$ and $\kappa \to \bar\kappa \in (0,\infty]$;

\item \label{it:ii}
both $t \to 0$ and $\kappa \to 0$ with $\kappa \gg \bkappa_2(\sfc^{1/D} \, t)$.
\end{ienumerate}
Our goal is to prove that in both subregimes
relation \eqref{ch3:eq:smile} holds.

We start with subregime \eqref{it:i}.
By Theorems~\ref{ch3:lemma:calltgamma}
and~\ref{ch3:th:tailprob}, relations \eqref{ch3:eq:casinfty}
and \eqref{ch3:eq:Cramer2} give
\begin{equation} \label{eq:plugin}
	\log c(\kappa,t) \sim \log \P(X_t > \kappa)
	\sim 
	- \sfC \left( \frac{\kappa}{ \sfc \, t^D}\right)^{\frac{1}{1-D}}
	\left(\log \frac{\kappa}{ \sfc^{1/D} \, t}\right)^{\frac{1/2-D}{1-D}}  \,.
\end{equation}
Next we apply Theorem~\ref{ch2:th:main1}:
since $\liminf\kappa > 0$ in this subregime,
by Remark~\ref{ch2:rem:simple} relation \eqref{ch2:eq:Vas>00simple}
holds, because
$|\log c(\kappa,t)| \gg |\log \kappa|$ by \eqref{eq:plugin}. Then we get
\begin{equation} \label{ch3:eq:inan}
	\sigma_\imp(\kappa,t) \sim \frac{\kappa}{\sqrt{2t\, (-\log  c(\kappa,t) )}}
	\sim \sqrt{\frac{ \sfc^{1/D}}{2 \sfC}} 
	\left(\frac{\frac{\kappa}{ \sfc^{1/D} \, t}}
	{\sqrt{\log \frac{\kappa}{ \sfc^{1/D} \, t}}}\right)^{\frac{1/2-D}{1-D}}  \,,
\end{equation}
which is precisely our goal \eqref{ch3:eq:smile}.

Next we consider subregime \eqref{it:ii}.
Again by Theorems~\ref{ch3:lemma:calltgamma}
and~\ref{ch3:th:tailprob}, relations \eqref{ch3:eq:cas0good}
and \eqref{ch3:eq:Cramer2} show that $-\log (c(\kappa,t)/\kappa)$
is asymptotically equivalent to the right hand side of 
\eqref{eq:plugin}. By Theorem~\ref{ch2:th:main1} we can apply
relation \eqref{ch2:eq:Vas<infty0}, which
by Remark~\ref{ch2:rem:simple} reduces to the first
line of \eqref{ch2:eq:Vas<infty0simple}. In analogy with
\eqref{ch3:eq:inan}, we obtain again our goal \eqref{ch3:eq:smile}.\qed

\subsection{Proof of Theorem~\ref{ch3:th:main}, part \eqref{iv:b}}

Extracting a subsequence, we
may consider a family of values of $(\kappa,t)$ with $t\to 0$ and
$\kappa \sim a \, \bkappa_2( \sfc^{1/D} \, t)$ for some $a \in (0,\infty)$,
and our goal is to prove \eqref{ch3:eq:smile2}.
By Theorems~\ref{ch3:lemma:calltgamma}
and~\ref{ch3:th:tailprob},
relations \eqref{ch3:eq:cas0good} and \eqref{ch3:eq:withfa} yield
\begin{equation*}
	\log \big( c(\kappa,t)/\kappa \big) 
	\sim \log \P(X_t > \kappa)\sim 
	- \sff( \tilde a) \log \frac{1}{ \lambda \, t} \,,
\end{equation*}
where $\tilde a$ is defined in \eqref{ch3:eq:withfa}.
By Theorem~\ref{ch2:th:main1} and Remark~\ref{ch2:rem:simple},
recalling the definition \eqref{ch3:eq:gh} of $\bkappa_1(\cdot)$,
relation \eqref{ch2:eq:Vas<infty0simple} gives
\begin{equation*}
 \sigma_\imp(\kappa,t) \sim 
 \frac{\kappa}{\sqrt{2t ( - \log ( c(\kappa,t) / \kappa ) )}} \sim
 \frac{\sqrt{\lambda}}{\sqrt{2 \, \sff( \tilde a)}} \, \frac{\kappa }
 {\bkappa_1(\lambda \, t)}\, ,
\end{equation*}
which proves our goal \eqref{ch3:eq:smile2}.\qed

\subsection{Proof of Theorem~\ref{ch3:th:main}, part \eqref{iv:c}}

Next we consider a family of values of $(\kappa,t)$ with $t\to 0$ and
$\sqrt{2D+1} \, \bkappa_1(\sigma_0^2 \, t) \leq \kappa \ll \bkappa_2(\sigma_0^2 \, t)$,
and our goal is to prove \eqref{ch3:eq:smile1}.
Plugging relation \eqref{ch3:eq:cas0bad2} from Theorem~\ref{ch3:lemma:calltgamma}
into the first line of relation \eqref{ch2:eq:Vas<infty0simple}
(recall Theorem~\ref{ch2:th:main1} and Remark~\ref{ch2:rem:simple}),
by the definition \eqref{ch3:eq:gh} of $\bkappa_1(\cdot)$ we obtain
\begin{equation*}
 \sigma_\imp(\kappa,t) \sim \frac{\kappa}{\sqrt{2t\, ( - \log ( c(\kappa,t) / \kappa ) )}}
	\sim \frac{ \sqrt{\lambda}}{\sqrt{2
	\Big(1- \frac{\log
	\left( \kappa / \bkappa_2(\sfc^{1/D}\,t) \right)}{\log
	(\lambda\,  t)} \Big) }}  \,
  \frac{\kappa}{\bkappa_1( \lambda \, t)}  \,,
\end{equation*}
proving our goal \eqref{ch3:eq:smile1}.\qed

\subsection{Proof of Theorem~\ref{ch3:th:main}, part \eqref{iv:d}}

Next we consider a family of values of $(\kappa,t)$ with $t\to 0$ and
$0 \leq \kappa \leq \sqrt{2D+1} \, \bkappa_1( \sigma_0^2 \, t)$, and our goal
is to prove \eqref{ch3:eq:smile0}, i.e.\ $\sigma_\imp(\kappa,t) \sim \sigma_0$.

First we consider the case of typical deviations, i.e.\ when
$\kappa \sim a \sqrt{ \sigma_0^2 \, t}$ for some $a\in[0,\infty)$.
In case $a > 0$, relation \eqref{ch3:eq:typ} from Theorem~\ref{ch3:lemma:calltgamma} gives
\begin{equation*}
	\frac{c(\kappa, t)}{\kappa} \to
 	D\left(a\right)
	\sim D\left(\frac{\kappa}{\sqrt{ \sigma_0^2 \, t}}\right) \,,
\end{equation*}
which plugged into relation \eqref{ch2:eq:Vas<infty0}
from Theorem~\ref{ch2:th:main1} yields our goal
$\sigma_\imp(\kappa,t) \sim \sigma_0$.
In case $a = 0$, i.e.\ if $\kappa = o(\sqrt{t})$,
relation \eqref{ch3:eq:typlast} from Theorem~\ref{ch3:lemma:calltgamma} gives
$c(\kappa, t) \sim \frac{\sigma_0}{\sqrt{2\pi}} \, \sqrt{t}$,
hence $c(\kappa,t)/\kappa \to \infty$. We can thus apply relation
\eqref{ch2:eq:Vas<infty0} from Theorem~\ref{ch2:th:main1},
in the simplified form given by the third line of \eqref{ch2:eq:Vas<infty0simple}
(recall Remark~\ref{ch2:rem:simple}), getting
our goal $\sigma_\imp(\kappa,t) \sim \sigma_0$.\footnote{If $\kappa = 0$ 
one should apply relation \eqref{ch2:eq:Vas00}, rather than \eqref{ch2:eq:Vas<infty0},
from Theorem~\ref{ch2:th:main1}, which however coincides with the
the third line of \eqref{ch2:eq:Vas<infty0simple}, so the conclusion is the same.}

Next we consider the case of atypical deviations, i.e.\ when 
$\kappa \gg \sqrt{ \sigma_0^2 \, t}$.
By Theorems~\ref{ch3:lemma:calltgamma}
and~\ref{ch3:th:tailprob},
relations \eqref{ch3:eq:cas0good} and \eqref{ch3:eq:withoutfa2} yield
\begin{equation*}
\log \big( c(\kappa,t)/\kappa \big) 
	\sim - \log \P(X_t > \kappa)\sim 
	- \frac{\kappa^2}{2\sigma_0^2 t} \,.
\end{equation*}
By Theorem~\ref{ch2:th:main1} and Remark~\ref{ch2:rem:simple},
since $\kappa \to 0$, the first line of relation \eqref{ch2:eq:Vas<infty0simple} gives
\begin{equation*}
	\sigma_\imp(\kappa,t) \sim \frac{\kappa}{\sqrt{2t\, ( - \log ( c(\kappa,t) / \kappa ) )}}
	\sim \sigma_0 \,,
\end{equation*}
proving our goal \eqref{ch3:eq:smile0}.
The proof of Theorem~\ref{ch3:th:main} is completed.\qed

\bigskip

\appendix

\section{Miscellanea}
\label{ch3:sec:app}

\subsection{Proof of relation \eqref{ch3:eq:V}}
\label{sec:avvol}

We recall that $(N_t)_{t\ge 0}$ denotes a Poisson process of intensity $\lambda$,
with jump times $\tau_1, \tau_2, \ldots$, while $\tau_0 \in (-\infty,0)$ is a fixed
parameter. The random variable $\tau_{N_t}$
represents the last jump time prior to $t$.

It is well-known that the random variable $t - \tau_{N_t}$,
conditionally on the event $\{N_t \ge 1\}$, is distributed like an exponential random
variable $Y \sim Exp(\lambda)$ conditionally on $\{Y \le t\}$. As a consequence,
the following equality in distribution holds:
\begin{equation*}
	(t-\tau_{N_t}) \overset{d}{=} Y \, \ind_{\{Y \le t\}} + (t + |\tau_0|) \, \ind_{\{Y > t\}} \,.
\end{equation*}
It follows easily that as $t \to \infty$ the random variable $t - \tau_{N_t}$
converges to $Y$ in distribution.
Moreover, for every $\alpha \in (0, 1)$ we have
\begin{equation*}
\begin{split}
	\E\bigg[ \frac{1}{(t-\tau_{N_t})^{\alpha}} \bigg] & = 
	\E\bigg[ \frac{1}{Y^\alpha} \,\bigg|\, Y \le t \bigg] 
	(1-e^{-\lambda t}) + \frac{1}{(t + |\tau_0|)^{\alpha}} \, e^{-\lambda t} \\
	& \,\xrightarrow[t\to\infty]{}\, \E\bigg[ \frac{1}{Y^\alpha} \bigg]
	= \int_0^\infty \frac{1}{y^\alpha} \lambda \, e^{-\lambda y} \, \dd y
	= \lambda^\alpha \, \Gamma(1-\alpha) \,.
\end{split}
\end{equation*}
Choosing $\alpha = 1-2D$ and recalling \eqref{ch3:eq:sigmat}, we obtain
$\lim_{t\to\infty} \E[\sigma_t^2] = V^2$, proving \eqref{ch3:eq:V}.\qed

\subsection{Martingale measures}
\label{ch3:sec:martingale}

Let $(Y_t)_{t\ge 0}$ be the martingale in
\eqref{ch3:eq:Ystart}, i.e.\ $\dd Y_t = \sigma_t \, \dd B_t$,
which represents the detrended log-price under the historical measure.
We recall that $(\sigma_t)_{t\ge 0}$ is the process defined in \eqref{ch3:eq:sigmat},
where $\tau_0 \in (-\infty,0)$ is a parameter and $(\tau_k)_{k\ge 1}$ 
are the jumps of a Poisson process $(N_t)_{t\ge 0}$ of intensity $\lambda$,
independent of the Brownian motion $(B_t)_{t\ge 0}$.

For $\tilde \lambda \in (0,\infty)$ and $T \in (0,\infty)$, define the equivalent probability measure
$\tilde\P_{\tilde \lambda, T}$ by
\begin{equation} \label{ch3:eq:mart1}
	\frac{\dd \tilde\P_{\tilde \lambda, T}}{\dd \P} :=
	e^{ -\int_0^T \frac{\sigma_s}{2} \, \dd B_s
	- \frac{1}{2} \int_0^T (\frac{\sigma_s}{2})^2 \, \dd s }
	\cdot e^{(\log \frac{\tilde \lambda}{\lambda} ) N_T
	- (\tilde \lambda - \lambda) T} =: R_1 \cdot R_2 \,.
\end{equation}
Note that $R_2$ is the Radon-Nikodym derivative (on the time interval $[0,T]$)
of the law of a Poisson process of intensity $\tilde \lambda$ with respect to
that of intensity $\lambda$.
Denoting by $\cG$ the $\sigma$-algebra generated by
$(N_t)_{t \in [0,T]}$, the volatility $(\sigma_t)_{t \in [0,T]}$ is a $\cG$-measurable process.
Conditionally on $\cG$, the trajectories $t \mapsto \sigma_t$
are thus \emph{deterministic}, 
hence the random variable $\int_0^T \frac{\sigma_s}{2} \, \dd B_s$ is Gaussian
with zero mean and variance $\int_0^T (\frac{\sigma_s}{2})^2 \dd s < \infty$
(by \eqref{ch3:eq:sigmat}, since $D < \frac{1}{2}$). Recalling the definition
\eqref{ch3:eq:mart1} of $R_1$, it follows immediately that $\E[R_1 | \cG] = 1$.
 
The previous observations show that \eqref{ch3:eq:mart1}
defines indeed a probability $\tilde\P_{\tilde \lambda, T}$,  since
\begin{equation*}
	\E[R_1 R_2] = \E[\E[R_1 | \cG] R_2] = \E[R_2] = 1 \,,
\end{equation*}
and $(N_t)_{t\in [0,T]}$ under $\tilde\P_{\tilde \lambda, T}$ is a Poisson process with
intensity $\tilde \lambda$.
Moreover, the process
\begin{equation} \label{ch3:eq:tildeB}
	\tilde B_t := B_t + \int_0^t \frac{\sigma_s}{2} \, \dd s \,, \qquad \text{i.e.} \qquad
	\dd \tilde B_t := \dd B_t + \frac{\sigma_t}{2} \, \dd t \,,
\end{equation}
is a Brownian motion under the conditional law
$\tilde\P_{\tilde \lambda, T}(\,\cdot\,|\cG)$, by Girsanov's theorem.
The fact that the distribution of $(\tilde B_t)_{t \in [0,T]}$
conditionally on $\cG$ does not depend on $\cG$ (it is the Wiener measure),
means that $(\tilde B_t)_{t \in [0,T]}$ is independent of $\cG$,
i.e.\ of $(N_t)_{t\in [0,T]}$.

Summarizing: under $\tilde\P_{\tilde \lambda,T}$ the process 
$(\tilde B_t)_{t \in [0,T]}$ in~\eqref{ch3:eq:tildeB} is a Brownian motion
and $(N_t)_{t\in [0,T]}$ is an independent Poisson process of
intensity $\tilde \lambda$.
Rewriting~\eqref{ch3:eq:Ystart} as
\begin{equation*}
	\dd Y_t = \sigma_t \, \dd \tilde B_t - \frac{1}{2} \sigma_t^2 \, \dd t \,,
\end{equation*}
by Ito's formula the process $(S_t := e^{Y_t})_{t \in [0,T]}$ solves the stochastic differential equation
\begin{equation}\label{ch4:eq:Stnew'}
	\dd S_t = S_t \, \dd Y_t + \frac{1}{2} S_t \, \dd \langle Y \rangle_t 
	= S_t \, \dd Y_t + \frac{1}{2} S_t \, \sigma_t^2 \, \dd t  
	= \sigma_t \, S_t \, \dd \tilde B_t \,.
\end{equation}
We have thus shown that under $\tilde\P_{\tilde \lambda,T}$ the price
$(S_t)_{t \in [0,T]}$ evolves according to \eqref{ch3:sdeprice}
(where the Brownian motion $\tilde B_t$ has been renamed $B_t$),
with the process $(\sigma_t)_{t\in [0,T]}$ still defined by \eqref{ch3:eq:sigmat},
except that the Poisson process $(N_t)_{t\in [0,T]}$ has now intensity $\tilde \lambda$.

\subsection{A minimization problem}
\label{ch3:sec:minimum}

Let us recall from \eqref{ch3:eq:f} the definition of $\sff: (0,\infty) \to \R$:
\begin{equation}\label{ch3:eq:fapp}
	\sff(a) := \min_{m \in \N_0} \sff_m(a) \,, \qquad
	\text{with} \qquad
	\sff_m(a) := 
	m + \frac{a^2}{ 2\, m^{1-2D}}  \,.
\end{equation}
We also recall that, since $D < \frac{1}{2}$, we can restrict the
minimum to $m\in\N = \{1,2,3,\ldots\}$.

For fixed $a \in (0,\infty)$, if we minimize $\sff_m(a)$ over
$m\in (0,\infty)$, rather than over $m\in\N$, the global minimum is
attained at the unique $\tilde m_a \in (0,\infty)$ with
$\frac{\partial}{\partial m} \sff_m(a) |_{m=\tilde m_a} = 0$, i.e.
\begin{equation*}
	\tilde m_a = \bigg( \sqrt{\tfrac{1}{2}-D} \, 
	 a \bigg)^{\frac{1}{1-D}} \,.
\end{equation*}
Since $m \mapsto \sff_m(a)$ is decreasing on $(0,\tilde m_a)$
and increasing on $(\tilde m_a, \infty)$, it follows that
\begin{equation} \label{ch3:eq:sffint}
	\sff(a) = \min \big\{ \sff_{\lfloor \tilde m_a \rfloor}(a),
	\sff_{\lceil \tilde m_a \rceil}(a) \big\} \,,
\end{equation}
where $\lfloor x \rfloor := \max\{k \in \Z: k \le x\}$ and 
$\lceil x \rceil := \min\{k \in \Z: k \ge x\}$ denote the lower and
upper integer part of $x$, respectively.
In particular, if $\tilde m_a = k \in \N$ is an integer, i.e.\ if
\begin{equation*}
	a = \hat a_k := \frac{ 1}{\sqrt{\frac{1}{2}-D}} k^{1-D} \,,
\end{equation*}
then $\sff(a) = \sff_k(a)$. Next we observe that for $a \in (\hat a_k, \hat a_{k+1})$ one has
$\tilde m_a \in (k, k+1)$, hence $\sff(a) = \min\{\sff_k(a), \sff_{k+1}(a)\}$ by \eqref{ch3:eq:sffint}.
By direct computation, one has
\begin{equation*}
	\sff_k(a) \le \sff_{k+1}(a) \qquad \iff \qquad
	a \le x_k := \frac{ 1}{\sqrt{\frac{1}{2k^{1-2D}}
	- \frac{1}{2(k+1)^{1-2D}}}}  \,.
\end{equation*}
(Note that $\hat a_k < x_k < \hat a_{k+1}$, by
convexity of $z \mapsto z^{-(1-2D)}$, and
$x_k \sim \hat a_k$ as $k \to \infty$.)
Setting $x_0 := 0$ for convenience,
the previous considerations show that
\begin{equation} \label{ch3:eq:sffex}
	\sff(a) = \sff_k(a) \qquad \text{for all} \ a \in [x_{k-1}, x_k)
	\text{ and } k\in\N \,.
\end{equation}
Since $\sff_k(x_k) = \sff_{k+1}(x_k)$ by construction, the function $\sff$ 
is continuous and strictly increasing (but it is \emph{not} convex, as one can check).
The asymptotics in \eqref{ch3:eq:asf} follow easily by \eqref{ch3:eq:sffex}
and \eqref{ch3:eq:fapp}, which yield $\sff(a) \sim \sff_{\tilde m_a}(a)$ as $a \to \infty$.

\section*{Acknowledgments}

We thank Fabio Bellini, Paolo Dai Pra and Carlo Sgarra for fruitful discussions.


\end{document}